\documentclass[10pt]{amsart}

\usepackage{amsfonts}
\usepackage{amssymb,amsmath,tabularx,graphicx,float,enumitem}
\usepackage{graphicx}
\usepackage{bm}
\usepackage{marginnote}
\usepackage{color}
\usepackage{stmaryrd}
\usepackage{multirow}
\usepackage{todonotes}

\usepackage{appendix}
\definecolor{dkblue}{RGB}{1,31,91} % This is a dark Blue  

\usepackage[foot]{amsaddr}

\usepackage[colorlinks=true, pdfstartview=FitV, linkcolor=dkblue, citecolor=dkblue, urlcolor=dkblue]{hyperref}    
\usepackage{hyperref}

\newtheorem{theorem}{Theorem}[section]
\newtheorem{lemma}[theorem]{Lemma}
\newtheorem{proposition}[theorem]{Proposition}
\newtheorem{corollary}[theorem]{Corollary}

\newtheorem{definition}[theorem]{Definition}

\newtheorem{remark}[theorem]{Remark}

\newcommand{\pd}[1]{\p_{#1}}
\newcommand{\pdd}[2]{\p_{#1}^{#2}}
\newcommand{\abs}[1]{\left\lvert{#1}\right\rvert}
\newcommand{\norm}[1]{\left\lVert{#1}\right\rVert}
\newcommand{\snorm}[1]{\left[{#1}\right]}
\newcommand{\hnorm}[3]{\norm{#1}_{C^{#2, #3}}}
\newcommand{\chnorm}[2]{\norm{#1}_{C^{#2}}}

\newcommand{\hsnorm}[2]{\snorm{#1}_{C^{0, #2}}}

\newcommand{\starnorm}[1]{\abs{#1}_{*}}

\newcommand{\mc}[1]{\mathcal{#1}}
\newcommand{\mb}[1]{\mathbb{#1}}
\newcommand{\p}{\partial}
\newcommand{\dist}{\mbox{dist}}
\newcommand{\sgn}{\mbox{sgn}}
\newcommand{\veps}{\varepsilon}

\newcommand{\paren}[1]{\left({#1}\right)}

\newcommand{\diff}{\triangle}

\newcommand{\wt}[1]{\widetilde{#1}}
\newcommand{\wh}[1]{\widehat{#1}}
\newcommand{\mbs}{{\mathbb{S}^1}}

\newcommand{\mbr}{\mathbb{R}}
\newcommand{\jump}[1]{\llbracket{#1}\rrbracket}
\newcommand{\dual}[2]{\left\langle{#1},{#2}\right\rangle}
\newcommand{\set}[2]{\left\{ #1 \mid #2 \right\}}

\begin{document}

\title{The Tension Determination Problem for an Inextensible Interface in 2D Stokes Flow}
% \author{Po-Chun Kuo, Ming-Chih Lai, Yoichiro Mori, Analise Rodenberg}

\author[P.-C. Kuo]{Po-Chun Kuo$^{\ast,\dagger}$}
\thanks{$^\dagger$ supported by NSF DMS-204144(USA)}
% \email[pck]{kuopo@sas.upenn.edu}
\author[M.-C. Lai]{Ming-Chih Lai$^\mathsection$}
\thanks{$^\mathsection$ supported by NSTC, Taiwan, under the research grant 110-2115-M-A49-011-MY3(TW)}
% \email[mcl]{mclai@math.nctu.edu.tw}
\author[Y. Mori]{Yoichiro Mori$^{\ast,\ddagger}$}
\thanks{$^\ddagger$ supported by the Simons Foundation Math+X Chair Fund and NSF DMS-1907583 and DMS-2042144(USA)}
% \email[ym]{y1mori@sas.upenn.edu}
\author{Analise Rodenberg$^\mathparagraph$}
% \email[A. Rodenberg]{arodenb2@kennesaw.edu}

\address{   $^\ast$Department of Mathematics, School of Arts and Sciences, University of Pennsylvania, Philadelphia, PA 19104, United States of America.
            $^\dagger$\href{mailto:kuopo@sas.upenn.edu}{kuopo@sas.upenn.edu}
            $^\ddagger$\href{y1mori@sas.upenn.edu}{y1mori@sas.upenn.edu}}
\address{$^\mathsection$Department of Applied Mathematics, National Yang Ming Chiao Tung University, Hsinchu, Taiwan 30010, ROC
\href{mclai@math.nctu.edu.tw}{mclai@math.nctu.edu.tw}}
\address{$\mathparagraph$ Department of Mathematics, Kennesaw State University, Marietta, GA 30060, United States of America
\href{arodenb2@kennesaw.edu}{arodenb2@kennesaw.edu}}
% \address{$^\dagger$\href{mailto:kuopo@sas.upenn.edu}{kuopo@sas.upenn.edu}}
% \address[M.-C. Lai]{Department of Applied Mathematics, National Yang Ming Chiao Tung University, Hsinchu, Taiwan 30010, ROC}
% \address[A. Rodenberg]{Department of Mathematics, Kennesaw State University, Marietta, GA 30060, United States of America}

% \thanks{Y. Mori was supported by the Simons Foundation Math+X Chair Fund and NSF DMS-1907583 and DMS-2042144. 
% P.-C. Kuo was supported by NSF DMS-204144.
% M.-C. Lai was supported by NSTC, Taiwan, under the research grant 110-2115-M-A49-011-MY3.}

\date{\today}
\maketitle

\begin{abstract}
Consider an inextensible closed filament immersed in a 2D Stokes fluid. Given a force density $\bm{F}$ defined on this filament, 
we consider the problem of determining the tension $\sigma$ on this filament that ensures the filament is inextensible.
This is a subproblem of dynamic inextensible vesicle and membrane problems, which appear in 
engineering and biological applications. We study the well-posedness and regularity properties of this problem 
in H\"older spaces. We find that the tension determination problem admits a unique solution if and only if the closed 
filament is {\em not} a circle. Furthermore, we show that the tension $\sigma$ gains one derivative with respect to the 
imposed line force density $\bm{F}$, and show that the tangential and normal components of $\bm{F}$ affect the regularity 
of $\sigma$ in different ways. We also study the near singularity of the tension determination problem as the interface 
approaches a circle, and verify our analytical results against numerical experiment.
%    We consider a static problem of a one dimensional inextensible filament immersed in a two dimensional steady Stokes fluid.
%    This problem is a subproblem of a dynmaic inextensible vesicle problem, applied on engineering, biology and other fields.
%    The stationary problem is solving velocity pressure and a Lagrange multiplier for its inextensibilty constraint with a fixed interface and a given external force.
%    We establish the well-posedness of the static problem in Holder spaces, which is helpful to understand well-posedness of the dynamic problem.
%    With boundary integral method, we first reduce the stationary problem into an equation only for solving the Lagrange multiplier.
%    Then, we split the operator on the multiplier into a principle part, which is a Hilbert transform on its first derivative function, and a remaining part.
%    Finally, the well-posedness is established by Fredholm Alternative theorem after taking an left inverse of the principle part adding an identity on the equation.
%    Furthermore, in this analysis, we find a singularity when the interface is a unit circle, so we discuss the singular phenomena when the interface is near a unit circle and use simulation to confirm our results.  
\end{abstract}

% \tableofcontents
\section{Introduction}

\subsection{Motivation and Model Formulation}\label{motivation}
Fluid structure interaction problems in which thin elastic structures interact with the surrounding fluid find many applications throughout the natural sciences
and engineering \cite{FML1997, KS2005, KS2006, OL2020, SSOG1986, VGZB2009}.
One of the simplest of such problems is the 2D Peskin problem, in which a 1D closed elastic structure is immersed in a 2D Stokes fluid. There have been extensive 
computational studies of this and related problems \cite{B2015, LL2001,P1997,P2002,SL2010}. 
More recently, the 2D Peskin problem has been studied analytically in \cite{CS2021,GMS2020,MRS2019,R2018,T2021}. 
In an important variant of this problem, the elastic structure is assumed to be inextensible, 
motivated in particular by the properties of lipid bilayer membranes.
This and related problems have been studied computationally by many authors as models for red blood cells and artificial membrane vesicles
\cite{OL2020,QGY2021,VRBP2009,VGZB2009, NG2005, S2004}.
A distinguishing feature of such inextensible interface problems is that the unknown tension $\sigma$ must be found as part 
of the problem. The tension $\sigma$ plays a role analogous to the pressure in incompressible flow problems.
In this paper, we consider the static problem of determining the tension $\sigma$ of a 1D inextensible interface immersed in a 2D Stokes fluid given 
a prescribed interfacial force density $\bm{F}$. 
Before we state our problem, let us first consider the following dynamic problem. 
Let $\Gamma_t$ denote a sufficiently smooth simple curve that depends on time $t$ which partitions $\mbr^2$ into the interior region $\Omega_1$ and its complement $\Omega_2=\mathbb{R}^2\backslash \lbrace\Omega_1\cup \Gamma_t \rbrace$.
The velocity field $\bm{u}$ and $p$ satisfy the Stokes equations in $\mathbb{R}^2\backslash\Gamma_t$
\begin{align}
\label{et:stokes}-\Delta \bm{u} + \nabla p &= 0 \quad \text{in } \mbr^2 \setminus \Gamma_t\\
\label{et:div_free}\nabla \cdot \bm{u} &= 0\quad \text{in } \mbr^2 \setminus \Gamma_t
\end{align}
We have assumed that the viscosity of the interior and exterior fluids are the same and normalized to $1$.
We let $\Gamma_t$ be inextensible. Parametrize $\Gamma_t$ as $\bm{X}(s,t)$ where $s$ is both an arclength and Lagrangian parametrization of the curve.
For definiteness, we assume that the parametrization is in the counter-clockwise direction along the curve $\Gamma_t$.
Since $s$ is the arclength parameter, we have
\begin{equation}
\label{et:inextX}
\abs{\bm{\tau}}=1, \; \bm{\tau}=\p_s\bm{X},
\end{equation}
where $\p_s$ is the partial derivative with respect to $s$ and $\bm{\tau}$ is the unit tangent vector on $\Gamma_t$.
We assume, without loss of generality, that the length of the string is $2\pi$ so that $s\in \mathbb{S}^1=\mathbb{R}/(2\pi\mathbb{Z})$. 
We impose the following interface conditions on $\Gamma_t$ to the Stokes equations \eqref{et:stokes} and \eqref{et:div_free}
\begin{align}
\label{et:velocity_cts}\jump{\bm{u}} &= 0 \quad \text{on } \Gamma_t\\
\label{et:stress_jump}\jump{\paren{\nabla\bm{u} + (\nabla\bm{u})^T - p\mb{I}}\bm{n}} &= \mc{F}(\bm{X}) + \p_{s}(\sigma\bm{\tau}), \quad \text{on } \Gamma_t,
\end{align}
where $\mathbb{I}$ is the $2\times 2$ identity matrix and $\bm{n}$ is the unit normal vector on $\Gamma_t$ pointing outward (from $\Omega_1$ to $\Omega_2$)
\begin{equation*}
\bm{n}=\p_s\bm{X}^\perp=\mc{R}_{\pi/2}\bm{\tau}=\mc{R}_{\pi/2}\p_s\bm{X}, \quad \mc{R}_{\pi/2}=\begin{pmatrix} 0 & 1 \\ -1 & 0 \end{pmatrix} .
\end{equation*}
In the above, $\jump{\cdot}$ denotes the jump in the enclosed value across $\Gamma_t$
\begin{align*}
\jump{f} := f|_{\Omega_1} - f|_{\Omega_2}.
\end{align*}
Thus, equation \eqref{et:velocity_cts} enforces continuity of the velocity field and \eqref{et:stress_jump} specifies the jump in stress across the interface $\Gamma_t$.
Note that the interfacial force given in the right hand side of \eqref{et:stress_jump} consists of two terms.
The first term $\mc{F}$ is a mechanical force determined by the configuration of $\bm{X}$. A typical choice is to let the string generate a bending force
\begin{equation}\label{et:F}
\mc{F}=-\p_s^4 \bm{X}.
\end{equation}
The second term in the right hand side of \eqref{et:stress_jump} is a tension force that ensures where the tension $\sigma(s,t)$ is to be determined as part of the problem to enforce the inextensiblility constraint. The string position moves with the local fluid velocity
\begin{equation}
\label{et:Xt}
\p_t \bm{X}(s,t)=\bm{u}(\bm{X}(s,t),t),
\end{equation}
where $\p_t$ is the partial derivative with respect to $t$, so the inextensiblility constraint can be written as
\begin{equation}\label{et:inextmod}
\pd{t}\abs{\pd{s}\bm{X}}^2=2\pd{s}\bm{X}\cdot \pd{t}\pd{s}\bm{X}=2\bm{\tau}\cdot \pd{s}\bm{u}=0.
\end{equation}
The above condition is equivalent to \eqref{et:inextX} assuming that the initial parametrization is with respect to arclength.
To specify the problem completely, we finally impose the condition that
$\bm{u} \to 0$ and that $p$ be bounded as $|\bm{x}|\to \infty$.

The above dynamic problem has been considered, from modeling and computational points of view, 
by different authors primarily as a 2D mechanical model for red blood cells in flow \cite{CL2014,HKL2014,P_2005,SSOG1986}.
We also point out that the problem of finding the steady states of the above dynamic problem, taking $\mc{F}$ as in \eqref{et:F}, 
reduces to the problem of finding the minimizers of the Willmore energy under a perimeter and interior area constraint. This constrained minimization problem and its 3D counterpart have been studied by many authors \cite{choksi2013global,seifert1997configurations,VRBP2009}.

In this paper, we consider the following static problem of determining the tension $\sigma(s)$ given a force density $\bm{F}(s)$ defined on the interface (Figure \ref{fig:TDP}).
This may be considered as a subproblem of the above dynamic problem.
Let $\Gamma$ be a fixed simple curve, parameterized by arclength as $\bm{X}(s), \; s\in \mathbb{S}^1=\mathbb{R}/2\pi\mathbb{Z}$
as above. The Stokes equations are satisfied in $\mathbb{R}^2\backslash \Gamma$ as in \eqref{et:stokes}-\eqref{et:div_free}
\begin{align}
\label{e:stokes}-\Delta \bm{u} + \nabla p &= 0 \quad \text{in } \mbr^2 \setminus \Gamma\\
\label{e:div_free}\nabla \cdot \bm{u} &= 0\quad \text{in } \mbr^2 \setminus \Gamma
\end{align}
Given an interfacial force density $\bm{F}(s)$, we impose the following interfacial conditions as in \eqref{et:velocity_cts}-\eqref{et:stress_jump}
\begin{align}
\label{e:velocity_cts}\jump{\bm{u}} &= 0 \quad \text{on } \Gamma\\
\label{e:stress_jump}\jump{\paren{\nabla\bm{u} + (\nabla\bm{u})^T - p\mb{I}}\bm{n}} &=\bm{F} + \p_{s}(\sigma\bm{\tau}), \quad \text{on } \Gamma,
\end{align}
We then have the inextensibility condition as in \eqref{et:inextmod}, which allows for the determination of $\sigma$
\begin{align}
\label{e:inextensible}\p_{s}\paren{\bm{u}(\bm{X}(s))}\cdot \bm{\tau} &= 0 \quad \text{on } \Gamma.
\end{align}
%Equation \eqref{e:inextensible} demands that the interface $\Gamma$ remain inextensible. 
%Equations \eqref{e:stokes}-\eqref{e:inextensible} comprise the system that we study in this paper.
%
%The above static problem should be seen as a subproblem of the following dynamic inextensible string problem, in which the interface, now denoted $\Gamma_t$ 
%to emphasize its dependence on time $t$, moves
%with the the local fluid velocity. The interfacial curve $\Gamma_t$ is parametrized as $\bm{X}(s,t)$. The Stokes equations \eqref{e:stokes} and \eqref{e:div_free}
%remain the same as in the dynamic problem except that the interior and exterior fluid regions are now time dependent.
%The 
%
%
%That is to say, if $\Gamma$ is a free interface with $\pd{t}\bm{X}(s,t)=\bm{u}(\bm{X}(s,t),t)$, then 
%\begin{align*}
%    \pd{t}\abs{\pd{s}\bm{X}}^2=2\pd{s}\bm{X}\cdot \pd{t}\pd{s}\bm{X}=2\bm{\tau}\cdot \pd{s}\bm{u}=0.
%\end{align*} 
We again impose the condition that $\bm{u} \to 0$ and that $p$ be bounded as $|\bm{x}|\to \infty$. 
In order for $\bm{u}\to 0$ as $\abs{\bm{x}}\to \infty$, we must impose the condition
\begin{equation*}
\int_{\Gamma}\bm{F}ds = 0.
\end{equation*}
Our problem is thus to solve for the unknown tension $\sigma$, together with $\bm{u}$ and $p$, given a force density $\bm{F}$ that satisfies
the mean zero condition given above.  
\begin{figure}[htbp]
    \centering
    \includegraphics[width=1\textwidth]{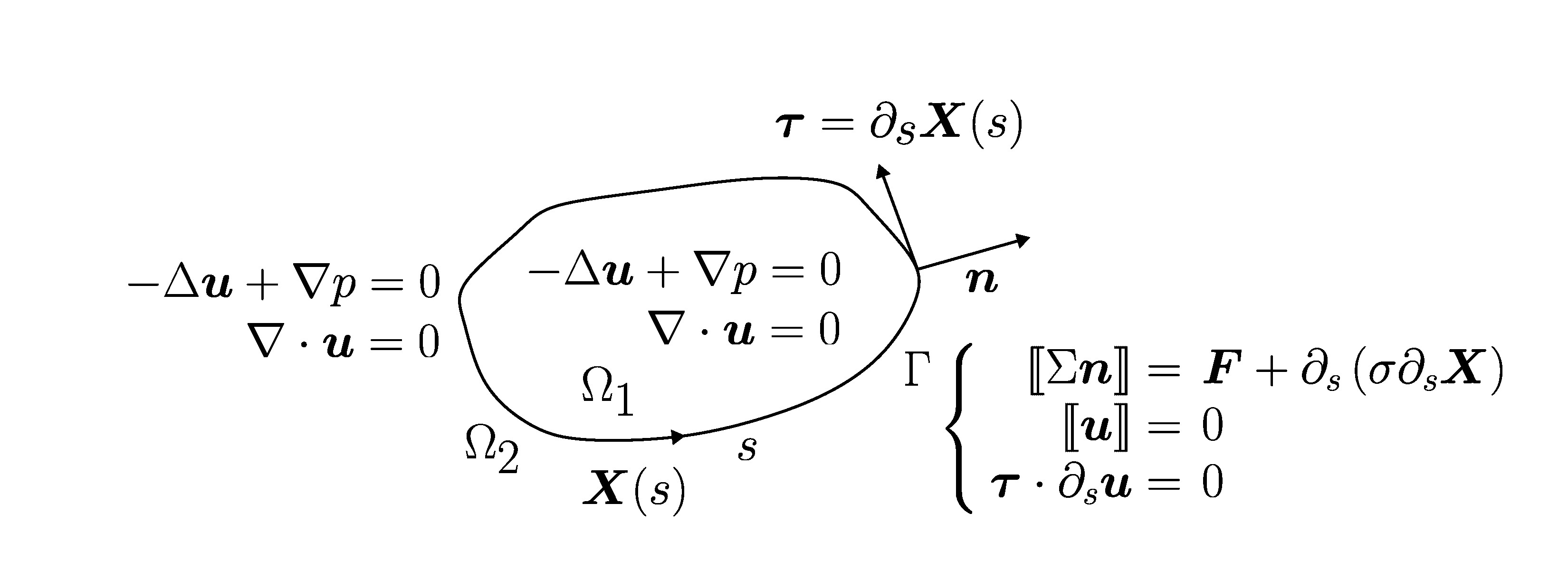}
    \caption{Schematic diagram for the tension determination problem.}
    \label{fig:TDP}
\end{figure}

In addition to its intrinsic interest, an understanding of the above static problem should pave the way toward an analysis of the dynamic problem.
Furthermore, an analysis of the above problem should give insight into the numerical algorithms for this problem. Indeed, 
all numerical algorithms to date for dynamic inextensible interface problems solves for $\sigma$ at each time step \cite{OL2020,VGZB2009,LSX2019}. 

The paper that is directly relevant to our paper is \cite{LSX2019}, where the authors provide an analysis of the above tension determination problem motivated 
by the need to develop numerical algorithms for the Navier-Stokes version of the dynamic problem.
There, the authors consider the problem in which a term $\alpha \bm{u}, \alpha>0$ is added to \eqref{e:stokes}, and $\Omega$ is bounded domain.
The authors define a notion of weak solution by formulating the problem as a saddle point problem, and prove an inf-sup inequality to establish existence and uniqueness
in an $L^2$ based Sobolev space. We shall comment on the relationship between this and our results where appropriate.

\subsection{Well-posedness}

Let $C^k(A),\; k=0,1,2, \cdots$ be the space of functions with continuous $k$-derivatives on the set 
$A$, where $A=\mbs, \mathbb{R}^2,$ or $\mathbb{R}^2\backslash \Gamma$.
We shall mostly work with $C^k(\mbs)$.
Define the norms on $C^k\paren{\mbs}$ as
\begin{align*}
    \chnorm{f}{k}=\sum_{i=0}^k \snorm{f}_{C^k},\quad \snorm{f}_{C^k}=\sup_{s\in \mbs}\abs{\pdd{s}{k}f}
\end{align*}
Next, a function $f$ in $C^0\paren{\mbs}$ is in the H\"older space $C^{0,\gamma}\paren{\mbs}$, $0<\gamma<1$ if $f$ satisfies
\begin{align*}
    \sup_{s, s'\in \mbs}\frac{\abs{f(s) - f(s')}}{\abs{s - s'}^\gamma}<\infty.
\end{align*}
For the definition of the norm of $C^{0,\gamma}\paren{\mbs}$, we may restrict the range of $s$ and $s'$.
For example, set the range as $\abs{s-s'}<1$.
Then, define the norm as 
\begin{align*}
    \hnorm{f}{0}{\gamma}:=  \chnorm{f}{0}+\hsnorm{f}{\gamma},\quad \hsnorm{f}{\gamma}=\sup_{\abs{s-s'}<1}\frac{\abs{f(s) - f(s')}}{\abs{s - s'}^\gamma}
\end{align*}
since
\begin{align*}
    \sup_{s, s'\in \mbs}\frac{\abs{f(s) - f(s')}}{\abs{s - s'}^\gamma}\leq 2\chnorm{f}{0}+\hsnorm{f}{\gamma}.
\end{align*}
Next, we define the H\"older space $C^{k,\gamma}\paren{\mbs}$. The function
$f$ is in $C^{k,\gamma}\paren{\mbs}$ if $f\in C^{k}\paren{\mbs}$ and $\pdd{s}{k}f$ is in $C^{0,\gamma}\paren{\mbs}$, where the the norm is defined as
\begin{align*}
    \hnorm{f}{k}{\gamma}:=  \chnorm{f}{k}+\hsnorm{\pdd{s}{k}f}{\gamma}.
\end{align*}
We will frequently write $f = f(s)$ and $f' = f(s')$, and use the notation
\begin{align*}
\Delta f := f(s) - f(s'), \quad 
\diff_h f = f(s + h) - f(s).
\end{align*}
To estimate expressions which feature denominators of the form $\abs{\Delta\bm{X}}$, we will need the following quantity
\begin{align*}
\starnorm{\bm{X}} := \inf_{s\neq s'}\abs{\frac{\bm{X}(s) - \bm{X}(s')}{s - s'}}.
\end{align*}
This condition allows us to estimate $\Delta \bm{X}$ from below
\begin{equation*}
\abs{\Delta \bm{X}}\geq \starnorm{\bm{X}}\abs{s-s'}.
\end{equation*}
It is easily seen that $\bm{X}(s)$ is a simple curve if and only if $\starnorm{\bm{X}}>0$. Indeed, 
$\starnorm{\bm{X}}=0$ if and only if $\abs{\pd{s}\bm{X}}=0$ at some $s$ or if there are points $s\neq s'$ such that $\bm{X}(s)\neq \bm{X}(s')$.
Given our arclength parametrization, $\abs{\p_s\bm{X}}=1\neq 0$; thus, $\starnorm{\bm{X}}>0$ is equivalent to the condition that $\bm{X}$ has no self intersections.
We shall also make use of the Lebesgue spaces $L^p(\mbs), 1\leq p\leq \infty$.

Before we state our result, we give a precise definition of what we mean by a solution to the tension determination problem.
Let the stress tensor be
\begin{equation}\label{sigF}
\Sigma(\bm{x})=\nabla \bm{u}(\bm{x})+(\nabla\bm{u}(\bm{x}))^{\rm T}-p(\bm{x})\mathbb{I},
\end{equation}
and define limits
\begin{equation}\label{Siglim}
\bm{F}_{\Omega_1}(s)=\lim_{t\to 0+} \Sigma(\bm{X}(s)-t\bm{n}(s))\bm{n}(s), \; \bm{F}_{\Omega_2}(s)=\lim_{t\to 0+} \Sigma(\bm{X}(s)+t\bm{n}(s))\bm{n}(s),
\end{equation}
where $\bm{n}(s)=\bm{X}(s)^\perp$ is the outward normal on $\Gamma$.
\begin{definition}[Solution of Tension Determination Problem]\label{STDP}
Assume $\bm{F}\in C^0(\mbs)$, $\bm{X}\in C^2(\mbs)$ and $\starnorm{\bm{X}}>0$. Let $\bm{u}, p, \sigma$ belong to the following function spaces
\begin{equation}\label{upfnspace}
\bm{u}\in C^2(\mathbb{R}^2\backslash \Gamma)\cap C^0(\mathbb{R}^2),\quad p\in C^1(\mathbb{R}^2\backslash\Gamma)\cap L^1_{\rm loc}(\mathbb{R}^2), \quad \sigma\in C^1(\mbs),
\end{equation}
where $L^1_{\rm loc}(\mathbb{R}^2)$ denotes the space of locally integrable functions in $\mathbb{R}^2$.
Suppose $\bm{u}$ and $p$ satisfy the following conditions in the far field
\begin{equation}\label{updecay}
\lim_{R\to \infty}\sup_{\abs{\bm{x}}=R}\abs{\bm{u}(\bm{x})}=0, \quad \lim_{R\to \infty}\sup_{\abs{\bm{x}}\geq R} \abs{p(\bm{x})}<\infty.
\end{equation}
We say that $\bm{u}, p, \sigma$ are a solution to the tension determination problem if the following conditions hold. 
\begin{enumerate}
\item $\bm{u}$ and $p$ satisfy the Stokes equations \eqref{e:stokes} and \eqref{e:div_free} in $\mathbb{R}^2\backslash \Gamma$.
\item $\bm{u}, p, \sigma $satisfy the condition \eqref{e:stress_jump} in the following sense. The limits in \eqref{sigF} exist, this convergence is uniform, 
and $\bm{F}_{\Omega_1}-\bm{F}_{\Omega_2}=\bm{F}+\p_s(\sigma \bm{\tau})$.\label{item:stressjump}
\item The inextensibility condition \eqref{e:inextensible} is satisfied in the following weak sense. For any $w\in C^1(\mbs)$, we have
\begin{equation*}
\int_\mbs \bm{u}(\bm{X}(s))\cdot\p_s (w\bm{\tau})ds=0.
\end{equation*}
\end{enumerate}
\end{definition}
The above represents the weakest possible condition on $\bm{F}, \bm{X}$ and $\sigma$ if we are to make pointwise sense of the interface condition \eqref{e:stress_jump}.
It turns out that, when $\bm{X}$ is merely $C^2(\mbs)$, the inextensibility condition \eqref{e:inextensible} cannot be satisfied pointwise. 
We hence impose this condition in a weak sense.
We now state our main well-posedness result. In the statement below, we let $\chi_{\Omega_1}(\bm{x})$ be the indicator function for $\Omega_1$
\begin{equation}\label{indOmega1}
\chi_{\Omega_1}(\bm{x})=\begin{cases} 1 &\text{ if } \bm{x}\in \Omega_1,\\ 0 &\text{ otherwise}.\end{cases}
\end{equation}
\begin{theorem}\label{t: well-posed_noncircle}
Suppose $\bm{X}\in C^{2}(\mbs)$ with $\starnorm{\bm{X}} > 0$, and
\begin{equation*}
\bm{F}\cdot \p_s\bm{X}\in C^{0,\gamma}(\mbs),\; \gamma\in(0,1),\quad \bm{F}\cdot \p_s\bm{X}^\perp\in C^{0}(\mbs), \quad \int_{\mbs} \bm{F} ds=0.
\end{equation*}
Then there exists a solution $\bm{u},p,\sigma$ to the tension determination problem in the sense of Definition \ref{STDP} with the following properties.
\begin{enumerate}
\item If $\bm{X}$ is not a circle, the general solution can be written as $\bm{u}, \sigma, p+c$ where $c\in \mathbb{R}$ is an arbitrary constant.
\item If $\bm{X}$ is a circle, the general solution can be written as $\bm{u}, \sigma+c_1, p+c_1 \chi_{\Omega_1}+c_2$ where $\chi_{\Omega_1}$ is as in \eqref{indOmega1} and $c_1,c_2\in \mathbb{R}$ are arbitrary constants.
\end{enumerate}
Furthermore, $\sigma\in C^{1,\gamma}(\mbs)$ and 
$\p_s \bm{u}(\bm{X}(s))\in L^p(\mbs), 1<p<\infty$, so that \eqref{e:inextensible} is satisfied for almost every $s\in \mbs$.
If, in addition $
\bm{F}\in C^{0,\gamma}(\mbs)$ and $\bm{X}\in C^{2,\gamma}(\mbs)$, then $\p_s\bm{u}(\bm{X}(s))\in C^\gamma(\mbs)$ and \eqref{e:inextensible} is satisfied pointwise.
\end{theorem}
We thus see that the tension determination problem has a unique solution $\sigma$ if and only if $\bm{X}$ is not a circle.
This suggests that, as $\Gamma$ approaches a circle, the problem of uniquely determining the tension $\sigma$ becomes increasingly singular.
We shall further investigate this solution behavior in Section \ref{property_L}.

\begin{remark}
The results of \cite{LSX2019} imply that, if $\bm{F}\in H^{-1/2}(\Gamma)$, then there is a suitable weak solution $\sigma\in L^2(\Gamma)$. 
Theorem \ref{t: well-posed_noncircle} shows that $\sigma$ is one derivative smoother than $\bm{F}$ in the H\"older scale. 
It is thus likely that even in the $L^2$ Sobolev scale, $\sigma$ gains one more derivative compared to $\bm{F}$. 

We also note that \cite{LSX2019} claims that the tension determination problem always has a unique (weak) solution, regardless of whether 
$\Gamma$ is a circle. This seems to be due to the fact that \cite{LSX2019} considers the weak solution corresponding to the following problem in which 
\eqref{e:stress_jump} is replaced by the following condition
\begin{align*}
    \jump{\paren{\nabla\bm{u} + (\nabla\bm{u})^T - p \mb{I}}\bm{n}}= \bm{F} + \p_{s}(\sigma\bm{\tau})+c_*\bm{n}
\end{align*}
for some constant $c_*$, and the following mean-zero constraint is imposed on $\sigma$
\begin{equation*}
\int_\Gamma \sigma ds=0.
\end{equation*}
It is straightforward to see from our results that $\sigma$, in the sense above of \cite{LSX2019}, is always uniquely determined regardless of whether $\Gamma$
is a circle. 
%It is however, not clear if the results of \cite{LSX2019} easily imply the dichotomy we have uncovered in Theorem \ref{t: well-posed_noncircle}.
%Indeed, \cite{LSX2019} does not address the question of the determination of $c_*$ in the above. 
\end{remark}

To prove the above well-posedness result, we first rewrite the problem in terms of a boundary integral equation for the unknown tension $\sigma$.
Let $\wh{\bm{F}}=\bm{F}+\p_s(\sigma\bm{\tau})$ denote the right hand side of \eqref{e:stress_jump}. It is well-known that 
the velocity field $\bm{u}$ and pressure $p$ can be expressed as 
\begin{align}
\bm{u}(\bm{x})  &= \wh{\mc{S}}[\wh{\bm{F}}](\bm{x}):= \int_{\mbs} G(\bm{x} - \bm{X}(s')) \wh{\bm{F}}(s')ds',\label{e:u_soln}\\
p(\bm{x})       &= \mc{P}[\wh{\bm{F}}](\bm{x}):=\int_{\mbs} \Pi(\bm{x} - \bm{X}(s')) \wh{\bm{F}}(s')ds',\label{e:p_soln}
\end{align}
where
\begin{equation*}
\begin{split}
G(\bm{r}) &= \frac{1}{4\pi}\paren{G_L(\bm{r})\mb{I} + G_T (\bm{r})}, \; G_L(\bm{r})=-\log \abs{\bm{r}}, \; G_T(\bm{r})=\frac{\bm{r}\otimes \bm{r}}{\abs{\bm{r}}^2},\\
\Pi(\bm{r}) &= \frac{1}{2\pi}\frac{\bm{r}^{T}}{\abs{\bm{r}}^2}.
\end{split}
\end{equation*}
Moreover, the stress tensor $\Sigma$ will be
\begin{align}\label{e:tensor_soln}
    \Sigma_{ij} \paren{\bm{x}}=\mc{T}[\wh{\bm{F}}](\bm{x}):=\int_{\mbs} \Theta_{ijk}(\bm{x} - \bm{X}(s')) \wh{F}_k(s')ds',
\end{align}
where
\begin{align*}
    \Theta_{ijk}(\bm{r})=-\frac{1}{\pi}\frac{r_i r_j r_k}{\abs{\bm{r}}^4}
\end{align*}
The relevant properties of the above potentials will be discussed in Section \ref{sect:classical}. 
Take the limit as $\bm{x}\to \bm{X}(s)$ in \eqref{e:u_soln} to obtain
\begin{align}\label{uXs}
\bm{u}(\bm{X}(s)) &= \wh{S}[\wh{\bm{F}}](\bm{X}(s))=\mc{S}[\wh{\bm{F}}](s) = \int_{\mbs} G(\bm{X}(s)-\bm{X}(s')) \wh{\bm{F}}(s')ds'.
\end{align}
Now, we can rewrite equation \eqref{e:inextensible} as
\begin{align*}
\p_{s}\bm{X}\cdot \p_{s}\mc{S}[\p_{s}(\sigma \p_{s}\bm{X})] = -\p_{s}\bm{X}\cdot \p_{s}\mc{S}[\bm{F}].
\end{align*}
Let us define operator $\mc{L}$ and $\mc{Q}$ as follows:
\begin{align}\label{e:L_defn}
\mc{L}\sigma=\mc{Q}[\p_s(\sigma\p_s\bm{X})], \; \mc{Q}[\bm{F}]=\p_{s}\bm{X}\cdot \p_{s} \mc{S}[\bm{F}].
\end{align}
Note that $\mc{L}$ depends on $\bm{X}$.
Thus, equation \eqref{e:inextensible} becomes
\begin{align}\label{e:rewrite_inextensible}
\mc{L}\sigma =-\mc{Q}[\bm{F}].
\end{align} 
To study the solvability of equation \eqref{e:rewrite_inextensible}, we need to understand the mapping properties 
of $\mc{L}$ and $\mc{Q}$ under certain assumptions on the regularity of $\Gamma$.
Let us write the operator $\mc{Q}$ as follows
\begin{equation}\label{QFT}
\begin{split}
\mc{Q}[\bm{g}]&= \p_{s}\bm{X}\cdot \p_{s}\mc{S}[\bm{g}] = \p_{s}\bm{X} \cdot \p_{s} \int_{\mbs} G(\bm{X}-\bm{X}')\bm{g}'ds'\\
&= \frac{1}{4\pi}\paren{\p_{s}\bm{X} \cdot \wt{F}_C[\bm{g}]+\p_s\bm{X}\cdot F_T[\bm{g}]}, \\
\wt{F}_C[\bm{g}]&=\p_s \int_{\mbs} G_L(\bm{X}-\bm{X}')\bm{g}'ds',\quad F_T[\bm{g}]=\p_s \int_{\mbs} G_T(\bm{X}-\bm{X}')\bm{g}'ds'.
\end{split}
\end{equation}
It is useful to further decompose the operator $\wt{F}_C$.
When $\abs{s-s'}\ll 1$, we have
\begin{align*}
 -\p_{s}G_L(\bm{X}-\bm{X}')=\frac{\Delta\bm{X}\cdot \p_{s}\bm{X}}{\abs{\Delta\bm{X}}^2}\approx \frac{1}{s-s'}\approx \frac{1}{2}\cot \left( \frac{s-s'}{2}\right).
\end{align*}
Note that $\frac{1}{2\pi}\cot \left( \frac{s-s'}{2}\right)$ is the kernel of the Hilbert transform $\mathcal{H}$ on $\mbs$, i.e.
\begin{align*}
    (\mc{H}f)(s):=\frac{1}{2\pi}\mbox{p.v.}\int_{\mbs} \cot \left( \frac{s-s'}{2}\right)f\paren{s'}ds'.
\end{align*}
We thus rewrite $\wt{F}_{C}$ as
\begin{equation}\label{FC}
\begin{split}
\frac{1}{4\pi}\wt{F}_C[{\bm{g}}]&=-\frac{1}{4}\mc{H}\bm{g}+\frac{1}{4\pi}F_C[\bm{g}],\\
    F_C[\bm{g}]&=\int_{\mbs}K_C(s,s')\mathbf{g}'ds', \; K_C(s,s')=\frac{1}{2}\cot \left( \frac{s-s'}{2}\right)-\frac{\Delta\bm{X}\cdot \p_{s}\bm{X}}{\abs{\Delta\bm{X}}^2},
    \end{split}
\end{equation}
so we obtain
\begin{align*}
    \mc{Q}[\bm{F}] =-\frac{1}{4}\pd{s}\bm{X}\cdot\mc{H}\bm{F}+\frac{1}{4\pi}\pd{s}\bm{X}\cdot\paren{F_{C}\left[\bm{F}\right]+F_{T}\left[\bm{F}\right]}.
\end{align*}
To state our main result for the operator $\mc{Q}$, we split $\bm{F}$ into tangential and normal components to $\bm{X}$
\begin{align}\label{Ff1f2}
    \bm{F}= f_1\bm{\tau}+f_2\bm{n}=f_1 \pd{s}\bm{X}+f_2 \pd{s}\bm{X}^\perp, %\; \p_s \bm{X}^\perp=\begin{pmatrix} 0 & 1 \\ -1 & 0 \end{pmatrix} \p_s\bm{X}.
\end{align}
\begin{proposition}\label{p:rhs_f}
Let $\gamma\in(0, 1)$ and $\bm{X}\in C^{2}(\mbs)$ with $\starnorm{\bm{X}} > 0$. 
Let $\bm{F}$ be in the form of \eqref{Ff1f2}, and suppose that $f_1\in C^{0,\gamma}(\mbs)$ and $f_2\in C^0(\mbs)$. Then, 
\begin{equation*}\label{e:rhs_splitting}
 \mc{Q}[\bm{F}]= -\frac{1}{4}\mc{H}f_1 + \mc{M}_1(f_1)+\mc{M}_2(f_2),
\end{equation*}
where $\mc{M}_1$ and $\mc{M}_2$ are bounded linear operators from $C^{0}(\mbs)$ to $C^{0, \alpha}(\mbs)$ for any $\alpha\in (0,1)$.
In particular, if $\bm{F}\in C^{0,\gamma}(\mbs)$, then $\mc{Q}(\bm{F})\in C^{0,\gamma}(\mbs)$.
\end{proposition}
To establish the above result, we first show in Proposition \ref{p:kernelest01} that $F_C$ and $F_T$ are operators that map functions in $C^0(\mbs)$ to $C^{0,\alpha}(\mbs), \alpha\in (0,1)$. 
This follows from the study of the properties of the associated kernels, and estimates similar to those used in \cite{MRS2019}.

We then turn to the Hilbert transform term. Our main technical result is in the following. Suppose $\bm{v}\in C^1(\mbs)$ and $f\in C^0(\mbs)$. 
In Proposition \ref{p:commute}, we shall establish the following commutator estimate
\begin{align*}
    \mc{H}(f\bm{v})=(\mc{H}f)\bm{v}+\bm{R}(f,\bm{v}),
\end{align*}
where $\bm{R}\in C^{0,\alpha}(\mbs)$. This immediately yields
\begin{equation*}
\p_s\bm{X}\cdot (\mc{H}(f_1 \pd{s}\bm{X}+f_2 \pd{s}\bm{X}^\perp))=\mc{H} f_1+ R
\end{equation*}
where $R\in C^{0,\alpha}(\mbs)$. This, together with the estimates on $F_C$ and $F_T$ discussed above, yields Proposition \ref{p:rhs_f}.

%In section \ref{p:well-posedness} we obtain the well-posedness of the inextensible problem.
The mapping properties of $\mc{L}$ are obtained as a direct consequence of the mapping properties of $\mc{Q}$ established in Proposition \ref{p:rhs_f}.
% First, according Proposition \ref{p:rhs_f}, we obtain the condition of the force $\bm{F}$.
%Then, we split $\mc{L}(\cdot)$ into a principle term $\frac{1}{4}\mc{H}(\p_{s}(\cdot))$ and a remainder term $B(\cdot)$ with lower order.
Indeed, note that
\begin{align*}
    \pd{s}\paren{\sigma\pd{s}\bm{X}}=\pd{s}\sigma\pd{s}\bm{X}+\sigma\pdd{s}{2}\bm{X}=\pd{s}\sigma\pd{s}\bm{X}+\widetilde{\sigma}\pd{s}\bm{X}^\perp, \; 
    \widetilde{\sigma}=\sigma \pdd{s}{2}\bm{X}\cdot\pd{s}\bm{X}^\perp.
\end{align*}
Applying Proposition \ref{p:rhs_f} with $f_1=\p_s\sigma$ and $f_2=\wt{\sigma}$, we obtain the following result.
\begin{proposition}\label{p:L_mapping_prop}
Given $\bm{X}\in C^{2}(\mbs)$ with $\starnorm{\bm{X}} > 0$, then, for any $\gamma \in (0, 1)$, we have $\mc{L}: C^{1, \gamma}(\mbs) \mapsto C^{0, \gamma}(\mbs)$. In particular, 

\begin{align*}
\mc{L}(\cdot) = -\frac{1}{4}\mc{H}(\p_{s}(\cdot)) + \mc{M}(\cdot),
\end{align*}
where $\mc{M}$ is a bounded operator from $C^{1}(\mbs)$ to $C^{0, \alpha}(\mbs)$ for any $\alpha \in (0, 1)$.
\end{proposition}
Finally, to solve \eqref{e:rewrite_inextensible}, we consider the following equation
\begin{equation*}
\paren{I + \frac{1}{4}\mc{H}\p_{s}}^{-1}\mc{L}\sigma=-\paren{I + \frac{1}{4}\mc{H}\p_{s}}^{-1}\mc{Q}[\bm{F}].
\end{equation*}
Thanks to Proposition \ref{p:L_mapping_prop}, the operator acting on $\sigma$ on the right hand side can be written as $I+\mc{K}$,
where $\mc{K}$ is a compact operator from $C^{1, \gamma}(\mbs)$ to itself. We may use Fredholm theory to establish the well-posedness, 
which implies that $\mc{L}$ is invertible if and only if its nullspace is trivial. We demonstrate that $\mc{L}$ has a non-trivial 
nullspace if and only if $\Gamma$ is a circle. If $\Gamma$ is a circle, the nullspace is given by the constant functions, i.e.
\begin{equation}\label{constcircle}
\mc{L}1=0 \text{ if } \Gamma \text{ is a circle}.
\end{equation}

\subsection{Behavior of $\mc{L}$ near the unit circle}
As we saw above, the operator $\mc{L}$ has a non-trivial null-space if and only if $\Gamma$ is a circle.
Thus, as $\Gamma$ approaches a circle, we expect $\mc{L}$ to become increasingly singular.
In Section \ref{property_L}, we study the behavior of $\mc{L}$ when $\Gamma$ is close to a circle.

For purposes of studying $\mc{L}$ near a circle, it is convenient to change the parametrization of $\bm{X}$ from the arclength $s$
to polar coordinate $\theta$. Condition \eqref{e:stress_jump} now becomes
\begin{equation*}
\jump{\paren{\nabla\bm{u} + (\nabla\bm{u})^T - p\mb{I}}\bm{n}}\abs{\pd{\theta}\bm{X}} = \bm{F}(\theta) + \p_{\theta}(\sigma(\theta) \bm{\tau}(\theta)) \quad \text{on } \Gamma.\label{e:jump_stress_theta}
\end{equation*}
where $\bm{F}(\theta)$ is the force density with respect to the $\theta$ variable. 
The tension determination problem in the $\theta$ variable can be reduced to an integral equation in exactly the same way 
as in the case of arclength parametrization. 
In fact, the problem we obtain turns out to be identical to \eqref{e:rewrite_inextensible} except the arclength variable $s$ is replaced by the polar coordinate $\theta$
\begin{equation*}
\begin{split}
\mc{L}_\theta \sigma&=-\mc{Q}_\theta [\bm{F}], \; \mc{L} \sigma=\mc{Q}_\theta[\p_\theta(\sigma\bm{\tau})], \; \mc{Q}_\theta[\bm{F}]=\bm{\tau}\cdot\p_\theta\mc{S}[\bm{F}], \\
\mc{S}_\theta [\bm{g}]&=\int_{\mathbb{S}^1} G(\bm{X}(\theta)-\bm{X}(\theta'))\bm{g}(\theta')d\theta'.
\end{split}
\end{equation*}
In Section \ref{s:Ltheta}, we obtain a precise relationship between $\mc{L}$ (in the arclength variable) and $\mc{L}_\theta$ defined above.
Like $\mc{L}$, it is shown that $\mc{L}_\theta$ maps $C^{1,\gamma}(\mbs)$ to $C^{\gamma}(\mbs)$. 
Furthermore, $\mc{L}_\theta$ is invertible if and only if $\mc{L}$ is invertible.
We may thus study the invertibility of $\mc{L}_\theta$ when $\Gamma$ is near a circle. Define
\begin{equation}\label{Xveps}
    \bm{X}_\veps=\bm{X}_c+\varepsilon\bm{Y}=\paren{1+\veps g}\bm{X}_c, \; \bm{X}_c=\begin{pmatrix} \cos \theta \\ \sin\theta \end{pmatrix}.
\end{equation}
where $g$ is a $C^2$ function. When $\veps=0$, $\bm{X}_\veps=\bm{X}_0$ is the unit circle. 
Let $\mc{L}_\veps$ be the operator $\mc{L}_\theta$ when $\bm{X}=\bm{X}_\veps$, then
\begin{equation*}
\begin{split}
    \mc{L}_\varepsilon\sigma
    &=\bm{\tau}_\varepsilon\cdot \pd{\theta}\mc{S}_\varepsilon\left[\pd{\theta}(\sigma\bm{\tau}_\varepsilon)\right]\\
    &=\bm{\tau}_\veps\cdot \pd{\theta}\int_\mbs G(\Delta \bm{X}_\veps)\pd{\theta'}(\sigma(\theta')\bm{\tau}_\veps(\theta'))d\theta',\;
    \bm{\tau}_\veps=\frac{\pd{\theta}\bm{X}_\varepsilon}{\abs{\pd{\theta}\bm{X}_\varepsilon}}.
    \end{split}
\end{equation*}
Here and henceforth, $\Delta f$ denotes the difference $f(\theta)-f(\theta')$ when applied to a function of $\theta$. When $\veps=0$, the arclength and 
polar coordinates coincide, and thus $\mc{L}_0$ has a nullspace of constant functions (see \eqref{constcircle}). Our goal is to understand the behavior of this null space.
We consider the following eigenvalue problem
\begin{equation*}%\label{eigprob}
\mc{L}_\veps \sigma_\veps=\lambda_\veps \sigma_\veps, \quad \int_{\mathbb{S}^1}\sigma_\veps^2 d\theta=2\pi, \;  \text{ where } \lambda_0=0, \; \sigma_0=1.
\end{equation*}
% The above is an eigenvalue perturbation problem. 
For small values of $\veps$, $\lambda_\veps$ is non-zero but small, and is expected to quantify the near-singularity of $\mc{L}_\veps$.
We will prove the following result.
\begin{theorem}\label{t:lambdaveps}
Suppose $g$ in \eqref{Xveps} is in $C^2(\mbs)$, and suppose it has the following Fourier expansion
\begin{equation}\label{gexp}
g(\theta)=g_0+\sum_{n\geq 1} \paren{g_{n1}\cos\theta+g_{n2} \sin\theta}.
\end{equation}
There is an $\veps_*>0$ so that if $\abs{\veps}\leq \veps_*$, there is a unique $\lambda_\veps$ that satisfies \eqref{eigprob} with the following properties:
\begin{enumerate}
\item $\lambda_\veps$ is smooth in $\veps$ and $\lambda_\veps\leq 0$.
\item If $\bm{X}_\veps$ is not a circle for $\veps\neq 0$, there are constants $C_1$ and $C_2$ that do not depend on $\veps$ so that:
\begin{equation*}
\norm{\mc{L}_\veps^{-1}}_{\mc{B}(C^{\gamma}(\mathbb{S}^1);C^{1,\gamma}(\mathbb{S}^1))}\leq C_1+\frac{C_2}{\abs{\lambda_\epsilon}} \text{ for } 0<\abs{\veps}\leq \veps_*
\end{equation*}
where the left hand side is the operator norm of $\mc{L}_\veps^{-1}$ as a map from $C^{\gamma}(\mathbb{S}^1)$ to $C^{1,\gamma}(\mathbb{S}^1)$. 
\item $\lambda_\veps$ has the following expansion around $\veps=0$:
\begin{equation}\label{lambdaest}
\lambda_\veps=\lambda_2 \veps^2+\mc{O}(\abs{\veps}^3), \; \lambda_2=-\frac{1}{8}\sum_{n\geq 2} n(n^2-1)(g_{n1}^2+g_{n2}^2).
\end{equation}
\end{enumerate}
\end{theorem}
The first item in the above theorem follows by the implicit function theorem and is proved in Section \ref{s:ep}. The non-positivity of $\lambda_\veps$ is shown in Section \ref{s:Ltheta}, 
as a consequence of the negative semi-definiteness of the operator $\mc{L}_\theta$. 
The second item, which demonstrates that magnitude of $\lambda_\veps$ controls the near singularity of $\mc{L}_\veps$, is also shown in the same Section.
The third item is the subject of Section \ref{s:lambda2}.

In Section \ref{s:numerics}, expression \eqref{lambdaest} is verified against numerical experiments. 
We use a boundary integral method to solve the tension determination problem and to compute the eigenvalues of the operator $\mc{L}_\veps$.
We see that the expression for $\lambda_2$ is in excellent agreement with the numerically calculated eigenvalues.
In particular, when $g_{n1}=g_{n2}=0$ for $n\geq 2$, we see that $\lambda_2=0$.
In this case, we expect that $\lambda_\veps=\mc{O}(\abs{\veps}^4)$ since $\lambda_\veps\leq 0$.
We indeed observe this behavior in our numerical experiments.
We summarize the results and discussion for future outlook in Section \ref{conclusion}.

In Appendix \ref{layer_potential_appendix}, we have collected some basic statements about layer potentials for Stokes flow and their proofs.
These results are standard and classical, but we have found it difficult to locate in the literature the precise statements we need in this paper.
Appendix \ref{unitcircle_appendix} contains some calculations needed to carry out perturbative calculations around the unit circle performed in Section \ref{s:lambda2}

\section{Well-posedness of the Tension Determination Problem}\label{est_IOs}
\subsection{Stokes Interface Problem and Layer Potentials}\label{sect:classical}
Consider the following Stokes interface problem
\begin{equation}\label{Stokes_interface}
\begin{split}
-\Delta \bm{u}+\nabla p=0,&\quad  \nabla \cdot \bm{u}=0, \text{ in } \mathbb{R}^2\backslash \Gamma,\\
\jump{\bm{u}}&=0, \quad \jump{\Sigma \bm{n}}=\bm{F} \text{ on } \Gamma.
\end{split}
\end{equation}
We seek a solution in the function spaces given in \eqref{upfnspace} and suppose that the stress jump condition is satisfied 
in the sense of item \ref{item:stressjump} of Definition \ref{STDP}. 
That is, given the stress $\Sigma(\bm{x})$ defined as in \eqref{sigF}, 
the uniform limits of \eqref{Siglim} exist and that the limiting functions satisfy $\bm{F}_{\Omega_1}-\bm{F}_{\Omega_2}=\bm{F}$.
We quote the following result.
\begin{theorem}\label{classical}
Suppose $\bm{X}\in C^2(\mbs)$ $\starnorm{\bm{X}}>0$ and $\bm{F}\in C^0(\mbs)$.
Then, $\bm{u}(\bm{x})=\wh{\mc{S}}[\bm{F}]$ and $p(\bm{x})=\mc{P}[\bm{F}](\bm{x})$ defined in \eqref{e:u_soln} and \eqref{e:p_soln} 
is a solution to the Stokes interface problem \eqref{Stokes_interface}.
Moreover, $\bm{u}(\bm{x})$ is continuous across the interface $\Gamma$ and
\begin{align*}
    \abs{\bm{u}(\bm{x})}\rightarrow 0 \mbox{ as } \bm{x}\rightarrow\infty \Leftrightarrow \int_{\mbs} \bm{F}ds=0.
\end{align*}
\end{theorem}
\begin{remark}
It is clear that $\bm{u}(\bm{x})=\wh{\mc{S}}[\bm{F}]$ and $p(\bm{x})=\mc{P}[\bm{F}](\bm{x})$ belong to the function spaces in \eqref{upfnspace}, and
satisfy the Stokes equation. It is standard that $\wh{\mc{S}}[\bm{F}]$ is continuous across the interface $\Gamma$.
The important part is to check whether $\bm{u}$ and $p$ indeed satisfy the stress interface condition in the sense specified above. 
%We can find some results about $\bm{u}, p$ in \cite{P_1992}, but we haven't found the proof for the lemma from textbooks and articles.
This result is classical and can be found for example in \cite{P_1992}. We leave the discussion about Theorem \ref{classical} in Appendix \ref{layer_potential_appendix}.
\end{remark}
As a consequence, we also have the following result. 
\begin{corollary}\label{coroIP}
Let $\bm{X}$ and $\bm{F}$ be as in Theorem \ref{classical}, and suppose $\bm{F}$ satisfies:
\begin{equation}\label{Fzero}
\int_{\mbs} \bm{F}ds=0.
\end{equation}
Let $\bm{u}=\wh{\mc{S}}[\bm{F}]$. Then, we have
\begin{equation*}
\frac{1}{2}\int_{\mbr^2\setminus \Gamma} \abs{\nabla \bm{u} + (\nabla\bm{u})^T}^2 d\bm{x}=\int_{\Gamma} \bm{u}\cdot \bm{F} ds.
\end{equation*}
\end{corollary}
\begin{proof}
This follows from the usual integration by parts argument for \eqref{Stokes_interface} where we set $\bm{u}=\wh{\mc{S}}[\bm{F}]$ and $p=\mc{P}[\bm{F}]$. Integration by parts is justified by Theorem \ref{classical}, and the sufficiently fast decay of $\bm{u}$ and $\nabla \bm{u}$ at infinity, which is in turn guaranteed by \eqref{Fzero}. We omit the details.
\end{proof}

We state the uniqueness statement for the Stokes interface problem as follow.
\begin{proposition}\label{prop:class_well}
Suppose $\bm{X}\in C^2(\mbs)$ $\starnorm{\bm{X}}>0$ and $\bm{F}\in C^0(\mbs)$ and satisfies \eqref{Fzero}. Consider a solution to the Stokes 
interface problem \eqref{Stokes_interface} that satisfies the growth condition \eqref{updecay} at infinity.
Then, the unique solution $\bm{u}$ is given by $\bm{u}=\wh{\mc{S}}[\bm{F}]$ with $p=\mc{P}[\bm{F}]+c$ where $c$ is an arbitrary constant.
\end{proposition}
This is also well-known, but we have not been able to find this precise statement and proof. We include a proof of this fact for completeness.
\begin{proof}
That $\bm{u}=\wh{\mc{S}}[\bm{F}]$ and $p=\mc{P}[\bm{F}]+c$ satisfy \eqref{Stokes_interface} follows from Theorem \ref{classical}. 
It is also clear that they satisfy the decay estimate \eqref{updecay}. What remains to be shown is that the problem \eqref{Stokes_interface} with $\bm{F}=0$
only admits the trivial solution $\bm{u}=0$ and $p=c$ where $c$ is an arbitrary constant. Let $\bm{w}=(w_1,w_2)$ and $v$ be compactly supported smooth functions in $\mathbb{R}^2$.
Multiply the Stokes equations by $\bm{w}$ and the incompressibility condition by $v$. Integrate by parts and use the interface condition with $\bm{F}=0$ to obtain
\begin{equation}\label{weakup}
\int_{\mathbb{R}^2} \paren{\bm{u}\cdot\nabla (\nabla \cdot \bm{w})+ \bm{u}\cdot\Delta \bm{w} +p\nabla \cdot \bm{w}}d\bm{x}=0, \quad \int_{\mathbb{R}^2} \bm{u}\cdot \nabla v d\bm{x}=0,
\end{equation}
Let $\phi$ be a compactly supported smooth function, and let $\bm{w}=\nabla \phi$. Plugging this into the first equation in the above, we have
\begin{equation*}
\int_{\mathbb{R}^2} \paren{2\bm{u}\cdot\nabla (\Delta \phi)+p\Delta \phi}d\bm{x}=\int_{\mathbb{R}^2}p\Delta \phi d\bm{x}=0, 
\end{equation*}
where we used the second equation in \eqref{weakup} with $v=\Delta \phi$ in the first equality above.
Since $p\in L^1_{\rm loc}(\mathbb{R}^2)$, $p$ is a distribution, and is weakly harmonic. By a result of Weyl (see, for example Appendix B of \cite{lax2002functional}), weakly harmonic functions are harmonic. 
Thus, $p$ is smooth and satisfies $\Delta p=0$. Given \eqref{updecay}, $p=c$ by Liouville's theorem. 
Putting $p=c$ in the first equation of \eqref{weakup}, and using the 
second equation in \eqref{weakup} with $v=\nabla \cdot \bm{w}$ we have
\begin{equation*}
\int_{\mathbb{R}^2} \bm{u}\cdot\Delta \bm{w}d\bm{x}=0.
\end{equation*}
This again implies that each component of $\bm{u}$ is weakly harmonic, and thus, harmonic. Given \eqref{updecay}, $\bm{u}=0$ by Liouville's theorem.
\end{proof}

\subsection{Properties of Operators $\mc{Q}$ and $\mc{L}$}
Let $\bm{u},p,\sigma$ be a solution to the tension determination problem in the sense of Definition \ref{STDP}. 
According to Proposition \ref{prop:class_well}, $\bm{u}(\bm{X}(s))$ can be expressed in terms of $\sigma$ and $\bm{F}$ as follows
\begin{equation*}
\bm{u}(\bm{X}(s))=\wh{\mc{S}}[\bm{F}+\p_s(\sigma\p_s\bm{X})](\bm{X}(s))=\mc{S}[\bm{F}+\p_s(\sigma\p_s\bm{X})](s)
\end{equation*}
where $\mc{S}$ was defined in \eqref{uXs}. We now study the properties of $\p_s\mc{S}[\bm{F}]$.
For this purpose, we make use of the decomposition discussed in \eqref{QFT} and \eqref{FC}.
We start with some technical results.
\begin{lemma}\label{prelim_ests}
Let $\bm{X}=\paren{X_1, X_2} \in C^{2}$ with $\starnorm{\bm{X}} > 0$ and $Y\in C^1$. Consider
\begin{equation*}
A_i(s,s')=\frac{\Delta X_i}{\abs{\Delta \bm{X}}}, \quad B_i(s,s')=\frac{1}{\abs{\Delta \bm{X}}}\paren{\p_sX_i-\frac{\Delta X_i}{s-s'}}.
\end{equation*}
We have
\begin{align}
\label{YXest}
\abs{\p_s\paren{\frac{\Delta Y}{\abs{\Delta \bm{X}}}}}&\leq C\frac{\norm{Y}_{C^1}\norm{\bm{X}}_{C^1}}{\starnorm{\bm{X}}^2}\abs{s-s'}^{-1},\\
\label{Aiest}
\abs{\p_s A_i}\leq C\frac{\norm{\bm{X}}_{C^2}}{\starnorm{\bm{X}}},\quad \abs{\p_{s'} A_i}&\leq C\frac{\norm{\bm{X}}_{C^2}}{\starnorm{\bm{X}}},\quad
\abs{\p_{s'}\p_s A_i}\leq C\frac{\norm{\bm{X}}_{C^2}^2}{\starnorm{\bm{X}}^2}\abs{s-s'}^{-1}, \\
\label{Biest}
\abs{B_i}\leq \frac{\norm{\bm{X}}_{C^2}}{\starnorm{\bm{X}}}, \quad \abs{\p_sB_i}&\leq C\frac{\norm{\bm{X}}_{C^2}^2}{\starnorm{\bm{X}}^2}\abs{s-s'}^{-1}.
\end{align}
\end{lemma}
A naive application of estimate \eqref{YXest} to $\p_sA_i$ will {\em not} produce the first 
inequality in \eqref{Aiest}, and here, we take advantage of an additional cancellation. Similar cancellations are used to establish \eqref{Biest}.
\begin{proof}
The estimate \eqref{YXest} can be established by direct computation as
\begin{equation*}
\abs{\p_s\paren{\frac{\Delta Y}{\abs{\Delta \bm{X}}}}}=\abs{\frac{\p_sY}{\abs{\Delta \bm{X}}}-\frac{\Delta Y\p_s\bm{X}\cdot \Delta \bm{X}}{\abs{\Delta \bm{X}}^3}}\leq C\frac{\norm{Y}_{C^1}\norm{\bm{X}}_{C^1}}{\starnorm{\bm{X}}^2}\abs{s-s'}^{-1}.
\end{equation*}
For the remaining inequalities, we make repeated use of the following:
\begin{equation}\label{psXdX}
    \abs{\pd{s}X_i- \frac{\Delta X_i}{s-s'}}
    %=&\pd{s}X_i-\frac{1}{s-s'}\int_0^1 \frac{d}{d \theta}  X_i \paren{\theta s+\paren{1-\theta}s'}d\theta\\
    =\abs{\int_0^1 \paren{\pd{s}X_i\paren{s}-\pd{s}X_i\paren{\theta s+\paren{1-\theta}s'}} d\theta}\leq \norm{\bm{X}}_{C^2}\abs{s-s'}.
\end{equation}
From this, the first inequality in \eqref{Biest} is immediate. Let us move on to the first inequality in \eqref{Aiest}.
\begin{equation*}
    \pd{s}\paren{\frac{\Delta X_i}{\abs{\Delta\bm{X}}}}
    = \frac{1}{\abs{\Delta\bm{X}}}\paren{\pd{s}X_i- \frac{\Delta X_i}{\abs{\Delta \bm{X}}}\frac{\Delta \bm{X}}{\abs{\Delta \bm{X}}}\cdot \p_s\bm{X}}.
\end{equation*}
The expression in the parenthesis on the right hand side of the above is
\begin{equation*}\label{psXdX2}
\begin{split}
&\abs{\pd{s}X_i- \frac{\Delta X_i}{\abs{\Delta \bm{X}}}\frac{\Delta \bm{X}}{\abs{\Delta \bm{X}}}\cdot \p_s\bm{X}}\\
    =&\abs{\pd{s}X_i- \frac{\Delta X_i}{s-s'}
    +\frac{\Delta X_i}{\abs{\Delta \bm{X}}}\sum_{k=1}^2\frac{\Delta X_k}{\abs{\Delta \bm{X}}}\paren{ \frac{\Delta X_k}{s-s'}-\pd{s}X_k}}\\
    \leq &C\norm{\bm{X}}_{C^2}\abs{s-s'}
\end{split}
\end{equation*}
where we used \eqref{psXdX} in the inequality. From this, we obtain the first inequality in \eqref{Aiest}. The second inequality in \eqref{Aiest} follows 
from the first inequality since
\begin{equation*}
\abs{\p_{s'}A_i}=\abs{\p_{s'}A_i(s',s)}.
\end{equation*}
For the third inequality in \eqref{Aiest}, we have
\begin{align*}
    \pd{s}\pd{s'}\paren{\frac{\Delta X_i}{\abs{\Delta\bm{X}}}}
    =   &\quad\frac{\Delta\bm{X}\cdot \pd{s'}\bm{X}'}{\abs{\Delta\bm{X}}^3}\paren{\pd{s}X_i- \frac{\Delta X_i}{\abs{\Delta \bm{X}}}\frac{\Delta \bm{X}}{\abs{\Delta \bm{X}}}\cdot \p_s\bm{X}}\\
        &+\frac{\Delta\bm{X}\cdot \pd{s}\bm{X}}{\abs{\Delta\bm{X}}^3}\paren{\pd{s'}X_i'- \frac{\Delta X_i}{\abs{\Delta \bm{X}}}\frac{\Delta \bm{X}}{\abs{\Delta \bm{X}}}\cdot \p_{s'}\bm{X}'}\\
        &+\frac{\Delta X_i}{\abs{\Delta\bm{X}}^3}\sum_{j=1}^2\pd{s}X_j
        \paren{\pd{s'}X_j'- \frac{\Delta X_j}{\abs{\Delta \bm{X}}}\frac{\Delta \bm{X}}{\abs{\Delta \bm{X}}}\cdot \p_s\bm{X}}.
        %\paren{\pd{s'}X_j'\abs{\Delta\bm{X}}^2-  \Delta X_j\sum_{k=1}^2\Delta X_k\pd{s}X_k}.
\end{align*}
The terms in the three parentheses on the right hand side of the above can be estimated in the same way as in \eqref{psXdX}. This yields the third inequality in \eqref{Aiest}.
Finally, for the second inequality in \eqref{Biest}, note first that
\begin{equation*}
 B_i=\frac{\Delta\p_{s} X_i}{\abs{\Delta \bm{X}}}+\frac{(\p_{s'} X'_i -\Delta X_i/(s-s'))}{\abs{\Delta \bm{X}}}=: B_{i,1}+B_{i,2}
 \end{equation*}
Using \eqref{YXest}, we have
 \begin{equation*}
 \abs{\p_sB_{i,1}}\leq C\frac{\norm{\bm{X}}_{C^2}^2}{\starnorm{\bm{X}}^2}\abs{s-s'}^{-1}.
 \end{equation*}
For $B_{i,2}$, we have
\begin{align*}
    \pd{s}B_{i,2}=-\frac{(\p_{s} X_i -\Delta X_i/(s-s'))}{\paren{s-s'}\abs{\Delta \bm{X}}}-\frac{\Delta\bm{X}\cdot\pd{s}\bm{X}}{\abs{\Delta \bm{X}}^3}\paren{\p_{s'} X'_i -\frac{\Delta X_i}{s-s'}}.
\end{align*}
Using \eqref{psXdX}, we see that
\begin{equation*}
\abs{\p_s B_{i,2}}\leq C\frac{\norm{\bm{X}}_{C^2}^2}{\starnorm{\bm{X}}^2}\abs{s-s'}^{-1}.
\end{equation*}
Combining the above estimates on $\p_s B_{i,1}$ and $\p_s B_{i,2}$, we obtain the second inequality in \eqref{Biest}.
\end{proof}

The above lemma allows us to prove the following estimates on $F_C[\bm{g}]$ and $F_T[\bm{g}]$, defined in \eqref{QFT} and \eqref{FC}.
\begin{proposition}\label{p:kernelest01}
Let $\bm{X}\in C^{2}(\mbs)$ with $\starnorm{\bm{X}} > 0$, $\bm{g} \in C^{0}(\mbs)$. Then, for any $\alpha \in (0, 1)$,
\begin{equation*}
    \hnorm{F_C[\bm{g}]}{0}{\alpha}\leq C\frac{\chnorm{\bm{X}}{2}^2}{\starnorm{\bm{X}}^2}\chnorm{\bm{g}}{0}.\label{e:kernelest01_1}
\end{equation*}
and
\begin{equation*}
    \hnorm{F_{T}[\bm{g}]}{0}{\alpha}\leq C\frac{\chnorm{\bm{X}}{2}^2}{\starnorm{\bm{X}}^2}\chnorm{\bm{g}}{0},\label{e:kernelest02_1}
\end{equation*}
where $C$ depends on $\alpha$, and is independent of $\bm{X}$ and $\bm{g}$.
\end{proposition}

\begin{proof}
First, let us estimate $F_{T}[\bm{g}]$. %Since
%\begin{align*}
   % \paren{F_{T}[\mathbf{g}]}_i=\frac{1}{4\pi}\sum_{j=1}^2\pd{s}\int_{\mbs}\frac{\Delta X_i\Delta X_j}{\abs{\Delta\bm{X}}^2}g_j'ds',
%\end{align*}
Let
\begin{align*}
    Q[u]=\p_{s}\int_\mbs K(s,s')u'ds', \; K(s,s')=\frac{\Delta X_i\Delta X_j}{\abs{\Delta\bm{X}}^2}, \; i,j=1,2.
\end{align*}
To estimate $F_T[\bm{g}]$, it suffices to estimate $Q[u]$.
Since
\begin{align*}
     &\pd{s} K(s,s')  
     =-\pd{s'} K(s,s')+\paren{\pd{s}+\pd{s'}}K(s,s')\\
    =&-\pd{s'} K(s,s')-2\frac{\Delta\bm{X}\cdot\Delta\pd{s}\bm{X}}{\abs{\Delta\bm{X}}^2}\frac{\Delta X_i\Delta X_j}{\abs{\Delta\bm{X}}^2}+\frac{\Delta \pd{s}X_i\Delta X_j+\Delta X_i\Delta\pd{s} X_j}{\abs{\Delta\bm{X}}^2}\\
    %  &+\sum_{i=0}^n\left\{\paren{\alpha_i \frac{\Delta\pd{s} Z_i}{\abs{\Delta\bm{Z}_0}}\frac{\Delta W_i}{\abs{\Delta\bm{Z}_0}}+\beta_i\frac{\Delta Z_i}{\abs{\Delta\bm{Z}_0}}\frac{\Delta\pd{s} W_i}{\abs{\Delta\bm{Z}_0}}}\right.\\
    %  &\qquad\qquad\left.\paren{\frac{\Delta Z_i}{\abs{\Delta\bm{Z}_0}}}^{\alpha_i-1}\paren{\frac{\Delta W_i}{\abs{\Delta\bm{Z}_0}}}^{\beta_i-1}\prod_{j \neq i}\paren{\frac{\Delta Z_j}{\abs{\Delta\bm{Z}_0}}}^{\alpha_j}\paren{\frac{\Delta W_j}{\abs{\Delta\bm{Z}_0}}}^{\beta_j}\right\}\\
     :=&-\pd{s'} K(s,s')+K_2(s,s').
\end{align*}
we can split $K[u]$ as
\begin{align*}
    Q[u]=\int_\mbs \paren{-\pd{s'} K(s,s')+K_2(s,s')}u'ds'
                :=Q_1[u]+Q_2[u].
\end{align*}
Let us estimate the kernel of $K_1$. 
\begin{equation}\label{ksprime}
    \abs{\p_{s'}K(s, s')}=\abs{\p_{s'}A_iA_j} \leq \abs{A_i\p_{s'}A_j+(\p_{s'}A_i)A_j}\leq C\frac{\chnorm{\bm{X}}{2}}{\starnorm{\bm{X}}},
    \end{equation}
    where we used the notation of Lemma \ref{prelim_ests} and used \eqref{Aiest} as well as the fact that $\abs{A_i}\leq 1$.
    We thus have
    \begin{equation}\label{K1C0}
    \abs{Q_1[u]}  \leq\int_\mbs \abs{ \p_{s'}K(s,s')u'}ds'\leq C\frac{\chnorm{\bm{X}}{2}}{\starnorm{\bm{X}}}\chnorm{u}{0}.
\end{equation}
To estimate the H\"older norm of $K_1[u]$, we need the following estimate
\begin{equation*}\label{kssprime}
\begin{split}
\abs{\p_s \p_{s'} K(s,s')} &\leq \abs{A_i\p_s\p_{s'}A_j+(\p_s\p_{s'}A_i)A_j+\p_sA_i\p_{s'}A_j+\p_{s'}A_i\p_sA_j}\\
&\leq C\frac{\chnorm{\bm{X}}{2}^2}{\starnorm{\bm{X}}^{2}}\abs{s-s'}^{-1},
\end{split}
\end{equation*}
where we used \eqref{Aiest} in the inequality above.
Without loss of generality, we set $0<h<2\pi$, and define the intervals
\begin{equation*}\label{e:interval_defns}
\mc{I}_s := (s - h/2 < s' < s + 3h/2),\quad \mc{I}_f := \mbs \setminus \mc{I}_s .
\end{equation*}
Then,
\begin{align*}
    \diff_h Q_1[u]&=\int_\mbs \diff_h \p_{s'}k(s,s')u'ds'=\int_{\mc{I}_s} \diff_h \p_{s'}K(s,s')u'ds'+\int_{\mc{I}_f} \diff_h \p_{s'}K(s,s')u'ds'.
\end{align*}
On $\mc{I}_s$,
\begin{equation*}\label{K1Isdiff}
    \abs{\int_{\mc{I}_s} \diff_h \p_{s'}K(s,s')u'ds'}
    \leq C\frac{\chnorm{\bm{X}}{2}}{\starnorm{\bm{X}}}\chnorm{u}{0}\int_{\mc{I}_s} ds'\leq C\frac{\chnorm{\bm{X}}{2}}{\starnorm{\bm{X}}}\chnorm{u}{0}h
\end{equation*}
where we used \eqref{ksprime}.
On $\mc{I}_f$, by Mean Value theorem, there exists $\xi$ between 0 and $h$ s.t. 
\begin{align*}
    \abs{\diff_h \p_{s'}K(s,s')}&=\abs{\pd{s}\p_{s'}k(s+\xi,s')}h\leq Ch\frac{\chnorm{\bm{X}}{2}^2}{\starnorm{\bm{X}}^{2}}|s+\xi - s'|^{-1}\\
    &\leq Ch\frac{\chnorm{\bm{X}}{2}^2}{\starnorm{\bm{X}}^{2}}\paren{|s+h - s'|^{-1}+|s- s'|^{-1}},
\end{align*}
so
\begin{align*}
    \abs{\int_{\mc{I}_f} \diff_h \p_{s'}K(s,s')u'ds'}
    \leq&  Ch\frac{\chnorm{\bm{X}}{2}^2}{\starnorm{\bm{X}}^{2}}\chnorm{u}{0}\int_{\mc{I}_f}\paren{|s+h - s'|^{-1}+|s- s'|^{-1}} ds'\\
    \leq&  C\frac{\chnorm{\bm{X}}{2}^2}{\starnorm{\bm{X}}^{2}}\chnorm{u}{0} h\abs{\log h}%C\frac{\chnorm{\bm{X}}{2}^2}{\starnorm{\bm{X}}^{2}}\chnorm{u}{0} h^\alpha, \; 0<\alpha<1.
\end{align*}
Using the above inequality and together with \eqref{K1Isdiff}, we have
\begin{equation*}
\begin{split}
\abs{\diff_h Q_1[u]}&\leq C\frac{\chnorm{\bm{X}}{2}}{\starnorm{\bm{X}}}\chnorm{u}{0}h+C\frac{\chnorm{\bm{X}}{2}^2}{\starnorm{\bm{X}}^{2}}\chnorm{u}{0} h\abs{\log h}\\
&\leq C\frac{\chnorm{\bm{X}}{2}^2}{\starnorm{\bm{X}}^{2}}\chnorm{u}{0} h^\alpha, \; 0<\alpha<1
\end{split}
\end{equation*}
where the last constant depends on $\alpha$. Together with \eqref{K1C0}, we obtain
\begin{align}\label{K1uest}
    \hnorm{Q_1[u]}{0}{\alpha}\leq C\frac{\chnorm{\bm{X}}{2}^2}{\starnorm{\bm{X}}^{2}}\chnorm{u}{0}.
\end{align}
%where $C$ depends on $\alpha$.

The estimate $Q_2[u]$ follows similarly. First, we have the following estimate, which directly follows from the expression for $K_2$
\begin{equation*}
\abs{K_2(s, s')}\leq C\frac{\chnorm{\bm{X}}{2}}{\starnorm{\bm{X}}}.
\end{equation*}
Using \eqref{YXest} and \eqref{Aiest}, we also have
    \begin{equation*}
    \abs{\pd{s} K_2(s, s')}\leq C\frac{\chnorm{\bm{X}}{2}^{2}}{\starnorm{\bm{X}}^{2}}|s - s'|^{ - 1}.
\end{equation*}
These estimates are the same as those of $\p_{s'}k(s,s')$, and thus, using the same steps as in the estimates for $K_1[u]$, we obtain
\begin{equation*}
\hnorm{Q_2[u]}{0}{\alpha}\leq C\frac{\chnorm{\bm{X}}{2}^2}{\starnorm{\bm{X}}^{2}}\chnorm{u}{0}.
\end{equation*}
where the constant $C$ depends only on $\alpha$. The above together with \eqref{K1uest} yields the estimate on $Q[u]$ and hence on $F_T[\bm{g}]$.

Next, for $F_{C}[\mathbf{g}]$, let us define
\begin{align}
R_C (s, s') :=  \frac{1}{2}\cot\paren{\frac{s - s'}{2}}-\frac{1}{s - s'} .\label{e:R_C_defn} 
\end{align}
We decompose the kernel $K_C(s,s')$ defined in \eqref{FC} into two parts
\begin{align*}
    K_C=K_L+R_C,
\end{align*}
where
\begin{align*}
K_L(s, s') = -\frac{\Delta\bm{X}\cdot \paren{\p_{s}\bm{X} - \frac{\Delta\bm{X}}{s - s}}}{\abs{\Delta\bm{X}}^2}.
\end{align*}
By Taylor expensions for $\cot$ function, $R_C$ is smooth in both $s$ and $s'$, and
\begin{align*}
    \abs{ R_C (s, s')}\leq C_0 \abs{s-s'},~~\abs{\p_s R_C (s, s')}\leq C_1
\end{align*}
for some constant $C_0, C_1$ which depend only on the expansion of $\cot(s)$.
Using \eqref{Aiest} and \eqref{Biest}, we see that
\begin{align*}
 \abs{K_L (s, s')} \leq C\frac{\chnorm{\bm{X}}{2}}{\starnorm{\bm{X}}}, \quad  \abs{\p_{s} K_L (s, s')} \leq C\frac{\chnorm{\bm{X}}{2}^2}{\starnorm{\bm{X}}^2}|s - s'|^{-1}.   
\end{align*}
Combining the the bounds on $R_C$ and $K_L$, we have
\begin{equation*}\label{e:K_C_bounds}
    \abs{K_C (s, s')} \leq C\frac{\chnorm{\bm{X}}{2}}{\starnorm{\bm{X}}}, \quad  \abs{\p_{s} K_C (s, s')} \leq C\frac{\chnorm{\bm{X}}{2}^2}{\starnorm{\bm{X}}^2}|s - s'|^{-1}
\end{equation*}
Using the same procedure used to prove the estimate on $Q_1[u]$, we obtain the estimate
\begin{align*}
    \hnorm{F_C[\bm{g}]}{0}{\alpha}\leq C\frac{\chnorm{\bm{X}}{2}^2}{\starnorm{\bm{X}}^2}\chnorm{\bm{g}}{0}
\end{align*}
for some constant $C$ which only depends on $\alpha$.
\end{proof}
We now have the following result on the the properties of $\p_s\mc{S}[\bm{F}]$.
\begin{corollary}\label{coropsS}
Let $\bm{X}\in C^2(\mbs)$ and $\starnorm{\bm{X}}>0$ and let $\bm{F}\in C^{0,\gamma}(\mbs)$.
Then, $\p_s\mc{S}[\bm{F}]\in C^{0,\gamma}(\mbs)$ and is given by
\begin{equation*}
\p_s\mc{S}[\bm{F}]=-\frac{1}{4}\mc{H}\bm{F}+F_C[\bm{F}]+F_T[\bm{F}].
\end{equation*}
If $\bm{F}\in C^0(\mbs)$, then the above is still valid, but $\p_s\mc{S}[\bm{F}]\in L^p(\mbs), 1<p<\infty$ and $\p_s\mc{S}[F]$ should 
be interpreted in the weak sense. That is, for any $\bm{w}\in C^1(\mbs)$, we have
\begin{equation*}
-\int_\mbs \p_s\bm{w}\cdot \bm{S}[\bm{F}]ds=\int_{\mbs} \bm{w}\cdot \p_s\mc{S}[\bm{F}]ds.
\end{equation*}
\end{corollary}
\begin{proof}
When $\bm{F}\in C^{0,\gamma}(\mbs)$, the decompositions of \eqref{QFT} and \eqref{FC} are valid pointwise, so we have
\begin{equation*}
\p_s\mc{S}[\bm{F}]=-\frac{1}{4}\mc{H}\bm{F}+F_C[\bm{F}]+F_T[\bm{F}].
\end{equation*}
Note that $\bm{F}\in C^{0,\gamma}(\mbs)\subset C^0(\mbs)$, and thus, by letting $\alpha=\gamma$ in Proposition \ref{p:kernelest01}, we have $F_C[\bm{F}]+F_T[\bm{F}]\in C^{0,\gamma}(\mbs)$. 
Since the Hilbert transform maps $C^{0,\gamma}(\mbs)$ to itself, we have $\p_s\mc{S}[\bm{F}]\in  C^{0,\gamma}(\mbs)$. If $\bm{F}\in C^0(\mbs)$, we need to use a standard approximation argument. 
Let $\bm{F}_k\in C^{0,\gamma}(\mbs)$ be a sequence of 
functions that converges to $\bm{F}$ in $C^0(\mbs)$.  Then, we have
\begin{equation*}
\p_s\mc{S}[\bm{F}_k]=-\frac{1}{4}\mc{H}\bm{F}_k+F_C[\bm{F}_k]+F_T[\bm{F}_k].
\end{equation*}
Multiplying the above by $\bm{w}\in C^1(\mbs)$ and integrating by parts, we have
\begin{equation*}
-\int_\mbs\p_s \bm{w}\cdot \mc{S}[\bm{F}_k]ds=\int_\mbs \bm{w}\cdot\paren{-\frac{1}{4}\mc{H}\bm{F}_k+F_C[\bm{F}_k]+F_T[\bm{F}_k]}ds=\int_\mbs \bm{w}\cdot \p_s\mc{S}[\bm{F}_k]ds.
\end{equation*}
Letting $k\to \infty$ and noting that the Hilbert transform is bounded from $L^p(\mbs)$ to itself when $1<p<\infty$, we obtain the desired result.
\end{proof}
Note that, as a consequence of the above corollary, $\mc{Q}[\bm{F}]=\p_s\bm{X}\cdot \p_s\mc{S}[\bm{F}]$ is well-defined for $\bm{F}\in C^0(\mbs)$.
We now proceed to prove finer properties of the operator $\mc{Q}$ as stated in Proposition \ref{p:rhs_f}.
Define the following commutator
\begin{align*}
    \left[\mathcal{H},f\right]g:= \mathcal{H}\paren{fg}-f\mathcal{H}\paren{g}.
\end{align*}
%to estimate $\mc{H}\bm{F}$ and $\mc{H}\pd{s}\paren{\sigma\pd{s}\bm{X}}$ in section \ref{p:well-posedness}.
%The transform of these terms is that set $\abs{\bm{v}\paren{s}}^2=1$,
%\begin{align}
%\begin{split}
%        &\mc{H}\paren{f \bm{v}}
%        =\paren{\bm{v}\cdot\mc{H}\paren{f \bm{v}}} \bm{v}+\paren{\bm{v}^\perp\cdot\mc{H}\paren{f \bm{v}}} \bm{v}^\perp\\
%        =&\paren{\abs{\bm{v}}^2\mc{H}f+\bm{v}\cdot\mc{H}\paren{f \bm{v}}-\abs{\bm{v}}^2\mc{H}f} \bm{v}+\paren{\bm{v}^\perp\cdot\mc{H}\paren{f \bm{v}}-\mc{H}\paren{ \bm{v}^\perp\cdot f\bm{v}}} \bm{v}^\perp\\
%        =&\paren{\mc{H}f+\bm{v}\cdot\left[\mc{H},\bm{v}\right]f}\bm{v}-\paren{\left[\mc{H},\bm{v}\cdot\right]f\bm{v}}\bm{v}^\perp.
%\end{split}\label{e: commute_decom}
%\end{align}
%Hence, let us estimate the commutator operator $\left[\mathcal{H},f\right]g$
We prove the following estimate.
\begin{proposition}\label{p:commute}
Given $f\in C^{1}(\mbs)$ and $g\in C^0\paren{\mbs}$, then for any $\alpha \in (0, 1)$,
\begin{equation*}
    \hnorm{\left[\mathcal{H},f\right]g}{0}{\alpha}\leq C\chnorm{f}{1}\chnorm{g}{0},\label{e:commut}
\end{equation*}
where $C$ depends on $\alpha$, and is independent of $f,g$.
\end{proposition}
\begin{proof}
First, we split $\left[\mathcal{H},f\right]g$ as
\begin{align*}
    \left[\mathcal{H},f\right]g
    =&\frac{1}{2\pi}\int_{\mbs}\cot\paren{\frac{s - s'}{2}}\paren{f\paren{s'}-f\paren{s}}g\paren{s'} ds'\\
    =&\frac{1}{\pi}\int_{\mbs}R_C (s, s')\paren{f\paren{s'}-f\paren{s}}g\paren{s'} ds'-\frac{1}{2\pi}\int_{\mbs}\frac{f\paren{s'}-f\paren{s}}{s'-s}g\paren{s'} ds'\\
    :=&I+II
\end{align*}
where $R_C$ is defined in \eqref{e:R_C_defn}.
Next, given the smoothness of $R_C$, we have
\begin{align*}
    \abs{I}\leq C \chnorm{f}{0}\chnorm{g}{0} \int_\mbs \abs{s-s'}ds'\leq C \chnorm{f}{0}\chnorm{g}{0},
\end{align*}
where $C$ only depends on the expansion of $\cot(s)$.
For $II$,
\begin{align*}
    \abs{II}\leq C\chnorm{f}{1}\chnorm{g}{0}\int_\mbs ds'\leq  C\chnorm{f}{1}\chnorm{g}{0}.
\end{align*}
Let us now examine $\p_s \paren{\left[\mathcal{H},f\right]g}$.
\begin{align*}
    \p_s I
    =&\quad\frac{1}{2\pi}\int_{\mbs}\p_s R_C (s, s')\paren{f\paren{s'}-f\paren{s}}g\paren{s'}ds'\\
     &-\frac{1}{2\pi}\int_{\mbs}R_C (s, s')\p_s f\paren{s}g\paren{s'} ds',
\end{align*}
so
\begin{align*}
    \abs{\p_s I}\leq C  \chnorm{f}{1}\chnorm{g}{0} \int_\mbs \abs{s-s'}ds' +C \chnorm{f}{1}\chnorm{g}{0} \int_\mbs \abs{s-s'}ds'
\end{align*}
where $C$ only depends on the expansion of $\cot(s)$.
For $\diff_h II$,
\begin{align*}
   \abs{ \p_s \frac{{f\paren{s'}-f\paren{s}}}{s-s'}}=\abs{\frac{1}{s-s'}\paren{\frac{f\paren{s}-f\paren{s'}}{s-s'}-\p_s f\paren{s}}}\leq \chnorm{f}{1}\frac{1}{\abs{ s-s'}}.
\end{align*}
With the technique of $\diff_h Q_1[u]$ in Proposition \ref{p:kernelest01},
\begin{align*}
    \abs{\diff_h II}\leq C  \chnorm{f}{1}\chnorm{g}{0} h\paren{1+\log h}\leq C  \chnorm{f}{1}\chnorm{g}{0} h^\alpha
\end{align*}
where $C$ only depends on $\alpha$. We thus obtain the desired estimate.
\end{proof}

%\section{Well-Posedness}\label{p:well-posedness}
We are now ready to prove Proposition \ref{p:rhs_f}.
\begin{proof}[Proof of Proposition \ref{p:rhs_f}]
Recall that 
\begin{equation*}\label{QF}
    \mc{Q}[\bm{F}]=-\frac{1}{4}\pd{s}\bm{X}\cdot\mc{H}\bm{F}+\frac{1}{4\pi}\pd{s}\bm{X}\cdot\paren{F_{C}\left[\bm{F}\right]+F_{T}\left[\bm{F}\right]}.
\end{equation*}
For the $\pd{s}\bm{X}\cdot F_{C}\left[\bm{F}\right]$ term,
\begin{align*}
    \pd{s}\bm{X}\cdot F_{C}\left[\bm{F}\right]=\pd{s}\bm{X}\cdot F_{C}\left[f_1 \p_{s}\bm{X}\right]+\pd{s}\bm{X}\cdot F_{C}\left[f_2 \p_{s}\bm{X}^\perp\right].
\end{align*}
By Proposition \ref{p:kernelest01},
\begin{align*}
        &\hnorm{\pd{s}\bm{X}\cdot F_{C}\left[f_1 \p_{s}\bm{X}\right]}{0}{\alpha}
    \leq C\hnorm{\p_{s}\bm{X}}{0}{\alpha}\hnorm{F_C \left[f_1 \p_{s}\bm{X}\right]}{0}{\alpha}\\
    \leq& C\frac{\chnorm{\bm{X}}{2}^3}{\starnorm{\bm{X}}^2}\chnorm{f_1 \p_{s}\bm{X}}{0}
    \leq C\frac{\chnorm{\bm{X}}{2}^3}{\starnorm{\bm{X}}^2}\chnorm{f_1 }{0}
\end{align*}
where we used $\abs{\p_s\bm{X}}=1$. Likewise, we have
\begin{align*}
    \hnorm{\pd{s}\bm{X}\cdot F_{C}\left[f_2 \p_{s}\bm{X}^\perp\right]}{0}{\alpha}\leq C\frac{\chnorm{\bm{X}}{2}^3}{\starnorm{\bm{X}}^2}\chnorm{f_2}{0}.
\end{align*}
Similarly, for the $\pd{s}\bm{X}\cdot F_{T}\left[\bm{F}\right]$ term, by Proposition \ref{p:kernelest01}, we obtain
\begin{align*}
        \hnorm{\pd{s}\bm{X}\cdot F_{T}\left[f_1 \p_{s}\bm{X}\right]}{0}{\alpha}
    \leq& C\frac{\chnorm{\bm{X}}{2}^3}{\starnorm{\bm{X}}^2}\chnorm{f_1 }{0},\\
    \hnorm{\pd{s}\bm{X}\cdot F_{T}\left[f_2 \p_{s}\bm{X}^\perp\right]}{0}{\alpha}
    \leq& C\frac{\chnorm{\bm{X}}{2}^3}{\starnorm{\bm{X}}^2}\chnorm{f_2}{0}.
\end{align*}
Let us now consider the first term in \eqref{QF} which involves the Hilbert transform.
First, note that, for $\bm{v}\in C^1$ and $f\in C^0$, we have
\begin{equation}\label{Hfv}
\mc{H}(f\bm{v})=(\mc{H}f)\bm{v}+[\mc{H},\bm{v}\cdot]f.
\end{equation}
Applying this to $f=f_1$ and $\bm{v}=\p_s\bm{X}$, we obtain
\begin{align*}
    \p_{s}\bm{X}\cdot\mc{H}\paren{f_1\p_{s}\bm{X}}=\mc{H}f_1+\p_{s}\bm{X}\cdot\paren{\left[\mc{H},\p_{s}\bm{X}\cdot\right]f_1}
\end{align*}
By Proposition \ref{p:commute}, we get
\begin{align*}
    \hnorm{\p_{s}\bm{X}\cdot\paren{\left[\mc{H},\p_{s}\bm{X}\right]f_1}}{0}{\alpha}
    \leq C\hnorm{\p_{s}\bm{X}}{0}{\alpha}\hnorm{\left[\mc{H},\p_{s}\bm{X}\cdot\right]f_1}{0}{\alpha}
    \leq C\chnorm{\bm{X}}{2}^2 \chnorm{f_1}{0}.
\end{align*}
Applying \eqref{Hfv} with $f=f_2$ and $\bm{v}=\p_{s}\bm{X}^\perp$ and using Proposition \ref{p:commute}, we obtain
\begin{align*}
        \hnorm{\p_{s}\bm{X}\cdot\mc{H}\paren{f_2\p_{s}\bm{X}^\perp}}{0}{\alpha}
    \leq\hnorm{\p_s\bm{X}\cdot\paren{\left[\mc{H},\p_{s}\bm{X}^\perp\cdot\right]f_2}}{0}{\alpha}
    \leq C\chnorm{\bm{X}}{2}^2 \chnorm{f_2}{0}.
\end{align*}
Therefore,
\begin{align*}
 \mc{Q}[\bm{F}]&= -\frac{1}{4}\mc{H}f_1 + \mc{M}_1(f_1)+\mc{M}_2(f_2),\\
    \mc{M}_1(f_1)=&-\frac{1}{4}\p_{s}\bm{X}\cdot\left[\mc{H},\p_{s}\bm{X}\right]f_1+\frac{1}{4\pi}\pd{s}\bm{X}\cdot\paren{F_{C}\left[f_1\p_{s}\bm{X}\right]+F_{T}\left[f_1\p_{s}\bm{X}\right]},\\
    \mc{M}_2(f_2)=&\frac{1}{4}\left[\mc{H},\p_{s}\bm{X}^\perp\cdot\right]f_2\p_{s}\bm{X}^\perp+\frac{1}{4\pi}\pd{s}\bm{X}\cdot\paren{F_{C}\left[f_2\p_{s}\bm{X}^\perp\right]+F_{T}\left[f_2\p_{s}\bm{X}^\perp\right]}.
\end{align*}
where $\mc{M}_1, \mc{M}_2$ are bounded maps from $C^0(\mathbb{S}^1)$ to $C^{0,\alpha}(\mbs)$ for any $\alpha\in (0,1)$. 
The statement on the mapping properties of $\mc{Q}$ follows from the well-known 
fact that the Hilbert transform is a bounded operator from $C^{0,\alpha}(\mbs)$ to itself for $\alpha\in (0,1)$ (see, for example, Section I.8.4 of \cite{katznelson2004introduction}).
\end{proof}

\begin{proof}[Proof of Proposition \ref{p:L_mapping_prop}]
%Let $\bm{X}\in C^{2}$. Further, suppose $\starnorm{\bm{X}} > 0$. 
% We will investigate the mapping properties of 
% \begin{align*}
% \mc{L}\sigma &= \p_{s}\bm{X}\cdot \p_{s}\mc{S}[\p_{s}(\sigma\p_{s}\bm{X})] = \p_{s}\bm{X} \cdot \p_{s} \int_{\mbs} G(\bm{X}, \bm{X}')\p_{s'}\paren{\sigma' \p_{s'}\bm{X}'}ds'\\
% &= \p_{s}\bm{X} \cdot \p_{s} \int_{\mbs} G_L(\bm{X}, \bm{X}')\p_{s'}\paren{\sigma' \p_{s'}\bm{X}'}ds'\\
% &+ \p_{s}\bm{X} \cdot \p_{s} \int_{\mbs} G_T(\bm{X}, \bm{X}')\p_{s'}\paren{\sigma' \p_{s'}\bm{X}'}ds' =: \mc{L}_L \sigma + \mc{L}_T \sigma.
% \end{align*}
% separately. We will show that $\mc{L}_L$ has the desired mapping properties and note that similar considerations hold for $\mc{L}_T$. We can rewrite the kernel of $\mc{L}_L$ as 
% \begin{align*}
% 4\pi\p_{s}G(\bm{X}, \bm{X}') = \frac{-\Delta\bm{X}\cdot \p_{s}\bm{X}}{\abs{\Delta\bm{X}}^2} = \frac{\Delta\bm{X}\cdot \p_{s'}\bm{X}'}{\abs{\Delta\bm{X}}^2} - \frac{\Delta\bm{X}\cdot\Delta\p_{s}\bm{X}}{\abs{\Delta\bm{X}}^2}.
% \end{align*}
Recall that
\begin{align*}
    \mc{L}\sigma=& \mc{Q}[\pd{s}\paren{\sigma\pd{s}\bm{X}}].
\end{align*}
Since
\begin{align*}
    \pd{s}\paren{\sigma\pd{s}\bm{X}}=\pd{s}\sigma\pd{s}\bm{X}+\sigma\pdd{s}{2}\bm{X}=\pd{s}\sigma\pd{s}\bm{X}+\widetilde{\sigma}\pd{s}\bm{X}^\perp
\end{align*}
where $\widetilde{\sigma}=\sigma\pd{s}\bm{X}^\perp\cdot\pdd{s}{2}\bm{X}$, we have
\begin{align*}
    \mc{L}\sigma=-\frac{1}{4}\mc{H}\pd{s}\sigma + \mc{M}_1(\pd{s}\sigma)+\mc{M}_2(\widetilde{\sigma})
\end{align*}
Since $\norm{\wt{\sigma}}_{C^0}\leq \norm{\bm{X}}_{C^2}\norm{\sigma}_{C^0}$, the result follows by Proposition \ref{p:rhs_f}.
\end{proof}

\subsection{Proof of Well-posedness}
We are now ready to prove the well-posedness of the tension determination problem.
\begin{proof}[Proof of Theorem \ref{t: well-posed_noncircle}, when $\Gamma$ is not a circle]
By Proposition \ref{prop:class_well}, solving the tension determination problem in the sense of Definition \ref{STDP} is equivalent to finding a $\sigma$ 
satisfying
\begin{equation}\label{STDP1}
\int_{\mbs} \p_s\paren{w\p_s\bm{X}}\cdot \mc{S}(\p_s(\sigma\p_s \bm{X})+\bm{F})ds=0 \text{ for any } w\in C^1(\mbs),
\end{equation}
which, thanks to and Proposition \ref{coropsS}, is equivalent to solving the following equation for $\sigma$
\begin{equation}\label{LsigQF}
\mc{L}\sigma = -\mc{Q}[\bm{F}].
\end{equation}
We would thus like to show that the above equation \eqref{LsigQF} has a unique solution in $C^{1, \gamma}(\mbs)$ when $\Gamma$ is not a circle. Proposition \ref{p:L_mapping_prop} implies
\begin{align*}
\mc{L}\sigma = \paren{-\frac{1}{4}\mc{H}\p_{s} + \mc{M}} \sigma
\end{align*}
where $\mc{M} $ is a bounded operator from $C^{1,\gamma}(\mathbb{S}^1)$ to $C^{\alpha}(\mathbb{S}^1)$ for any $\alpha$ in $(0, 1)$. We may take $\gamma<\alpha$. 
Thus, solving the above equation is equivalent to solving
\begin{align*}
\paren{\paren{I + \frac{1}{4}\mc{H}\p_{s}} - \paren{I + \mc{M}}} \sigma = -\mc{Q}[\bm{F}]
\end{align*}
where $I$ is the identity map.
It is well-known that 
\begin{align*}
\paren{I + \frac{1}{4}\mc{H}\p_{s}}^{-1}
\end{align*}
is a bounded operator from $C^{0,\alpha}(\mbs)$ to $C^{1,\alpha}(\mbs)$
for any $0 < \alpha < 1$. Thus, solving \eqref{e:rewrite_inextensible} is equivalent to solving
\begin{align*}
\paren{ I + \mc{K}} \sigma = \tilde{F}
\end{align*}
where 
\begin{align*}
\mc{K} &= \paren{I + \frac{1}{4}\mc{H}\p_{s}}^{-1}(I + \mc{M}),\\
\tilde{F} &= \paren{I + \frac{1}{4}\mc{H}\p_{s}}^{-1}\mc{Q}[\bm{F}] \in C^{1, \gamma}(\mbs).
\end{align*}
Note that $\mc{K}$ is a bounded operator from $C^{1, \gamma}$ to $C^{1, \alpha}$, and since we have chosen $\alpha>\gamma$, 
$\mc{K}$ is in fact a compact operator from $C^{1, \gamma}$ to itself. Thus, by the Fredholm Alternative Theorem, $I+\mc{K}$ is invertible if and only if
\begin{align*}
(I + \mc{K})\sigma = 0
\end{align*}
has a unique solution. We must thus show that, if $\sigma\in C^{1,\gamma}(\mbs)$ satisfies
$\mc{L}\sigma = 0$ then $\sigma=0$. %Consider the single layer solutions \eqref{upF} with force density $\p_s(\sigma\p_s\bm{X})$.
Let
\begin{equation*}
\bm{u}(\bm{x})=\wh{\mc{S}}[\p_s(\sigma\p_sX)](\bm{x}), \; p(\bm{x})=\mc{P}[\p_s(\sigma\p_sX)][\bm{x}].
\end{equation*}
The above $\bm{u}$ and $p$, together with $\sigma$, solves the tension determination problem in the sense of Definition \ref{STDP}.
In particular, we have (see also \eqref{STDP1}):
\begin{equation}\label{usig0}
\int_{\mbs} \bm{u}(\bm{X}(s))\cdot\p_s(w\p_s\bm{X})ds=0 \text{ for any } w\in C^1(\mbs).
\end{equation}
By Corollary \ref{coroIP}, we have
\begin{equation*}
\frac{1}{2}\int_{\mbr^2\setminus \Gamma} \abs{\nabla \bm{u} + (\nabla\bm{u})^T}^2 d\bm{x}=\int_{\Gamma} \bm{u}(\bm{X}(s)) \cdot \p_{s}\paren{\sigma\p_{s}\bm{X}} ds=0
\end{equation*}
where we used \eqref{usig0} in the last equality.
Noting that $\bm{u}$ is smooth in $\mathbb{R}^2\backslash \Gamma$, we have
\begin{align*}
\nabla\bm{u} + (\nabla\bm{u})^T = 0
\end{align*}
and therefore $\bm{u}$ must be a rigid rotation. Since $\bm{u} \to 0$ as $|\bm{x}|\to\infty$, we conclude that $\bm{u}=0$ in the exterior region $\Omega_2$.
Using Lemma \ref{classical}, $\bm{u}$ is continuous across $\Gamma$, and thus must be identically equal to $0$ in the interior region $\Omega_1$ as well. 
Since $\bm{u}$ and $p$ satisfy the Stokes equation \eqref{e:stokes}, we have
\begin{align*}
\nabla p = 0 \text{ for } \mathbb{R}^2\backslash \Gamma.
\end{align*}
Thus, $p$ is constant within $\Omega_1$ and $\Omega_2$. Let $\diff p := p|_{\Omega_1} - p|_{\Omega_2}$. Inserting this into \eqref{e:stress_jump} with $\bm{F}=0$, we have
\begin{equation}\label{diffp}
\diff p \bm{n} = \p_{s}(\sigma\bm{\tau}) = (\p_{s}\sigma)\bm{\tau} + \sigma \p_{s}\bm{\tau}.%\label{e:tension_kernel}
\end{equation}
Since $\bm{\tau}\cdot \p_{s}\bm{\tau}=0$, we have
\begin{equation}\label{sigmaconst}
\p_{s}\sigma = 0
\end{equation}
which implies that $\sigma$ is constant on $\Gamma$. Equating the normal components implies
\begin{equation}\label{sigkappa}
\diff p = \bm{n}\cdot \sigma \p_{s}\bm{\tau}=-\sigma \kappa(s),
\end{equation}
where $\kappa$ is the curvature.
If $\Gamma$ is not a circle, $\kappa(s)$ is not a constant. Thus, $\sigma=0$.
%Thus, the above is solvable if and only if $\sigma = \diff p = 0$. Therefore, if $\bm{X}$ is not a circle, 
%\begin{align*}
%\mc{L}\sigma = 0
%\end{align*}
%has the unique solution $\sigma = 0$. Therefore, by the Fredholm Alternative theorem, $\sigma\in C^{1, \gamma}(\Gamma)$ is unique.
%
%Next, $\bm{u}(s)=\mc{S}\left[\bm{F}+\p_{s}\paren{\sigma\p_{s}\bm{X}}\right](s)$.
%Since $\bm{F}+\p_{s}\paren{\sigma\p_{s}\bm{X}}\in C^{0, \gamma}(\Gamma)$, by \eqref{e:single_layer_potential} and previous propositions, $\pd{s}\bm{u}\in C^{0, \gamma}(\Gamma)$.
%By the smoothness of the single layer potentials of $\bm{u}$ and $p$, $\bm{u}, p\in C^{\infty}(\mbr^2\setminus\Gamma)$ and $p$ is unique up to a constant.

\end{proof}
%On the other hand, if $\Gamma$ is a circle, then, $\kappa(s)$ is constant and any $\diff p$, $\sigma$ with $\diff p / \sigma = \kappa$ is a solution.
%Therefore, we have to consider the case $\Gamma$ is a circle as a special case.
Let us now turn to the case when $\Gamma$ is a circle.
\begin{proof}[Proof of Theorem \ref{t: well-posed_noncircle} when $\Gamma$ is a circle]
We must solve \eqref{LsigQF}. To do so,
we first prove that, for $\bm{g}\in C^0(\mbs)$, we have
\begin{equation}\label{Qzero}
\int_{\mbs}\mc{Q}[\bm{g}]ds=0.%-\int_{\mbs} \p_{ss}\bm{X}\cdot \mc{S}[\bm{g}]ds
\end{equation}
Indeed, 
\begin{equation*}
\int_{\mbs}\mc{Q}[\bm{g}]ds=\int_{\mbs}\bm{\tau}\cdot \p_s\mc{S}[\bm{g}]ds=-\int_{\mbs} \p_s^2\bm{X}\cdot \mc{S}[\bm{g}]ds=\int_\mbs\bm{n}\cdot \mc{S}[\bm{g}]ds,
\end{equation*}
where we used the fact that $\Gamma$ is a unit circle, and thus, $\p_s^2\bm{X}=-\bm{n}$. Let $\bm{u}(\bm{x})=\wh{\mc{S}}[\bm{g}](\bm{x})$.
Then, 
\begin{equation}\label{nSg}
\int_\mbs\bm{n}\cdot \mc{S}[\bm{g}]ds=\int_\mbs \bm{n}\cdot \bm{u}(\bm{X}(s))ds=\int_{\Omega_1} \nabla \cdot \bm{u}d\bm{x}=0,
\end{equation}
where we used the divergence theorem and the fact that $\bm{u}$ is divergence free. We thus have \eqref{Qzero}.

Define the following function space with zero average
\begin{equation}\label{barC}
\bar{C}^{k,\gamma}(\mbs)=\lbrace w\in C^{k,\gamma}(\mbs) | \int_{\mbs} w ds=0\rbrace, \quad k=0,1,2,\cdots, \gamma\in (0,1).
\end{equation}
From \eqref{Qzero}, we see that $\mc{Q}[\bm{F}]\in \bar{C}^{0,\gamma}(\mbs)$.
When $\sigma\in \bar{C}^{1,\gamma}(\mbs)\subset C^{1,\gamma}(\mbs)$, clearly, $\mc{L}\sigma\in C^{0,\gamma}(\mbs)$. Since $\mc{L}\sigma=\mc{Q}[\p_s(\sigma\p_s\bm{X})]$, 
we see from \eqref{Qzero} that in fact, $\mc{L}\sigma\in \bar{C}^{0,\gamma}(\mbs)$.
We may thus regard \eqref{LsigQF} as an equation for $\sigma\in \bar{C}^{1,\gamma}(\mbs)$ with right hand side in $\bar{C}^{0,\gamma}(\mbs)$.
The question of well-posedness thus reduces to the question of invertibility of $\mc{L}$ as a map from $\bar{C}^{1,\gamma}(\mbs)$ to $\bar{C}^{0,\gamma}(\mbs)$.

Note that the operator $\mc{H}\p_s$ maps $\bar{C}^{1,\gamma}(\mbs)$ to $\bar{C}^{0,\gamma}(\mbs)$.
We may thus use exactly the same argument as in the non-circle case to show that $\mc{L}$ is invertible if and only if its kernel is trivial.
Thus, suppose $\mc{L}\sigma=0$. Again, using exactly the same argument as in the non-circle case, we deduce \eqref{diffp}, from which we 
see that $\sigma$ must be a constant, as in \eqref{sigmaconst}. However, since $\kappa=1$ (does not depend on $s$) for a circle, we cannot conclude that $\sigma=0$ as in \eqref{sigkappa}.
The kernel of $\mc{L}$ thus consists of constant $\sigma$. Since we have restricted $\mc{L}$ to act on $\bar{C}^{1,\gamma}(\mbs)$, this implies that $\sigma=0$.

For $\bm{F}$ satisfying the assumptions of the theorem statement, we thus see that \eqref{LsigQF} has a solution $\sigma_0\in \bar{C}^{1,\gamma}(\mbs)$. If we take any 
general solution $\sigma\in C^{1,\gamma}(\mbs)$, we have
\begin{equation*}
\mc{L}(\sigma-\sigma_0)=0.
\end{equation*}
The above argument shows that $\sigma-\sigma_0=c$ for some constant $c$.
\end{proof}
For future use, we collect some results that we proved above.
\begin{lemma}\label{r:Lcirc}
When $\Gamma$ is a circle, $\mc{S}[\p_s\bm{\tau}]=0$. The opearator $\mc{L}$ has an eigenvalue of $0$ with the constant functions as eigenfunctions. 
Moreover, $\mc{L}$ is an invertible operator from $\bar{C}^{1,\gamma}(\mathbb{S}^1)$ to $\bar{C}^\gamma(\mathbb{S}^1)$, $0<\gamma<1$.
\end{lemma}
\begin{proof}
The first statement follows from \eqref{nSg}. Indeed we have
\begin{equation*}
\int_{\mbs} \bm{n}\cdot \mc{S}[\bm{g}]ds=\int_{\mbs} \bm{g}\cdot \mc{S}[\bm{n}]ds=0,
\end{equation*}
where we used the symmetry of $\mc{S}$. Since $\bm{g}$ is arbitrary, we see that $\mc{S}[\bm{n}]=\mc{S}[\p_s\bm{\tau}]=0$.
The rest of the statements were already shown in the proof of the Theorem above.
\end{proof}

\section{Behavior of $\mc{L}$ when $\Gamma$ is close to a circle}\label{property_L}
\subsection{Operator $\mc{L}$ under general non-degenerate parametrization of $\Gamma$}\label{s:Ltheta}
In this subsection, we translate our results for $\mc{L}$ in the previous section to the case when $\bm{X}$ is given an arbitrary non-degenerate parametrization.
Let $\bm{X}(s)$ be a simple $C^2$ curve parametrized by the arclength coordinate $s\in \mathbb{S}_L^1=\mathbb{R}/L\mathbb{Z}$. Let us reparametrize this curve 
by a strictly monotone increasing $C^2$ function $s=\Phi(\theta), \; \theta\in \mathbb{S}^1=\mathbb{R}/2\pi\mathbb{Z}$ 
so that the new parametrization is given by $\wt{\bm{X}}(\theta)=\bm{X}(\Phi(\theta))$. 
Now that our curve is parametrized by $\theta$, the stress jump condition \eqref{e:stress_jump} in the tension determination problem is replaced by
\begin{equation*}
\jump{\paren{\nabla\bm{u} + (\nabla\bm{u})^T - p\mb{I}}\bm{n}}\p_\theta\Phi = \wt{\bm{F}}(\theta) + \p_{\theta}(\wt{\sigma}(\theta) \wt{\bm{\tau}}(\theta)), \; 
\wt{\bm{\tau}}=\frac{\p_\theta \wt{\bm{X}}}{\abs{\p_\theta \wt{\bm{X}}}} \quad \text{on } \Gamma.
\end{equation*}
where $\wt{\bm{F}}=\bm{F}(\Phi(\theta))\p_\theta\Phi$ and $\wt{\sigma}(\theta)=\sigma(\Phi(\theta))$. Using the steps identical to those that led to equation \eqref{e:L_defn}, we see that the equation satisfied by $\wt{\sigma}$ is given by
\begin{equation}\label{defLtheta}
\begin{split}
\mc{L}_\theta \wt{\sigma}&=-\mc{Q}_\theta[\wt{\bm{F}}], \; \mc{L}_\theta\wt{\sigma}=\mc{Q}_\theta[\p_\theta(\wt{\sigma}\wt{\bm{\tau}})], \; \mc{Q}_\theta[\wt{\bm{F}}]=\wt{\bm{\tau}}\cdot\p_\theta \mc{S}_\theta [\wt{\bm{F}}], \\
\mc{S}_\theta [\bm{g}]&=\int_{\mathbb{S}^1} G(\wt{\bm{X}}(\theta)-\wt{\bm{X}}(\theta'))\bm{g}(\theta')d\theta'.
\end{split}
\end{equation}
Let $\mc{L}_s$ be the operator $\mc{L}$ in \eqref{e:L_defn} (except here, the length of the curve is not necessarily normalized to $2\pi$).
It is easily seen that the map $\mc{L}_\theta$ and $\mc{L}_s$ have the following relationship. For a function $w(s), \; s\in \mathbb{S}_L^1$, define the map:
\begin{equation*}
(\Phi^* w)(\theta)=w(\Phi(\theta))
\end{equation*}
and likewise for $(\Phi^{-1})^*$. We have
\begin{equation*}
\mc{L}_\theta g(\theta)= (\p_\theta \Phi)((\mc{L}_sg(\Phi^{-1}(s)))(\Phi(\theta)))=(\p_\theta \Phi)\paren{\Phi^*\circ \mc{L}_s \circ (\Phi^{-1})^* g}(\theta)
\end{equation*}
Recall from Proposition \ref{p:L_mapping_prop} that $\mc{L}_s$ is a bounded operator from $C^{1,\gamma}(\mathbb{S}_L^1)$ to $C^\gamma(\mathbb{S}_L^1)$ where $0<\gamma<1$
(Proposition \ref{p:L_mapping_prop} proves this when length $L=2\pi$, but it is clear that this statement remains true for arbitrary $L$). 
Note also that $\Phi^*$ is an isomorphism from $C^{\alpha}(\mathbb{S}_L^1)$ to $C^\alpha(\mathbb{S}^1)$ for any $0\leq \alpha\leq 2$.
The map $(\Phi^{-1})^{*}$ is simply the inverse map of $\Phi^*$. 
Multiplication by $\p_\theta \Phi$ is an isomorphism from $C^\alpha(\mathbb{S}^1)$ to itself as long as $0\leq \alpha\leq 1$. 
The above expression thus implies that $\mc{L}_\theta$ is a bounded operator from $C^{1,\gamma}(\mathbb{S}^1)$ to $C^\gamma(\mathbb{S}^1)$ for $0<\gamma<1$.
Furthermore, we see that $\mc{L}_\theta$ is invertible if and only if $\mc{L}_s$ is invertible. 
We have the following proposition. It is convenient to introduce the $L^2$ inner product:
\begin{equation*}
\dual{f}{g}=\int_{\mathbb{S}^1} f(\theta)g(\theta)d\theta,\quad \dual{\bm{f}}{\bm{g}}=\int_{\mathbb{S}^1} \bm{f}(\theta)\cdot \bm{g}(\theta)d\theta.
\end{equation*}
% where $\bar{g}$ is the complex conjugate. 
We also note that the function spaces $\bar{C}^{k,\gamma}(\mathbb{S}^1)$ was defined in \eqref{barC}.
\begin{proposition}\label{p:Ltheta}
Let $\Gamma$ be a simple closed curve in $\mathbb{R}^2$ such that its arclength parametrization $\bm{X}(s)$ is a $C^2$ function.
Let $\theta$ be an alternate $C^2$ nondegenerate parametrization such that  $s=\Phi(\theta)$, $\p_\theta\Phi>0$.
The operator $\mc{L}_\theta$ defined in \eqref{defLtheta} is a bounded operator from $C^{1,\gamma}(\mathbb{S}^1)$ to $C^{\gamma}(\mathbb{S}^1)$ where $0<\gamma<1$. 
The spectrum of $\mc{L}_\theta$ consist of discrete and real eigenvalues.
Furthermore we have the following.
\begin{enumerate}
\item If $\Gamma$ is not a circle, all eigenvalues are negative. 
\item If $\Gamma$ is a circle, the eigenvalues are negative except for a simple eigenvalue at $0$. 
\end{enumerate}
\end{proposition}
\begin{proof}
We have already proved the first statement. When $\Gamma$ is not a circle, we know that $\mc{L}_\theta$ is invertible, and thus, $\mc{L}_\theta^{-1}$
exists and it is a bounded operator from $C^{\gamma}(\mathbb{S}^1)$ to $C^{1,\gamma}(\mathbb{S}^1)$. Thus, $\mc{L}_\theta^{-1}$ is a compact operator 
on $\mathbb{C}^\gamma(\mathbb{S}^1)$ with trivial kernel, and thus, the spectrum is discrete and consist of eigenvalues. 
Thus, the spectrum of $\mc{L}_\theta$ consists of eigenvalues, and they are discrete and non-zero.
To prove positivity, we note the following. Let $\sigma$ and $\mu$ be $C^{1,\gamma}(\mathbb{S}^1)$ functions that are real valued.
Then,
\begin{equation}\label{Lthetasym}
\begin{split}
\dual{\mu}{\mc{L}_\theta \sigma}&=\int_{\mathbb{S}^1}\mu(\theta) \wt{\bm{\tau}}(\theta)\cdot\p_\theta G(\bm{X}(\theta)-\bm{X}(\theta'))\paren{\sigma(\theta')\wt{\bm{\tau}}(\theta')}d\theta'd\theta\\
&=-\int_{\mathbb{S}^1}\p_\theta\paren{\mu(\theta)\wt{\bm{\tau}}(\theta)}\cdot G(\bm{X}(\theta)-\bm{X}(\theta'))\paren{\sigma(\theta')\wt{\bm{\tau}}(\theta')}d\theta'd\theta\\
&=\dual{\mc{L}_\theta \mu}{\sigma},
\end{split}
\end{equation}
where we integrated by parts in the second equality and used the symmetry of $G(\bm{X}(\theta)-\bm{X}(\theta'))$ and the Fubini's theorem in the last equality.
Substituting $\mu=\sigma$ into the above expression, we have
\begin{equation}
\begin{split}\label{sigmaLthetasigma}
\dual{\sigma}{\mc{L}_\theta \sigma}&=-\int_{\mathbb{S}^1}\p_\theta\paren{\sigma\wt{\bm{\tau}}}\cdot \mc{S}_\theta[\p_\theta(\sigma \wt{\bm{\tau}})]d\theta.
\end{split}
\end{equation}
As in \eqref{e:u_soln}, define
\begin{equation*}
\bm{u}(\bm{x})= \wh{\mc{S}}_\theta[\p_\theta(\sigma\wt{\bm{\tau}})](\bm{x}):= \int_{\mbs} G(\bm{x} - \bm{X}(\theta))\p_\theta(\sigma \wt{\bm{\tau}})d\theta.
\end{equation*}
Using Corollary \ref{coroIP} with $\Gamma$ parametrized by $\theta$, we find that:
\begin{equation*}
\int_{\mathbb{S}^1}\p_\theta\paren{\sigma\wt{\bm{\tau}}}\cdot \mc{S}_\theta[\p_\theta(\sigma \wt{\bm{\tau}})]d\theta=\frac{1}{2}\int_{\mathbb{R}^2\backslash \Gamma} \abs{\nabla \bm{u}+\paren{\nabla \bm{u}}^T}^2 d\bm{x}\geq 0.
\end{equation*}
Combining the above with \eqref{sigmaLthetasigma}, we have
\begin{equation*}
\dual{\sigma}{\mc{L}_\theta\sigma}\leq 0.
\end{equation*}
The symmetry \eqref{Lthetasym} with the semi-negativity above immediately shows that all eigenvalues must be non-positive. Since the eigenvalues of $\mc{L}_\theta$ are non-zero, 
they must be negative.

When $\Gamma$ is a circle we note that $\mc{L}_s$ is invertible as an operator from $\bar{C}^{1,\gamma}(\mathbb{S}_L^1)$ to $\bar{C}^{\gamma}(\mathbb{S}_L^1)$ as observed in 
proof of Theorem \ref{t: well-posed_noncircle} when $\Gamma$ is a circle (see the discussion following \eqref{barC} and Lemma \ref{r:Lcirc}). 
Using this fact, we can prove our assertion in the same way as in the non-circle case. We omit the details.
\end{proof}
\begin{remark}
The above proof shows that $\mc{L}_\theta$ is a symmetric negative semidefinite operator on $C^{1,\gamma}(\mathbb{S}^1)$, which is a dense subset of $L^2(\mathbb{S}^1)$. 
We thus see that $\mc{L}_\theta$ has a Friedrichs extension as a self-adjoint operator on $L^2(\mathbb{S}^1)$. We will not be making use of this fact in what follows.
\end{remark}

\subsection{Eigenvalue problem for $\mc{L}$ near a circle}\label{s:ep}
We shall henceforth consider the case when $\bm{X}$ is close to a unit circle. Suppose $\Gamma$ is close to a unit circle in the $C^2$ sense.
Then, it is clear that $\bm{X}$ can be written as
\begin{align*}
    \bm{X}(\theta)=\bm{X}_c(\theta)+\bm{Y}(\theta)=(1+g(\theta))\bm{X}_c(\theta), \; \bm{X}_c(\theta)=\begin{pmatrix} \cos(\theta)\\ \sin(\theta) \end{pmatrix}
\end{align*}
where $g(\theta)$ is a $C^2$ function. 
This is simply the polar coordinate representation of the curve $\Gamma$. 
In the above and henceforth, we drop the $\wt{\cdot}$ in Section \ref{s:Ltheta} when referring to quantities parametrized in the $\theta$ coordinate.
We will discuss the propoerties of $\mc{L}_\theta$ when $\bm{X}$ is given by \eqref{Xveps}, which we reproduce here:
\begin{equation*}
    \bm{X}=\bm{X}_\veps=\bm{X}_c+\varepsilon\bm{Y}=\paren{1+\veps g}\bm{X}_c.
\end{equation*}
We shall henceforth drop the subscript $\theta$ when we refer to $\mc{L}_\theta, \mc{S}_\theta$ and $\mc{Q}_\theta$ of \eqref{defLtheta}.
Instead, we will let $\mc{L}_\veps, \mc{S}_\veps, \mc{Q}_\veps$ denote the respective operators when $\bm{X}=\bm{X}_\veps$.
Let us write out $\mc{L}_\veps$
\begin{equation}\label{Lveps}
\begin{split}
    \mc{L}_\varepsilon\sigma
    &=\bm{\tau}_\varepsilon\cdot \pd{\theta}\mc{S}_\varepsilon\left[\pd{\theta}(\sigma\bm{\tau}_\varepsilon)\right]\\
    &=\bm{\tau}_\veps\cdot \frac{1}{4\pi}\pd{\theta}\int_\mbs (G_L(\Delta\bm{X}_\varepsilon)\mathbb{I}+G_T(\Delta \bm{X}_\veps))\pd{\theta'}(\sigma'\bm{\tau}'_\varepsilon)d\theta',\;
    \bm{\tau}_\veps=\frac{\pd{\theta}\bm{X}_\varepsilon}{\abs{\pd{\theta}\bm{X}_\varepsilon}}.
    \end{split}
\end{equation}
where $\sigma'=\sigma(\theta')$ and similarly for other symbols with a prime.

We are interested in the behavior of $\mc{L}_\veps$ when $\veps$ is close to $0$. 
For this purpose, we first examine the regularity of $\mc{L}_\veps$ with respect to $\veps$.
% We fix some notation.
For Banach spaces $U$ and $V$, let us denote by $\mc{B}(U,V)$ the set of Banach space of bounded operators from $U$ to $V$ 
topologized by the uniform operator topology. For a Banach space $W$ and an open interval $I\subset \mathbb{R}$ we shall use the notation $C^n (I; W)$ to denote the 
set of $n$-times continuously differentiable maps from $I$ to $W$. The smooth maps from $I$ to $W$ will be denoted by $C^\infty(I;W)$.
\begin{proposition}\label{p:Lvepsreg}
Suppose $\bm{Y}(\theta)$ is a $C^2$ function. Then, there is an $\veps_0>0$ such that $\mc{L}_\veps$ is a smooth map from $\veps\in (-\veps_0,\veps_0)$ to $\mc{B}(C^{1,\gamma}(\mathbb{S}^1),C^{\gamma}(\mathbb{S}^1))$. In other words, $\mc{L}_\veps\in C^\infty((-\epsilon_0,\epsilon_0);\mc{B}(C^{1,\gamma}(\mathbb{S}^1),C^{\gamma}(\mathbb{S}^1)))$. 
\end{proposition}
\begin{proof}
Pick a $M>0$ and $m>0$ such that
\begin{equation*}
\begin{split}
\norm{\bm{Y}}_{C^2}\leq M, &\quad \sup_{\abs{\veps}<\veps_0}\norm{\bm{X}_\epsilon}_{C^2}\leq M,\\
\inf_{\abs{\veps}<\veps_0}\starnorm{\bm{X}_\epsilon}\geq m, &\quad
\starnorm{\bm{X}}=\inf_{\theta\neq \theta'} \frac{\abs{\bm{X}(\theta)-\bm{X}(\theta')}}{\abs{\theta-\theta'}}.
\end{split}
\end{equation*}
This is always possible by taking $\veps_0>0$ small enough.
It is clear that
\begin{equation}\label{tauveps}
    \bm{\tau}_\varepsilon\in C^\infty \paren{ (-\varepsilon_0,\varepsilon_0) ;C^{1}\paren{\mbs}}.
\end{equation}
Next, let us consider $\pd{\theta}\mc{S}_\varepsilon\left[\cdot\right]$.
Take derivatives of $G_L\paren{\Delta\bm{X}_\varepsilon}$ with respect to $\varepsilon$.
\begin{align*}
    \frac{d}{d\varepsilon}G_L\paren{\Delta\bm{X}_\varepsilon}=-\frac{\Delta\bm{X}_\varepsilon\cdot\Delta\bm{Y}}{\abs{\Delta\bm{X}_\varepsilon}^2}.
\end{align*}
Moreover, for all $\alpha_0,\alpha_1, \beta_0, \beta_1\in \mb{N}\cup \{0\}$,
\begin{align*}
    &\frac{d}{d\varepsilon}\frac{\paren{\Delta X_{\varepsilon,1}}^{\alpha_0}\paren{\Delta X_{\varepsilon,2}}^{\beta_0}\paren{\Delta Y_{1}}^{\alpha_1}\paren{\Delta Y_{2}}^{\beta_1}}{\abs{\Delta\bm{X}_\varepsilon}^{\alpha_0+\alpha_1+\beta_0+\beta_1}}\\
    =&-\paren{\alpha_0+\alpha_1+\beta_0+\beta_1}\frac{\paren{\Delta X_{\varepsilon,1}}^{\alpha_0}\paren{\Delta X_{\varepsilon,2}}^{\beta_0}\paren{\Delta Y_{1}}^{\alpha_1}\paren{\Delta Y_{2}}^{\beta_1}\left(\Delta X_{\varepsilon,1}\Delta Y_{1}+\Delta X_{\varepsilon,2}\Delta Y_{2}\right)}{\abs{\Delta\bm{X}_\varepsilon}^{\alpha_0+\alpha_1+\beta_0+\beta_1+2}}\\
     &+\frac{\paren{\Delta X_{\varepsilon,1}}^{\alpha_0-1}\paren{\Delta X_{\varepsilon,2}}^{\beta_0-1}\paren{\Delta Y_{1}}^{\alpha_1}\paren{\Delta Y_{2}}^{\beta_1}\left(\alpha_0 \Delta X_{\varepsilon,2}\Delta Y_{1}+\beta_0 \Delta X_{\varepsilon,1}\Delta Y_{2} \right)}{\abs{\Delta\bm{X}_\varepsilon}^{\alpha_0+\alpha_1+\beta_0+\beta_1}}.
\end{align*}
Therefore, $\frac{d^n}{d\varepsilon^n}G_L\paren{\Delta\bm{X}_\varepsilon}$ is the sum of the terms of the form:
\begin{align*}
    C_{\alpha_0,\alpha_1, \beta_0, \beta_1}\frac{\paren{\Delta X_{\varepsilon,1}}^{\alpha_0}\paren{\Delta X_{\varepsilon,2}}^{\beta_0}\paren{\Delta Y_{1}}^{\alpha_1}\paren{\Delta Y_{2}}^{\beta_1}}{\abs{\Delta\bm{X}_\varepsilon}^{\alpha_0+\alpha_1+\beta_0+\beta_1}},
\end{align*}
where $\alpha_0+\alpha_1+\beta_0+\beta_1\leq 2n$. Likewise, $\frac{d^n}{d\varepsilon^n}G_T\paren{\Delta\bm{X}_\varepsilon}$ is the sum of the terms of
\begin{align*}
    C_{\alpha_0,\alpha_1, \beta_0, \beta_1}\frac{\paren{\Delta X_{\varepsilon,1}}^{\alpha_0}\paren{\Delta X_{\varepsilon,2}}^{\beta_0}\paren{\Delta Y_{1}}^{\alpha_1}\paren{\Delta Y_{2}}^{\beta_1}}{\abs{\Delta\bm{X}_\varepsilon}^{\alpha_0+\alpha_1+\beta_0+\beta_1}},
\end{align*}
where $\alpha_0+\alpha_1+\beta_0+\beta_1\leq 2n+2$.
Hence, by an argument similar to the proof of Proposition \ref{p:kernelest01} (see also \cite[Lemma 2.2]{MRS2019}), for all $\abs{\varepsilon}<\varepsilon_0, n\in \mathbb{N}$, $\frac{d^{n+1}}{d\varepsilon^{n+1}}\pd{\theta}\mc{S}_\varepsilon\left[\cdot\right]$ exists, and
\begin{align*}
    \norm{\frac{d^{n+1}}{d\varepsilon^{n+1}}\pd{\theta}\mc{S}_\varepsilon\left[\cdot\right]}_{\mc{B}\paren{C^{0},C^{0,\alpha}}}
    \leq& \sum_{\alpha_0+\alpha_1+\beta_0+\beta_1\leq 2n+4} C_{\alpha_0,\alpha_1, \beta_0, \beta_1} \frac{\chnorm{\bm{X}_\varepsilon}{2}^2\chnorm{\bm{Y}}{2}^{\alpha_1+\beta_1}}{\starnorm{\bm{X}_\varepsilon}^{\alpha_1+\beta_1+2}}\\
    \leq & C\frac{M^{2n+6}}{m^{2n+6}}
\end{align*}
where $C$ depends on $n,\alpha$. We thus see that:
\begin{align*}
    \frac{d^{n}}{d\varepsilon^{n}}\pd{\theta}\mc{S}_\varepsilon\left[\cdot\right]\in C\paren{(-\varepsilon_0,\varepsilon_0); \mc{B}\paren{C^{0}\paren{\mbs},C^{0,\alpha}\paren{\mbs}}}.
\end{align*}
Using the above, \eqref{tauveps} and the expression for $\mc{L}_\veps$ given in \eqref{Lveps}, we see that 
\begin{equation*}
    \mc{L}_\varepsilon\in C^n\paren{(-\varepsilon_0,\varepsilon_0); \mc{B}\paren{C^{1,\gamma}\paren{\mbs},C^{0,\gamma}\paren{\mbs}}} \text{ for } n\in \mathbb{N}.
\end{equation*}
\end{proof}

We now consider the eigenvalue problem:
\begin{equation}\label{eigprob}
\mc{L}_\veps \sigma_\veps=\lambda_\veps \sigma_\veps, \quad \int_{\mathbb{S}^1}\sigma_\veps^2 d\theta=2\pi, \;  \text{ where } \lambda_0=0, \; \sigma_0=1.
\end{equation}
Note that, when $\veps=0$, the polar coordinate and the arclength coordinate coincide. Thus, 
from Lemma \ref{r:Lcirc}, we see that $\lambda_0=0$ is indeed an eigenvalue and the constant function $\sigma_0=1$ is an eigenvector. 
The above can be seen as an eigenvalue perturbation problem for small values of $\veps$. We now establish this solvability.
\begin{proposition}\label{p:eprob_reg}
There is a $\veps_1>0$ such that \eqref{eigprob} has a solution for $\abs{\veps}<\veps_1$, where $\lambda_\veps$ is smooth in $\veps$ and $\sigma_\veps$
is smooth in $\veps$ with values in $C^{1,\gamma}(\mathbb{S}^1), 0<\gamma<1$.
\end{proposition}
\begin{proof}
Let %$F: \mb{R}\times \paren{\mb{R}\times C^{1,\gamma}\paren{\mbs}}\rightarrow \mb{R}\times C^{0,\gamma}\paren{\mbs}$
\begin{align*}
    F\paren{\sigma,\lambda,\veps}=\begin{pmatrix} F_1 \\ F_2 \end{pmatrix}=
    \begin{pmatrix}
    \mc{L}_\varepsilon\sigma-\lambda\sigma\\
    \int_\mbs \sigma^2 d\theta-2\pi
    \end{pmatrix}.
\end{align*}
$F$ maps $(\sigma,\lambda,\veps)\in C^{1,\gamma}(\mathbb{S}^1)\times \mathbb{R}\times \mathbb{R}$ to $C^{\gamma}(\mathbb{S}^1)\times \mathbb{R}$.
Given Proposition \ref{p:Lvepsreg}, 
$F \in C^n(C^{1,\gamma}\paren{\mbs}\times\mathbb{R}\times (-\veps_0,\veps_0);C^{0,\gamma}\paren{\mbs}\times \mathbb{R}), n\in \mathbb{N}$. 
We check the invertibility of the Fr\'echet derivative of the above with respect to $\sigma$ and $\lambda$ at 
$\lambda=\lambda_0, \; \sigma=\sigma_0, \veps=0$. 
This derivative, which we denote by $DF(\sigma_0,\lambda_0,0)\in \mc{B}(C^{1,\gamma}(\mathbb{S}^1)\times \mathbb{R}; C^{\gamma}(\mathbb{S}^1)\times \mathbb{R})$ is given by
\begin{equation*}
DF(\sigma_0,\lambda_0,0)\begin{pmatrix} w\\ \mu \end{pmatrix}= \begin{pmatrix} \mc{L}_0 w- \mu \\ 2\int_{\mathbb{S}^1} w d\theta\end{pmatrix}, \; 
\begin{pmatrix} w\\ \mu \end{pmatrix}\in C^{1,\gamma}(\mathbb{S}^1)\times \mathbb{R}.
\end{equation*}
This linear operator is one-to-one and onto. To show that this is a one-to-one map, let us solve
\begin{equation*}
\mc{L}_0 w- \mu=0, \; \int_{\mathbb{S}^1} w d\theta=0.
\end{equation*}
From the first equation, we have
\begin{equation*}
\dual{1}{\mc{L}_0 w-\mu}=\dual{\mc{L}_0 1}{w}-2\pi \mu=-2\pi \mu=0,
\end{equation*}
where we used the symmetry of $\mc{L}_0$ (see \eqref{Lthetasym}) and $\mc{L}_0 1=0$. Thus, $\mu=0$. Lemma \ref{r:Lcirc} immediately shows that $w=0$.
To show that $DF(\sigma_0,\lambda_0,0)$ is onto, let us solve
\begin{equation}\label{DFeqn}
\mc{L}_0 w- \mu=f, \; 2\int_{\mathbb{S}^1} w d\theta=\nu, \; f\in C^{\gamma}(\mathbb{S}^1), \; \nu\in \mathbb{R}.
\end{equation}
We let
\begin{equation*}
\mu=-\frac{1}{2\pi}\dual{f}{1}, \; w=\mc{L}_0^{-1}\bar{f}+\frac{\nu}{4\pi}, \; \bar{f}=f-\frac{1}{2\pi}\dual{f}{1},
\end{equation*}
it is easily checked that this satisfies \eqref{DFeqn}. Note here that $\bar{f}\in \bar{C}^{\gamma}(\mathbb{S})$ and thus $\mc{L}_0^{-1}\bar{f}$ is well-defined 
by Lemma \ref{r:Lcirc}.

An application of the implicit function theorem yields the desired result.
\end{proof}

We have the following corollary, which shows that the behavior of $\lambda_\veps$ determines the near singularity of $\mc{L}_\veps$ when $\veps$ is small.
\begin{corollary}\label{c:Lvepsinv}
Suppose $\bm{X}_\veps$ is not a circle for $\epsilon\neq 0$. Then, there is an $\veps_2>0$ such that
\begin{equation*}
\norm{\mc{L}_\veps^{-1}}_{\mc{B}(C^{1,\gamma}(\mathbb{S}^1);C^\gamma(\mathbb{S}^1))}\leq C_1+\frac{C_2}{\abs{\lambda_\epsilon}} \text{ for } 0<\abs{\veps}\leq \veps_2
\end{equation*}
where the constants $C_1$ and $C_2$ do not depend on $\veps$ and $\lambda_\veps$ is the solution to \eqref{eigprob}.
\end{corollary}
\begin{proof}
Let $\lambda_\veps, \sigma_\veps$ be as in \eqref{eigprob}, whose existence and smooth dependence on $\veps$ is guaranteed by the previous proposition.
Define the following projection operator
\begin{equation*}
\mc{P}_\epsilon w= \frac{1}{2\pi} \dual{w}{\sigma_\veps} \sigma_{\veps}.
\end{equation*}
This is clearly a bounded operator on $C^{\gamma}(\mathbb{S}^1)$ as well as on $C^{1,\gamma}(\mathbb{S}^1)$. Now, define the operator
\begin{equation*}
\mc{N}_\veps w=\mc{L}_\veps (1-\mc{P}_\veps) w-\mc{P}_\veps w.
\end{equation*}
Clearly, $\mc{N}_\veps\in C^{\infty}((-\epsilon_1,\epsilon_1); \mc{B}(C^{1,\gamma}(\mathbb{S}^1);C^{\gamma}(\mathbb{S}^1)))$. Let us examine $\mc{N}_0$
\begin{equation*}
\mc{N}_0 w=\mc{L}_0 (1-\mc{P}_0) w-\mc{P}_0 w.
\end{equation*}
It is clear that this operator is invertible. Indeed, $(1-\mc{P}_0)w\in \bar{C}^{1,\gamma}(\mathbb{S}^1)$ if $w\in C^{1,\gamma}(\mathbb{S}^1)$, and thus we may 
use Lemma \ref{r:Lcirc}. Since $\mc{N}_\epsilon$ varies smoothly with $\veps$, $\mc{N}_\veps$ is invertible for $\abs{\veps}\leq \veps_2$ for some $\veps_2>0$.
Using this operator, we may write $\mc{L}_\veps$ as
\begin{equation*}
\mc{L}_\veps w=\mc{N}_\veps (1-\mc{P}_\veps) w+ \lambda_\veps \mc{P}_\veps w. 
\end{equation*}
From this, it is immediate that 
\begin{equation*}
\mc{L}_\veps^{-1} w=\mc{N}_\veps^{-1}(1-\mc{P}_\veps) w +\frac{1}{\lambda_\veps}\mc{P}_\veps w, \text{ for } 0<\abs{\veps}\leq \veps_2,
\end{equation*}
where we used the fact that $\lambda_\veps\neq 0$ given our assumption that $\bm{X}_\veps$ is not a circle for $\veps\neq 0$.
By taking the $C^{\gamma}(\mathbb{S}^1)$ norm on both sides of the above, we obtain the desired result.
\end{proof}

\subsection{Computation of $\lambda_2$}\label{s:lambda2}

Here, we explicitly compute the solution to \eqref{eigprob} in a power series expansion up to order $2$.
Let:
\begin{equation}\label{lambdaexp}
\lambda_\veps=\lambda_0+\lambda_1\veps+\lambda_2 \veps^2+\cdots.
\end{equation}
This power series expansion is justified since $\lambda_\veps$ is a smooth function of $\veps$ as shown in Proposition \ref{p:eprob_reg}.
We know that $\lambda_0=0$. We also know from Proposition \ref{p:Ltheta} that $\lambda_\veps\leq 0$. Thus, we immediately see that 
\begin{equation}\label{lam1zero}
\lambda_1=0.
\end{equation}
The first potentially non-trivial term in the expansion is thus $\lambda_2$.
Our goal in this subsection is to establish the last item in Theorem \ref{t:lambdaveps}.
We now solve \eqref{eigprob} in powers of $\veps$.
Let us expand $\lambda_\veps$ as in \eqref{lambdaexp}, and similarly expand $\mc{L}_\veps$ and $\sigma_\veps$. Substituting this expression into \eqref{eigprob}, we have:
\begin{equation*}
\begin{split}
&(\mc{L}_0+\veps \mc{L}_1+\veps^2 \mc{L}_2+\cdots)(\sigma_0+\veps \sigma_1+\veps^2\sigma_2+\cdots)\\
=&(\lambda_0+\veps \lambda_1+\veps^2\lambda_2+\cdots)(\sigma_0+\veps \sigma_1+\veps^2\sigma_2+\cdots),\\
&\int_0^{2\pi} (\sigma_0+\veps \sigma_1+\veps^2\sigma_2+\cdots)^2 d\theta=2\pi.
\end{split}
\end{equation*}
The leading order term in $\veps$ simply gives
\begin{equation*}
\mc{L}_0\sigma_0=\lambda_0\sigma_0, \; \int_0^{2\pi} \sigma_0^2 d\theta=2\pi,
\end{equation*}
for which the solution is $\lambda_0=0, \sigma_0=1$. The first order term in $\veps$ gives
\begin{equation*}
\mc{L}_0\sigma_1=\lambda_1\sigma_0-\mc{L}_1\sigma_0, \; \int_{\mathbb{S}^1} \sigma_1 d\theta=0.
\end{equation*}
By Lemma \ref{r:Lcirc}, the above equation for $\sigma_0$ can be solved uniquely if and only if the right hand side has the zero mean, i.e.
\begin{equation*}
\dual{1}{\lambda_1-\mc{L}_1 1}=0, \text{ and thus } \lambda_1=\frac{1}{2\pi} \dual{1}{\mc{L}_1 1}.
\end{equation*}
where we used $\sigma_0=1$.
We already know that $\lambda_1=0$. Thus, 
\begin{equation*}
\sigma_1=-\mc{L}_0^{-1} \mc{L}_1 1, \quad \dual{1}{\mc{L}_1 1}=0,
\end{equation*}
where $\mc{L}_0^{-1}$ is understood as being acting on $\bar{C}^{\gamma}(\mathbb{S}^1)$.
Let us look at the determination of $\lambda_2$. We have
\begin{equation*}
\mc{L}_0 \sigma_2=-\mc{L}_1\sigma_1-\mc{L}_2\sigma_0+\lambda_2 \sigma_0, \; \int_{\mathbb{S}^1} \sigma_2 d\theta=0.
\end{equation*}
Again, the above is uniquely solvable for $\sigma_2$ if and only if the right hand side has zero mean. From this, we see that
\begin{equation}\label{lambda2exp}
\lambda_2=\frac{1}{2\pi}\paren{\dual{1}{\mc{L}_2 1}+\dual{1}{\mc{L}_1\mc{L}_0^{-1}\mc{L}_1 1}}.
\end{equation}
Our task then is to compute the above expression.

\subsubsection{The operators $\p_\theta\mc{S}_0, \mc{Q}_0$ and $\mc{L}_0$}
%We compute $\p_\theta \mc{S}_0, \mc{Q}_0$ and $\mc{L}_0$. 
We compute the kernel $\p_\theta G_L(\Delta \bm{X}_c)$ and $\p_\theta G_T(\Delta \bm{X}_c)$ (see \eqref{Lveps}).
Note that
\begin{align*}
    \p_\theta\bm{X}_c=
    \begin{pmatrix}
    -\sin{\theta}\\
    \cos{\theta}
    \end{pmatrix},\quad
    \Delta \bm{X}_c=2\sin{\left(\frac{\theta-\theta'}{2}\right)}
    \begin{pmatrix}
    -\sin{\left(\frac{\theta+\theta'}{2}\right)}\\
    \cos{\left(\frac{\theta+\theta'}{2}\right)}
    \end{pmatrix}.
\end{align*}
We thus have
\begin{align*}
    \p_\theta G_L(\Delta \bm{X}_c(\theta))
    &= \frac{-\Delta\bm{X}_c\cdot \p_{\theta}\bm{X}_c}{\abs{\Delta\bm{X}_c}^2} 
    =-\frac{1}{2}\cot{\left(\frac{\theta-\theta'}{2}\right)},\\
    \p_{\theta}G_T(\Delta \bm{X}_c(\theta)) &=
    \p_\theta\paren{\frac{\Delta \bm{X}_c\otimes \Delta \bm{X}_c}{\abs{\Delta \bm{X}_c}^2}}
    =\frac{1}{2}
    \begin{pmatrix}
    \sin{(\theta+\theta')}&-\cos{(\theta+\theta')}\\
    -\cos{(\theta+\theta')}&-\sin{(\theta+\theta')}
    \end{pmatrix}.
\end{align*}
Therefore, we have
\begin{equation*}
\begin{split}
\p_\theta \mc{S}_0[\bm{f}]&=\frac{1}{4\pi}\int_{\mbs} \paren{\p_\theta G_L(\Delta \bm{X}_c(\theta)) \mathbb{I}+ \p_{\theta}G_T(\Delta \bm{X}_c(\theta))}\bm{f}(\theta')d\theta'\\
&=-\frac{1}{4}\mc{H}\bm{f}+\frac{1}{8\pi}\int_\mbs\begin{pmatrix}
    \sin{(\theta+\theta')}&-\cos{(\theta+\theta')}\\
    -\cos{(\theta+\theta')}&-\sin{(\theta+\theta')}
    \end{pmatrix} \bm{f}(\theta')d\theta'.
\end{split}
\end{equation*}
This is all we need in the computation of $\lambda_2$.
We will, however, compute $\mc{Q}_0$ and $\mc{L}_0$ for later reference. 
This computation will allow us to obtain an explicit solution to the tension determination problem when $\Gamma$ is a circle.
From the above, we immediately see that
\begin{align}
\mc{Q}_0[\bm{F}]&=\p_\theta \bm{X}_c\cdot\pd{\theta}\mc{S}[\bm{F}]
       =-\p_\theta \bm{X}_c\cdot\frac{1}{4}\mc{H}\bm{F}-\frac{1}{8\pi} \int_\mbs \bm{X}_c(\theta')\cdot \bm{F}(\theta')d\theta',
\label{e:circle_sigma_operator}
\end{align}
Furthermore, we have
\begin{equation*}
\mc{L}_0\sigma=\mc{Q}_0[\p_\theta(\sigma\p_\theta \bm{X}_c)]=-\p_\theta \bm{X}_c\cdot\frac{1}{4}\mc{H}(\p_\theta\paren{\sigma \p_\theta \bm{X}_c})+\frac{1}{8\pi} \int_\mbs \sigma(\theta') d\theta'
\end{equation*}
We may now solve the equation \eqref{e:rewrite_inextensible} explicitly when $\Gamma$ is a circle. 
The following results can be proved using Lemma \ref{c:hilbert_01} and Lemma \ref{c:hilbert_02}
\begin{lemma}
\begin{equation*}
\mc{L}_0 (\sin(n\theta))=-\frac{n}{4}\sin(n\theta), \quad \mc{L}_0(\cos(n\theta))=-\frac{n}{4}\cos(n\theta), \; n\in \mathbb{N}, \; 
\end{equation*}
The above in fact shows that $\mc{L}_0\sigma=-\frac{1}{4}\mc{H}\p_\theta\sigma$.
\end{lemma}
\begin{lemma}
For $n \geq 2$,
\begin{align*}
    \mc{Q}_0\left[\cos \paren{n\theta}\bm{X}_c\right]&=\mc{Q}_0\left[\sin \paren{n\theta}\bm{X}_c\right]=0,\\
    \mc{Q}_0\left[\cos \paren{n\theta}\pd{\theta}\bm{X}_c\right]=-\frac{1}{4}\sin \paren{n\theta}&, \quad
    \mc{Q}_0\left[\sin \paren{n\theta}\pd{\theta}\bm{X}_c\right]=\frac{1}{4}\cos \paren{n\theta}.
\end{align*}
Moreover,
\begin{align*}
    \mc{Q}_0\left[\cos \theta\bm{X}_c\right]=\frac{1}{8}\cos \theta,&& 
    \mc{Q}_0\left[\sin \theta\bm{X}_c\right]=\frac{1}{8}\sin \theta,\\
    \mc{Q}_0\left[\cos \theta\pd{\theta}\bm{X}_c\right]=-\frac{1}{8}\sin \theta,&&
    \mc{Q}_0\left[\sin \theta\pd{\theta}\bm{X}_c\right]=\frac{1}{8}\cos \theta.
\end{align*}
and
\begin{align*}
    \mc{Q}_0\left[\bm{X}_c\right]=0,\quad
    \mc{Q}_0\left[\pd{\theta}\bm{X}_c\right]=0
\end{align*}
\end{lemma}
From the above two lemmas, the following is immediate.
\begin{proposition}\label{t: computation_circle}
Suppose $\Gamma$ is a circle.
Suppose $\bm{F}$ is given in terms of the following Fourier expansion
\begin{align*}
    \bm{F}=\sum_{n=0}^\infty \paren{ a_n \cos \paren{n\theta}+b_n\sin \paren{n\theta}}\bm{X}_{c}+\sum_{n=0}^\infty \paren{c_n \cos \paren{n\theta}+d_n\sin \paren{n\theta}}\pd{\theta} \bm{X}_{c}.
\end{align*}
Then, the unique solution to \eqref{e:rewrite_inextensible}, satisfying $\int_{\mbs} \sigma d\theta=0$ is given by
\begin{align*}
    \sigma=\frac{a_1+d_1}{2}\cos\theta+\frac{b_1-c_1}{2}\sin\theta+\sum_{n=2}^\infty \paren{\frac{d_n}{n} \cos \paren{n\theta}-\frac{c_n}{n}\sin \paren{n\theta}}.
\end{align*}
\end{proposition}

\subsubsection{Computation of $\mc{L}_1 1$}
Let us expand $\bm{\tau}_\veps$ and $\mc{S}_\veps$ in powers of $\veps$:
\begin{equation}\label{tauSexp}
\bm{\tau}_\veps=\bm{\tau}_0+\veps\bm{\tau}_1+\veps^2\bm{\tau}_2+\cdots, \quad
\mc{S}_\veps=\mc{S}_0+\veps \mc{S}_1+\veps^2\mc{S}_2+\cdots.
\end{equation}
Using \eqref{Lveps} and collecting the first order term in $\veps$, we obtain
\begin{equation*}
\mc{L}_1 1=\bm{\tau}_1 \cdot \p_\theta \mc{S}_0[\p_\theta\bm{\tau}_0]+\bm{\tau}_0 \cdot \p_\theta \mc{S}_1[\p_\theta\bm{\tau}_0]+ \bm{\tau}_0 \cdot \p_\theta \mc{S}_0[\p_\theta \bm{\tau}_1].
\end{equation*}
By Proposition \ref{r:Lcirc}, we know that $\mc{S}_0[\p_\theta \bm{\tau}_0]=0$. Thus, 
\begin{equation}\label{f:L1S0}
\mc{L}_1 1=\bm{\tau}_0 \cdot \p_\theta \mc{S}_1[\p_\theta\bm{\tau}_0]+ \bm{\tau}_0 \cdot \p_\theta \mc{S}_0[\p_\theta \bm{\tau}_1].
\end{equation}
To proceed further, we need the concrete expressions for $\mc{L}_1$.
Let us expand $G_L(\Delta\bm{X}_\varepsilon), G_T(\Delta\bm{X}_\varepsilon)$ in powers of $\veps$
\begin{equation*}
     G_L(\Delta\bm{X}_\varepsilon)= G_{L0}+G_{L1}\varepsilon+G_{L2}\varepsilon^2+\cdots,\quad G_T(\Delta\bm{X}_\varepsilon)= G_{T0}+G_{T1}\varepsilon+G_{T2}\varepsilon^2+\cdots.
\end{equation*}
Using the above, the operators $\mc{S}_i$ in \eqref{tauSexp} can be written as
\begin{align*}
        \mc{S}_i \left[\bm{f}\right]=\frac{1}{4\pi} \int_\mbs \paren{ G_{Li}+G_{Ti}}\bm{f}'d\theta', \; i=0,1,2.
\end{align*}
Now, let us examine the two terms on the right hand side of \eqref{f:L1S0}.
For $\bm{\tau}_0\cdot \pd{\theta}\mc{S}_1\left[\pd{\theta}\bm{\tau}_0\right]$, we have
\begin{align*}
    \bm{\tau}_0\cdot \pd{\theta}\mc{S}_1\left[\pd{\theta}\bm{\tau}_0\right]
    =&\pd{\theta}\paren{\bm{\tau}_0\cdot \mc{S}_1\left[\pd{\theta}\bm{\tau}_0\right]}-\pd{\theta}\bm{\tau}_0\cdot \mc{S}_1\left[\pd{\theta}\bm{\tau}_0\right]\\
    =&-\pd{\theta}\paren{\pd{\theta}\bm{X}_{c}\cdot \mc{S}_1\left[\bm{X}_{c}\right]}-\bm{X}_{c}\cdot \mc{S}_1\left[\bm{X}_{c}\right].
\end{align*}
By \eqref{t:GL01}, \eqref{t:GT01},and \eqref{e:GT01d00},
\begin{align*}
     &4\pi\pd{\theta}\bm{X}_{c}\cdot \mc{S}_1\left[\bm{X}_{c}\right]
    =\pd{\theta} \bm{X}_{c}\cdot\int_\mbs \paren{G_{L1}+G_{T1}}\bm{X}_{c}' d\theta' \\
    =&-\int_\mbs\frac{1}{2}\pd{\theta}\bm{X}_{c}\cdot\Delta\bm{Y}+\frac{\pd{\theta}\bm{X}_{c}\cdot\bm{X}_{c}'}{\abs{\Delta\bm{X}_{c}}^2}\paren{\bm{X}_{c}'+2\Delta\bm{X}_{c}}\cdot\Delta\bm{Y}d\theta'
\end{align*}
Then, by \eqref{t:GL01}, \eqref{t:GT01},and \eqref{e:GT0100},
\begin{align*}
     4\pi\bm{X}_{c}\cdot \mc{S}_1\left[\bm{X}_{c}\right]
    = \bm{X}_{c}\cdot\int_\mbs \paren{G_{L1}+G_{T1}}\bm{X}_{c}' d\theta' 
    =-\int_\mbs\bm{X}_{c}\cdot\bm{X}_{c}'\frac{\Delta\bm{X}_{c}\cdot\Delta \bm{Y}}{\abs{\Delta\bm{X}_{c}}^2}d\theta' .
\end{align*}
For the $\bm{\tau}_0\cdot\pd{\theta}\mc{S}_0\left[\pd{\theta}\bm{\tau}_1\right]$ term in \eqref{f:L1S0}, by \eqref{e:circle_sigma_operator} and \eqref{e:tau011},
\begin{align*}
    \bm{\tau}_0\cdot\pd{\theta}\mc{S}_0\left[\pd{\theta}\bm{\tau}_1\right]
    =-\frac{1}{4}\pd{\theta}\bm{X}_{c}\cdot \mc{H}\pd{\theta} \bm{\tau}_1 -\frac{1}{8\pi}\int_\mbs  \bm{X}_{c}\cdot\pd{\theta'}\bm{\tau}'_1 d\theta'
    =-\frac{1}{4}\pd{\theta}\bm{X}_{c}\cdot \mc{H} \pd{\theta} \bm{\tau}_1 .
\end{align*}
Therefore,
\begin{align}
\begin{split}
        \mc{L}_1 1
    =&\frac{1}{4\pi}\pd{\theta}\int_\mbs\frac{1}{2}\pd{\theta}\bm{X}_{c}\cdot\Delta\bm{Y}+\frac{\pd{\theta}\bm{X}_{c}\cdot\bm{X}_{c}'}{\abs{\Delta\bm{X}_{c}}^2}\paren{\bm{X}_{c}'+2\Delta\bm{X}_{c}}\cdot\Delta\bm{Y}d\theta'\\
     &+\frac{1}{4\pi}\int_\mbs\bm{X}_{c}\cdot\bm{X}_{c}'\frac{\Delta\bm{X}_{c}\cdot\Delta \bm{Y}}{\abs{\Delta\bm{X}_{c}}^2}d\theta'-\frac{1}{4}\pd{\theta}\bm{X}_{c}\cdot \mc{H} \pd{\theta} \bm{\tau}_1. 
\end{split}\label{f:L1S0_simple}
\end{align}

Recall from \eqref{Xveps}, \eqref{gexp} that 
\begin{equation}\label{f:fourier_Y}
    \bm{Y}=g\bm{X}_c,\quad g=g_0 +\sum_{n\geq 1} g_{n1}\cos \paren{n\theta}+g_{n2}\sin \paren{n\theta}.
\end{equation}
In the following computations, we will split $\Delta \bm{Y}$ as
\begin{align*}
    \Delta \bm{Y}   &=g\Delta\bm{X}_c+\Delta g\bm{X}_c'.
\end{align*}
Now, let us compute \eqref{f:L1S0_simple}.
\begin{proposition}\label{t: eig_1}
If $\bm{Y}$ is expanded as the form of \eqref{f:fourier_Y} in $C^{2}(\mbs)$, then
\begin{align*}
    \mc{L}_1 1=0.
\end{align*}
\end{proposition}
We remark that the above result gives an independent proof of the fact that $\lambda_1=0$.
\begin{proof}[Proof of Proposition \ref{t: eig_1}]
We leave some computations in Lemma \ref{c:Y_01}.
In the first integral term of \eqref{f:L1S0_simple},
\begin{align*}
    \int_\mbs \pd{\theta}\bm{X}_c\cdot \Delta \bm{Y}d\theta'
    =-\int_\mbs g' \pd{\theta}\bm{X}_c\cdot\bm{X}_c'd\theta'
    =\pi \left(g_{11}\sin\theta-g_{12} \cos\theta\right).
\end{align*}

\begin{align*}
     &\int_\mbs\frac{\paren{\pd{\theta}\bm{X}_c\cdot\bm{X}_c'}\paren{\Delta\bm{Y}\cdot\bm{X}_c'}}{\abs{\Delta\bm{X}_c}^2}d\theta'
    =\int_\mbs -\frac{1}{2}g \pd{\theta}\bm{X}_c\cdot\bm{X}_c'+\frac{\pd{\theta}\bm{X}_c\cdot\bm{X}_c'}{\abs{\Delta\bm{X}_c}^2}\Delta g d\theta'\\
    =&\pi\sum_{n\geq 1} g_{n1}\sin \paren{n\theta}-g_{n2}\cos \paren{n\theta}.
\end{align*}

\begin{align*}
     2\int_\mbs \frac{\Delta\bm{X}_c\cdot\Delta \bm{Y}}{\abs{\Delta\bm{X}_c}^2}\pd{\theta}\bm{X}_c\cdot \bm{X}_c'd\theta'
    =\int_\mbs  \left(g+g'\right)\pd{\theta}\bm{X}_c\cdot \bm{X}_c'd\theta'
    =\pi \left(g_{12} \cos\theta-g_{11}\sin\theta\right)
\end{align*}
Therefore,
\begin{align*}
     &\frac{1}{4\pi}\int_\mbs\frac{1}{2}\pd{\theta}\bm{X}_c\cdot\Delta\bm{Y}+\frac{\paren{\pd{\theta}\bm{X}_c\cdot\bm{X}_c'}\paren{\Delta\bm{Y}\cdot\bm{X}_c'}}{\abs{\Delta\bm{X}_c}^2}+2\frac{\paren{\Delta\bm{X}_c\cdot\Delta\bm{Y}}}{\abs{\Delta\bm{X}_c}^2}\paren{\pd{\theta}\bm{X}_c\cdot\bm{X}_c'}d\theta'\\
    =&\frac{1}{8}g_{11}\cos\theta+\frac{1}{8}g_{12} \sin\theta+\frac{1}{4}\sum_{n\geq 2} n\paren{ g_{n1}\cos \paren{n\theta}+g_{n2}\sin \paren{n\theta}}.
\end{align*}
In the second integral term of \eqref{f:L1S0_simple}, by Lemma \ref{c:Y_01},
\begin{align*}
     &\frac{1}{4\pi}\int_\mbs\bm{X}_c\cdot\bm{X}_c'\frac{\Delta\bm{X}_c\cdot\Delta \bm{Y}}{\abs{\Delta\bm{X}_c}^2}d\theta'
    =\frac{1}{4\pi}\bm{X}_c\cdot\int_\mbs\frac{1}{2}\paren{g+g'}\bm{X}_c' d\theta'\\
    =&\frac{1}{8}\left(g_{11}\cos\theta+g_{12}\sin\theta\right).
\end{align*}
Finally, in the last term, by Lemma \ref{c:tau_01},
\begin{align*}
    -\frac{1}{4}\pd{\theta}\bm{X}_c\cdot \mc{H} \pd{\theta} \bm{\tau}_1
    =-\frac{1}{4}\sum_{n\geq 1} n g_{n1}\cos \paren{n\theta}+n g_{n2}\sin \paren{n\theta}.
\end{align*}
In conclusion, after summing the all terms, we obtain
\begin{align*}
    \mc{L}_1 1=0.
\end{align*}
\end{proof}

\subsubsection{Computation of $\dual{1}{\mc{L}_2 1}$}
Substituting the expansions of $\mc{S}_\veps$ and $\bm{\tau}_\veps$ given in \eqref{tauSexp} into \eqref{Lveps} and collecting terms of order $\veps^2$, we have
\begin{equation*}
\begin{split}
\mc{L}_2 1=& \bm{\tau}_2\cdot \pd{\theta}\mc{S}_0[\pd{\theta}\bm{\tau}_0]+\bm{\tau}_0\cdot \pd{\theta}\mc{S}_2\left[\pd{\theta}\bm{\tau}_0\right]+\bm{\tau}_0\cdot \pd{\theta}\mc{S}_0\left[\pd{\theta}\bm{\tau}_2\right]\\
                +&\bm{\tau}_1\cdot \pd{\theta}\mc{S}_1\left[\pd{\theta}\bm{\tau}_0\right]+\bm{\tau}_1\cdot \pd{\theta}\mc{S}_0\left[\pd{\theta}\bm{\tau}_1\right]+\bm{\tau}_0\cdot \pd{\theta}\mc{S}_1\left[\pd{\theta}\bm{\tau}_1\right].
                \end{split}
\end{equation*}
Our goal is to compute $\dual{1}{\mc{L}_2 1}$.
By Lemma \ref{r:Lcirc}, we have
\begin{align*}
    \bm{\tau}_2\cdot \pd{\theta}\mc{S}_0\left[\pd{\theta}\bm{\tau}_0\right]=0, \quad
    \dual{\bm{\tau}_0}{\pd{\theta}\mc{S}_0\left[\pd{\theta}\bm{\tau}_2\right]}=-\dual{\p_\theta \bm{\tau}_2}{\mc{S}_0\left[\pd{\theta}\bm{\tau}_0\right]}=0
\end{align*}
Since the kernel of $\mc{S}_1$ is symmetric, we have
\begin{align*}
\dual{\bm{\tau}_1}{\pd{\theta}\mc{S}_1\left[\pd{\theta}\bm{\tau}_0\right]}=
-\dual{\p_\theta\bm{\tau}_1}{\mc{S}_1\left[\pd{\theta}\bm{\tau}_0\right]}=
-\dual{\p_\theta\bm{\tau}_0}{\mc{S}_1\left[\pd{\theta}\bm{\tau}_1\right]}=
\dual{\bm{\tau}_0}{\pd{\theta}\mc{S}_1\left[\pd{\theta}\bm{\tau}_1\right]}.
\end{align*}
Thus, we have
\begin{equation}\label{f:L2S0}
\dual{1}{\mc{L}_2 1}=
\dual{\bm{\tau}_0}{\pd{\theta}\mc{S}_2[\pd{\theta}\bm{\tau}_0]}+2\dual{\bm{\tau}_1}{\pd{\theta}\mc{S}_1[\pd{\theta}\bm{\tau}_0]}+\dual{\bm{\tau}_1}{\pd{\theta}\mc{S}_0[\pd{\theta}\bm{\tau}_1]}
\end{equation}
We will evaluate these three terms in turn.

\begin{lemma}\label{L:020}
\begin{align*}
      \dual{\bm{\tau}_0}{\pd{\theta}\mc{S}_2\left[\pd{\theta}\bm{\tau}_0\right]}
     =-\frac{\pi}{4}\sum_{n\geq 1}n\paren{g_{n1}^2+g_{n2}^2}
\end{align*}
\end{lemma}

\begin{proof}
By \eqref{t:GL02} and \eqref{e:GT0200}, we have
\begin{equation*}
\begin{split}
    &4\pi\int_\mbs \bm{\tau}_0\cdot \pd{\theta}\mc{S}_2\left[\pd{\theta}\bm{\tau}_0\right]d\theta
    =-4\pi\int_\mbs \pd{\theta}\bm{\tau}_0\cdot \mc{S}_2\left[\pd{\theta}\bm{\tau}_0\right]d\theta\\
    =&\quad \int_\mbs\int_\mbs \left[ \frac{1}{2}\frac{\abs{\Delta\bm{Y}}^2}{\abs{\Delta\bm{X}_c}^2}-\frac{\paren{\Delta\bm{X}_c\cdot\Delta\bm{Y}}^2}{\abs{\Delta\bm{X}_c}^4}\right]\bm{X}_c\cdot \bm{X}_c'  d\theta' d\theta\\
     &-\int_\mbs\int_\mbs\frac{\paren{\bm{X}_c\cdot\Delta\bm{Y}}\paren{\Delta\bm{Y}\cdot \bm{X}_c'}}{\abs{\Delta\bm{X}_c}^2}+\frac{1}{4}\abs{\Delta\bm{Y}}^2 d\theta' d\theta
\end{split}\label{e:eig22_020}    
\end{equation*}
For the first term, by Lemma \ref{c:Y_01},
\begin{align*}
     \frac{1}{2}\frac{\abs{\Delta\bm{Y}}^2}{\abs{\Delta\bm{X}_c}^2}-\frac{\paren{\Delta\bm{X}_c\cdot\Delta\bm{Y}}^2}{\abs{\Delta\bm{X}_c}^4}
    =-\frac{1}{4}\left(g^2+{g'}^2\right)+\frac{\abs{\Delta g}^2}{8\sin^2\paren{\frac{\theta-\theta'}{2}}}.
\end{align*}
Then,
\begin{align*}
    \int_\mbs\int_\mbs g^2\bm{X}_c\cdot \bm{X}_c' d\theta' d\theta=
    \int_\mbs\int_\mbs {g'}^2\bm{X}_c\cdot \bm{X}_c' d\theta' d\theta=0.
\end{align*}
Next, by Lemma \ref{t:Toland},
\begin{align*}
     &\int_\mbs\int_\mbs\frac{\abs{\Delta g}^2}{8\sin^2\paren{\frac{\theta-\theta'}{2}}}d\theta' d\theta
    =\pi \int_\mbs g\mc{H}\pd{\theta}gd\theta-\pi \int_\mbs \frac{1}{2}\mc{H}\pd{\theta}g^2 d\theta\\
    =&\pi^2\sum_{n\geq 1} n\left[g_{n1}^2+g_{n2}^2\right].
\end{align*}
Moreover,
\begin{align*}
     \int_\mbs\int_\mbs\frac{\abs{\Delta g}^2}{8\sin^2\paren{\frac{\theta-\theta'}{2}}}\abs{\Delta \bm{X}_c}^2d\theta' d\theta
    =\int_\mbs\int_\mbs\frac{\abs{\Delta g}^2}{2}d\theta' d\theta
    =2\pi^2\sum_{n\geq 1} \paren{g_{n1}^2+g_{n2}^2}.
\end{align*}
Therefore,
\begin{align*}
    &\int_\mbs\int_\mbs\frac{1}{2}\frac{\abs{\Delta\bm{Y}}^2}{\abs{\Delta\bm{X}_c}^2}-\frac{\paren{\Delta\bm{X}_c\cdot\Delta\bm{Y}}^2}{\abs{\Delta\bm{X}_c}^4}d\theta' d\theta
    =\int_\mbs\int_\mbs\frac{\abs{\Delta g}^2}{8\sin^2\paren{\frac{\theta-\theta'}{2}}}\bm{X}_c\cdot \bm{X}_c'd\theta' d\theta\\
    =&\int_\mbs\int_\mbs\frac{\abs{\Delta g}^2}{8\sin^2\paren{\frac{\theta-\theta'}{2}}}\paren{1-\frac{1}{2}\abs{\Delta \bm{X}_c}^2}d\theta' d\theta
    =\pi^2\sum_{n\geq 1}\paren{n-1} \paren{g_{n1}^2+g_{n2}^2}.
\end{align*}

Next, by Lemma \ref{c:Y_01},
\begin{align*}
     \frac{\paren{\bm{X}_c\cdot\Delta\bm{Y}}\paren{\Delta\bm{Y}\cdot \bm{X}_c'}}{\abs{\Delta\bm{X}_c}^2}+\frac{1}{4}\abs{\Delta\bm{Y}}^2
    =\quad-\frac{\abs{\Delta g}^2}{4}+\frac{\abs{\Delta g}^2}{4\sin^2\paren{\frac{\theta-\theta'}{2}}}.
\end{align*}
Again, by the above results,
\begin{align*}
     \int_\mbs\int_\mbs\frac{\paren{\bm{X}_c\cdot\Delta\bm{Y}}\paren{\Delta\bm{Y}\cdot \bm{X}_c'}}{\abs{\Delta\bm{X}_c}^2}+\frac{1}{4}\abs{\Delta\bm{Y}}^2 d\theta' d\theta
    = \pi^2\sum_{n\geq 1}\paren{2n-1} \paren{g_{n1}^2+g_{n2}^2}.
\end{align*}
In conclusion, we obtain
\begin{align*}
     \int_\mbs \bm{\tau}_0\cdot \pd{\theta}\mc{S}_2\left[\pd{\theta}\bm{\tau}_0\right]d\theta
     =-\frac{\pi}{4}\sum_{n\geq 1}n \paren{g_{n1}^2+g_{n2}^2}.
\end{align*}
\end{proof}

\begin{lemma}\label{L:110}
\begin{align*}
     \dual{\bm{\tau}_1}{\pd{\theta}\mc{S}_1\left[\pd{\theta}\bm{\tau}_0\right]}
    =\frac{\pi}{4}\sum_{n\geq 1}n\paren{g_{n1}^2+g_{n2}^2}.
\end{align*}
\end{lemma}

\begin{proof}
By \eqref{e:GL0110} and \eqref{e:GT0110},
\begin{align*}
    &4\pi \int_\mbs \bm{\tau}_1\cdot \pd{\theta}\mc{S}_1\left[\pd{\theta}\bm{\tau}_0\right] d\theta\\
    =&-\int_\mbs\int_\mbs \pd{\theta}\bm{\tau}_1 \cdot\paren{G_{L1}+G_{T1}}\pd{\theta'}\bm{\tau}_0'd\theta' d\theta\\
    =&\int_\mbs\int_\mbs \frac{\paren{\pd{\theta}\bm{\tau}_1\cdot\Delta\bm{X}_c}\bm{X}_c-\paren{\pd{\theta}\bm{\tau}_1\cdot\bm{X}_c'}\Delta\bm{X}_c}{\abs{\Delta\bm{X}_c}^2}\cdot\Delta \bm{Y} d\theta' d\theta-\frac{1}{2}\int_\mbs\int_\mbs\pd{\theta}\bm{\tau}_1\cdot\Delta \bm{Y}d\theta' d\theta\\
\end{align*}
and by \eqref{e:tau011},
\begin{align*}
    &-\frac{1}{2}\int_\mbs\int_\mbs\pd{\theta}\bm{\tau}_1\cdot\Delta \bm{Y}d\theta' d\theta
    =\frac{1}{2}\int_\mbs\int_\mbs\bm{\tau}_1\cdot\pd{\theta}\Delta \bm{Y}d\theta' d\theta\\
    =&\frac{1}{2}\int_\mbs\int_\mbs\paren{\pd{\theta}\bm{Y}\cdot\bm{X}_c}\bm{X}_c\cdot\pd{\theta}\bm{Y}d\theta' d\theta
    =\pi\int_\mbs\abs{\bm{\tau}_1}^2 d\theta
\end{align*}
Therefore,
\begin{align}
\begin{split}
    &4\pi \int_\mbs \bm{\tau}_1\cdot \pd{\theta}\mc{S}_1\left[\pd{\theta}\bm{\tau}_0\right] d\theta\\
    =&\int_\mbs\int_\mbs \frac{\paren{\pd{\theta}\bm{\tau}_1\cdot\Delta\bm{X}_c}\bm{X}_c-\paren{\pd{\theta}\bm{\tau}_1\cdot\bm{X}_c'}\Delta\bm{X}_c}{\abs{\Delta\bm{X}_c}^2}\cdot\Delta \bm{Y} d\theta' d\theta+\pi\int_\mbs\abs{\bm{\tau}_1}^2 d\theta.
\end{split}\label{e:eig22_110}
\end{align}

First, by Lemma \ref{c:tau_01} and Lemma \ref{c:Y_01},
\begin{align*}
     &\frac{\paren{\pd{\theta}\bm{\tau}_1\cdot\Delta\bm{X}_c}\paren{\bm{X}_c\cdot\Delta \bm{Y}}}{\abs{\Delta\bm{X}_c}^2}-\frac{\paren{\pd{\theta}\bm{\tau}_1\cdot\bm{X}_c'}\paren{\Delta\bm{X}_c\cdot\Delta \bm{Y}}}{\abs{\Delta\bm{X}_c}^2}\\
    =&\quad\frac{1}{2}g\pdd{\theta}{2} g-\frac{1}{2}\pdd{\theta}{2} g\paren{g+2g'}\bm{X}_c\cdot \bm{X}_c'
     -\frac{1}{2}\pd{\theta}g\paren{2g+g'}\pd{\theta}\bm{X}_c\cdot\bm{X}_c'\\
     &+\frac{1}{2}\cot\paren{\frac{\theta-\theta'}{2}}\Delta g \pd{\theta}g\bm{X}_c\cdot\bm{X}_c'.
\end{align*}
Then, when we integrate the four above terms,
\begin{align*}
     \int_\mbs\int_\mbs \frac{1}{2}g \pdd{\theta}{2} gd\theta' d\theta
    =-\pi^2\sum_{n\geq 1}n^2\paren{g_{n1}^2+g_{n2}^2}.
\end{align*}
\begin{align*}
     -\int_\mbs\int_\mbs\frac{1}{2}\pdd{\theta}{2} g\paren{g+2g'}\bm{X}_c\cdot \bm{X}_c' d\theta' d\theta
    =\pi^2\paren{g_{11}^2+g_{12}^2}.
\end{align*}
\begin{align*}
     -\int_\mbs\int_\mbs\frac{1}{2}\pd{\theta}g\paren{2g+g'}\pd{\theta}\bm{X}_c\cdot\bm{X}_c' d\theta' d\theta
    =-\frac{\pi^2}{2}\left(g_{11}^2+g_{12}^2\right).
\end{align*}

\begin{align*}
     &\int_\mbs\int_\mbs \frac{1}{2}\cot\paren{\frac{\theta-\theta'}{2}}\Delta g \pd{\theta}g\bm{X}_c\cdot\bm{X}_c' d\theta' d\theta
    =-\pi\int_\mbs \pd{\theta}g\bm{X}_c\cdot\mc{H}\left[g\bm{X}_c\right]d\theta\\
    =&\pi^2\left[\frac{1}{2}\paren{g_{11}^2+g_{12}^2}+\sum_{n\geq 2}n\paren{g_{n1}^2+g_{n2}^2}\right].
\end{align*}
Next, by Lemma \ref{c:tau_01}
\begin{align*}
     \int_\mbs \abs{\bm{\tau}_1}^2d\theta
    =\pi \sum_{n\geq 1}n^2 \paren{  g_{n2}^2+ g_{n1}^2}.
\end{align*}
Therefore,
\begin{align*}
     \int_\mbs \bm{\tau}_1\cdot \pd{\theta}\mc{S}_1\left[\pd{\theta}\bm{\tau}_0\right] d\theta
    =\frac{\pi}{4}\sum_{n\geq 1}n\paren{  g_{n2}^2+ g_{n1}^2}.
\end{align*}
\end{proof}

\begin{lemma}\label{L:101}
\begin{align*}
    \int_\mbs \bm{\tau}_1\cdot \pd{\theta}\mc{S}_0\left[\pd{\theta}\bm{\tau}_1\right] d\theta
    =-\frac{\pi}{4}\sum_{n\geq 1}n^3\left( g_{n1}^2+ g_{n2}^2\right).
\end{align*}
\end{lemma}
\begin{proof}
For $\bm{\tau}_1\cdot \pd{\theta}\mc{S}_0\left[\pd{\theta}\bm{\tau}_1\right]$, similar with the result of \eqref{e:circle_sigma_operator} and by \eqref{e:tau011},
\begin{align*}
     &\bm{\tau}_1\cdot\pd{\theta}\mc{S}_0\left[\pd{\theta}\bm{\tau}_1\right]=\paren{\bm{X}_c\cdot\pd{\theta}\bm{Y}}\bm{X}_c\cdot\pd{\theta}\mc{S}_0\left[\pd{\theta}\bm{\tau}_1\right]\\
    =&\paren{\bm{X}_c\cdot\pd{\theta}\bm{Y}}\left[-\frac{1}{4}\bm{X}_c\cdot  \mc{H}\pd{\theta}\bm{\tau}_1-\frac{1}{8\pi}\int_\mbs  \pd{\theta'}\bm{X}'_c\cdot\pd{\theta'}\bm{\tau}'_1 d\theta'\right]\\
    =&-\frac{1}{4}\bm{\tau}_1\cdot  \mc{H}\pd{\theta}\bm{\tau}_1-\frac{1}{8\pi}\paren{\bm{X}_c\cdot\pd{\theta}\bm{Y}}\int_\mbs  \bm{X}'_c\cdot\pd{\theta'}\bm{Y}' d\theta'.
\end{align*}
For the first term of \eqref{e:eig22_110}, by Lemma \ref{c:tau_01},
\begin{align*}
     \bm{X}_c\cdot\mc{H}\pd{\theta}\bm{\tau}_1
    =\sum_{n\geq 1} n^2\left[g_{n2}\cos \paren{n\theta}- g_{n1}\sin \paren{n\theta}\right],
\end{align*}
so
\begin{align*}
     \int_\mbs\bm{\tau}_1\cdot\mc{H}\pd{\theta}\bm{\tau}_1d\theta
    =\pi\sum_{n\geq 1}n^3\paren{ g_{n1}^2+ g_{n2}^2}.
\end{align*}
Then, for the second term,
\begin{align*}
    \bm{X}'_c\cdot\pd{\theta'}\bm{Y}'=\pd{\theta}g\Longrightarrow
    \int_\mbs  \bm{X}'_c\cdot\pd{\theta'}\bm{Y}' d\theta'=0.
\end{align*}
Therefore,
\begin{align*}
    \int_\mbs \bm{\tau}_1\cdot \pd{\theta}\mc{S}_0\left[\pd{\theta}\bm{\tau}_1\right] d\theta
    =-\frac{\pi}{4}\sum_{n\geq 1}n^3\paren{ g_{n1}^2+ g_{n2}^2}.
\end{align*}
\end{proof}

\begin{proof}[Proof of Theorem \ref{t:lambdaveps}]
The first item was proved in Proposition \ref{p:eprob_reg} and the non-positivity of $\lambda_\veps$ was proved in Proposition \ref{p:Ltheta}.
The second item was proved in Corollary \ref{c:Lvepsinv}.
We prove the last item. In the expansion \eqref{lambdaexp}, we already know that $\lambda_0=\lambda_1=0$ (see \eqref{lam1zero}).
Let us compute $\lambda_2$ using \eqref{lambda2exp}. Since $\mc{L}_1 1=0$ by Proposition \ref{t: eig_1}, the second term in \eqref{lambda2exp} vanishes.
By \eqref{f:L2S0}, we have
\begin{equation*}
\begin{split}
\lambda_2&=\frac{1}{2\pi}\dual{1}{\mc{L}_2 1}\\
&=\frac{1}{2\pi}\paren{\dual{\bm{\tau}_0}{\pd{\theta}\mc{S}_2[\pd{\theta}\bm{\tau}_0]}+2\dual{\bm{\tau}_1}{\pd{\theta}\mc{S}_1[\pd{\theta}\bm{\tau}_0]}+\dual{\bm{\tau}_1}{\pd{\theta}\mc{S}_0[\pd{\theta}\bm{\tau}_1]}}\\
&=-\frac{1}{8}\sum_{n\geq 2} n(n^2-1)\paren{g_{n1}^2+g_{n2}^2}
\end{split}
\end{equation*}
where we used Lemma \ref{L:020}, \ref{L:110} and \ref{L:101}.
\end{proof}

\subsection{Numerical Results}\label{s:numerics}
In this section, we will use the computational boundary integral method to numerically verify our analytical results above.% $\lambda_1, \lambda_2$.
We discretize the equation 
\begin{align*}
     \bm{\tau}\cdot\pd{\theta} \mc{S}\left[\pd{\theta}(\sigma\bm{\tau})\right]=\bm{\tau}\cdot\pd{\theta} \mc{S}\left[\bm{F}\right].
\end{align*}
For the partial derivative with respect to $\theta$ we use fast Fourier transform to compute the discrete derivative.
For the singular integral operator $\mc{S}[\cdot]$, we use the discretization in \cite[Section 3.1]{SL2010}. 
In the case when $\Gamma$ is a circle, the zero mean condition on $\sigma$ is solved together with the above to obtain a unique solution.

We first check our numerical method against the analytical expressions obtained in Proposition \ref{t: computation_circle}. 
%We $\sigma$ and $\bm{u}$ for $\bm{F}$ on $\Gamma$ with low frequency.
Set $\Gamma$ to be a unit circle, $\bm{X}(\theta)=(\cos \theta,\sin \theta)$. For $\bm{F}$, we let
\begin{equation*}
 \bm{F}(\theta)=
    \begin{bmatrix}
    \sin 2\theta+ 4\cos\theta-4\sin\theta\\
    -\cos 2\theta+ 4\sin\theta+4\cos\theta
    \end{bmatrix}
\end{equation*}
Since $\Gamma$ is a unit circle, $\sigma$ can only be found up to an additive constant. Imposing the zero mean condition on $\sigma$, Proposition \ref{t: computation_circle} shows us that
\begin{equation}\label{UX}
\sigma(\theta)=\sin(\theta), \; \bm{U}(\theta)=\bm{u}(\bm{X}(\theta))=\begin{pmatrix} -2\sin(\theta) \\ 2\cos(\theta)\end{pmatrix}
\end{equation}
%The given force $\bm{F}$ on $\Gamma$ and the solutions $\bm{u}, p$ in $\Omega\setminus\Gamma$, $\sigma$ on $\Gamma$ are
%\begin{align*}
%    \bm{u}^+&=
%    \begin{bmatrix}
%    \frac{-2y}{x^2+y^2}\\
%    \frac{2x}{x^2+y^2}
%    \end{bmatrix},
%    &\bm{u}^-=
%    \begin{bmatrix}
%    \frac{-2y}{r^2}\\
%    \frac{2x}{r^2}
%    \end{bmatrix},\\
%    p^+&=\pi r^2,   & p^-=\pi r^2-4L^2,
%\end{align*}
%\begin{align*}
%    \sigma&=r\sin\theta\\
%    \bm{F}&=
%    \begin{bmatrix}
%    r\sin 2\theta+ 4rL^2\cos\theta-\frac{4}{r}\sin\theta\\
%    -r\cos 2\theta+ 4rL^2\sin\theta+\frac{4}{r}\cos\theta
%    \end{bmatrix}
%\end{align*}
%Now, we set $r=1, L=1$. Since $\Gamma$ is a circle, 
The numerical results (Table \ref{tab:err_ unit_circle}) show the accuracy of the computational boundary integral method
and indicates that the method is spectrally accurate.

\begin{table}[htbp]
    \centering
    \begin{tabular}{ccccccccc}
    \hline
     $N$&$\norm{\sigma_h-\sigma}_\infty$&$\norm{\sigma_h-\sigma}_2$&$\norm{\bm{U}_h-\bm{U}}_\infty$&$\norm{\bm{U}_h-\bm{U}}_\infty$\\
     \hline
     32 &5.3290E-15 &1.3114E-15 &5.1833E-15 &1.6195E-15 \\
     64 &6.4392E-15 &2.1441E-15 &5.8937E-15 &2.0663E-15 \\
%     128&7.6605E-15 &3.3306E-15 &7.3368E-15 &3.0437E-15 \\
%     256&9.8809E-15 &6.5288E-15 &9.0382E-15 &5.0520E-15 \\
%     512&1.6209E-14 &1.2601E-14 &1.1193E-14 &9.4445E-15 \\
%    1024&2.7200E-14 &3.1342E-14 &1.4827E-14 &1.6936E-14 \\
%    2048&3.3639E-14 &9.1245E-14 &2.1085E-14 &3.6027E-14 \\
     \hline
    \end{tabular}
    \caption{The errors of $\sigma, \bm{U} $ on the unit circle $\Gamma$ measured in the $L^\infty$ and $L^2$ norms. $\sigma_h$ and $\bm{U}_h$ 
    denote the numerical solution for $\sigma$ and $\bm{U}$ (see \eqref{UX}) respectively and $N$ is the number of discretization points on $\Gamma$.}
    \label{tab:err_ unit_circle}
\end{table}
Next, we compare the numerical and theoretical values of $\lambda_\varepsilon$.
We consider two examples: (1) $\bm{Y}=\cos\theta\bm{X}$ and (2) $\bm{Y}=\frac{1}{2}\paren{1+\cos\paren{ 2\theta}}\bm{X}$.
We numerically compute the leading eigenvalue of the discretization $\mc{L}_\veps$, and compare the resulting value with 
the asymptotic expression \eqref{lambdaest} in Theorem \ref{t:lambdaveps}.
%$\lambda_2$ term of $\lambda_\varepsilon$ now, we will use the values of $\lambda_2$ of our examples to expect the behavior of $\lambda_\varepsilon$.
\quad\\
\textbf{Example 1}\\
$\bm{Y}=\cos\theta\bm{X}$, so $\lambda_2=0$ by Theorem \ref{t:lambdaveps}.
Since $\lambda_\varepsilon\leq 0$, we must have $\lambda_3=0$ and $\lambda_4\leq 0$.
That means $\lambda_\varepsilon=\mc{O}\paren{ \varepsilon^4}$ when $\varepsilon$ is the neighborhood of $0$.
Therefore, for the numerical results of $\lambda_\varepsilon$, we expect
\begin{align*}
    \lim_{\varepsilon\rightarrow 0} \frac{\lambda_\varepsilon}{\varepsilon^2}=0 \mbox{ and } \lim_{\varepsilon\rightarrow 0} \frac{\lambda_\varepsilon}{\varepsilon^4}\leq 0.
\end{align*}
Figure \ref{fig:ex01_1} shows the results of $\lambda_\varepsilon$.
As we can see, $\lambda_\varepsilon$ is very flat in the neighborhood of $0$.
This shows $\abs{\lambda_2}$ is very small.
Next, Figure \ref{fig:ex01_2} shows the numerical values of $\frac{\lambda_\varepsilon}{\varepsilon^2}$ and $\frac{\lambda_\varepsilon}{\varepsilon^4}$.
In the left figure, $\lim_{\varepsilon\rightarrow 0} \frac{\lambda_\varepsilon}{\varepsilon^2}\approx 0$.
In the right figure, $\lim_{\varepsilon\rightarrow 0} \frac{\lambda_\varepsilon}{\varepsilon^4}\approx -0.046875= -\frac{3}{64}<0$.
The numerical results seem to suggest that $\lambda_4=-\frac{3}{64}$.
\begin{figure}[htbp]
    \centering
    \includegraphics[width=1\textwidth]{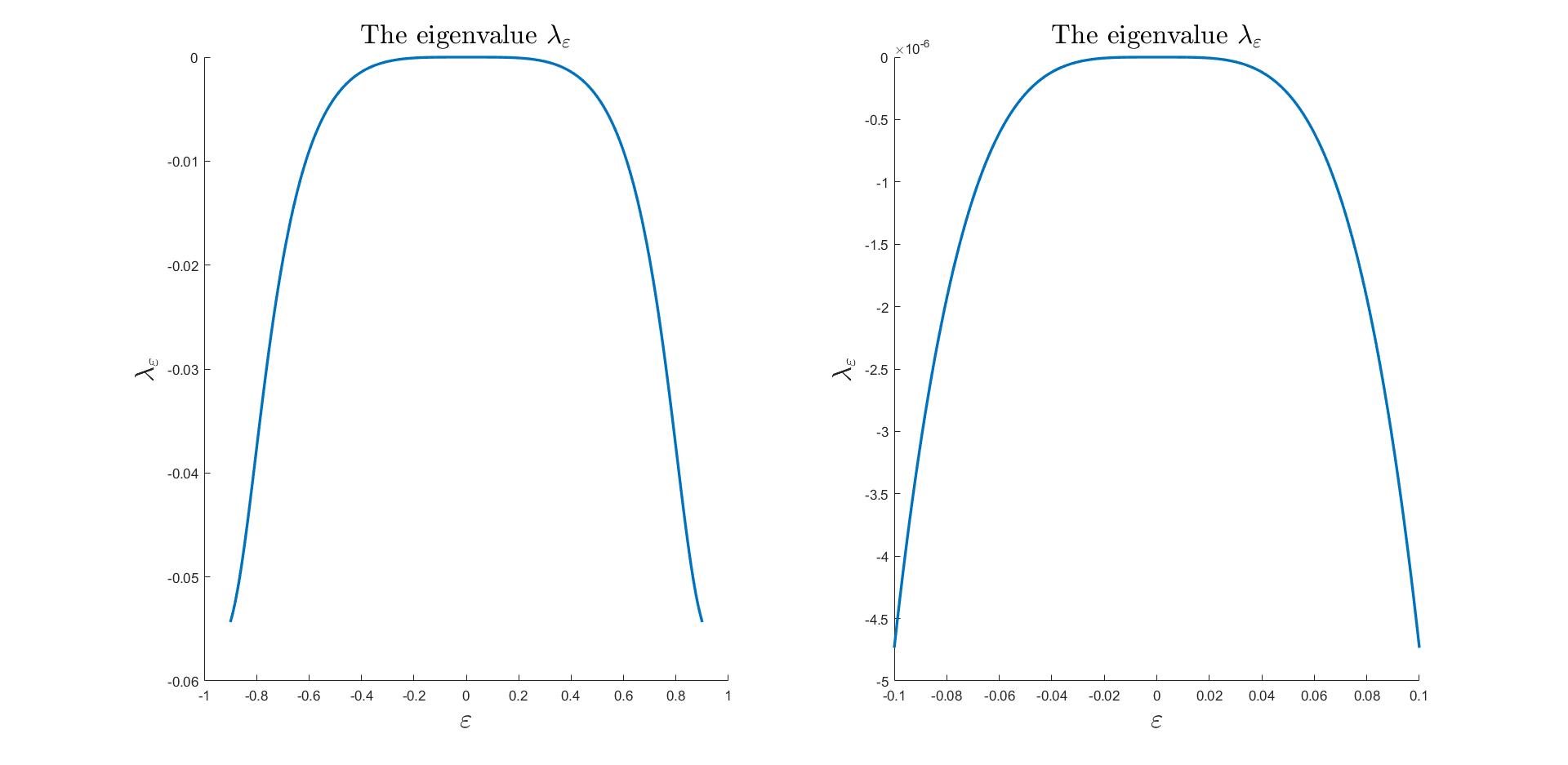}
    \caption{The eigenvalue $\lambda_\varepsilon$ respect to $\varepsilon$ in \textbf{Example 1} : $\bm{X}_\varepsilon\paren{\theta}=\paren{1+\varepsilon\cos{\theta}}\bm{X}\paren{\theta}$}
    \label{fig:ex01_1}
\end{figure}
\begin{figure}[htbp]
    \centering
    \includegraphics[width=1\textwidth]{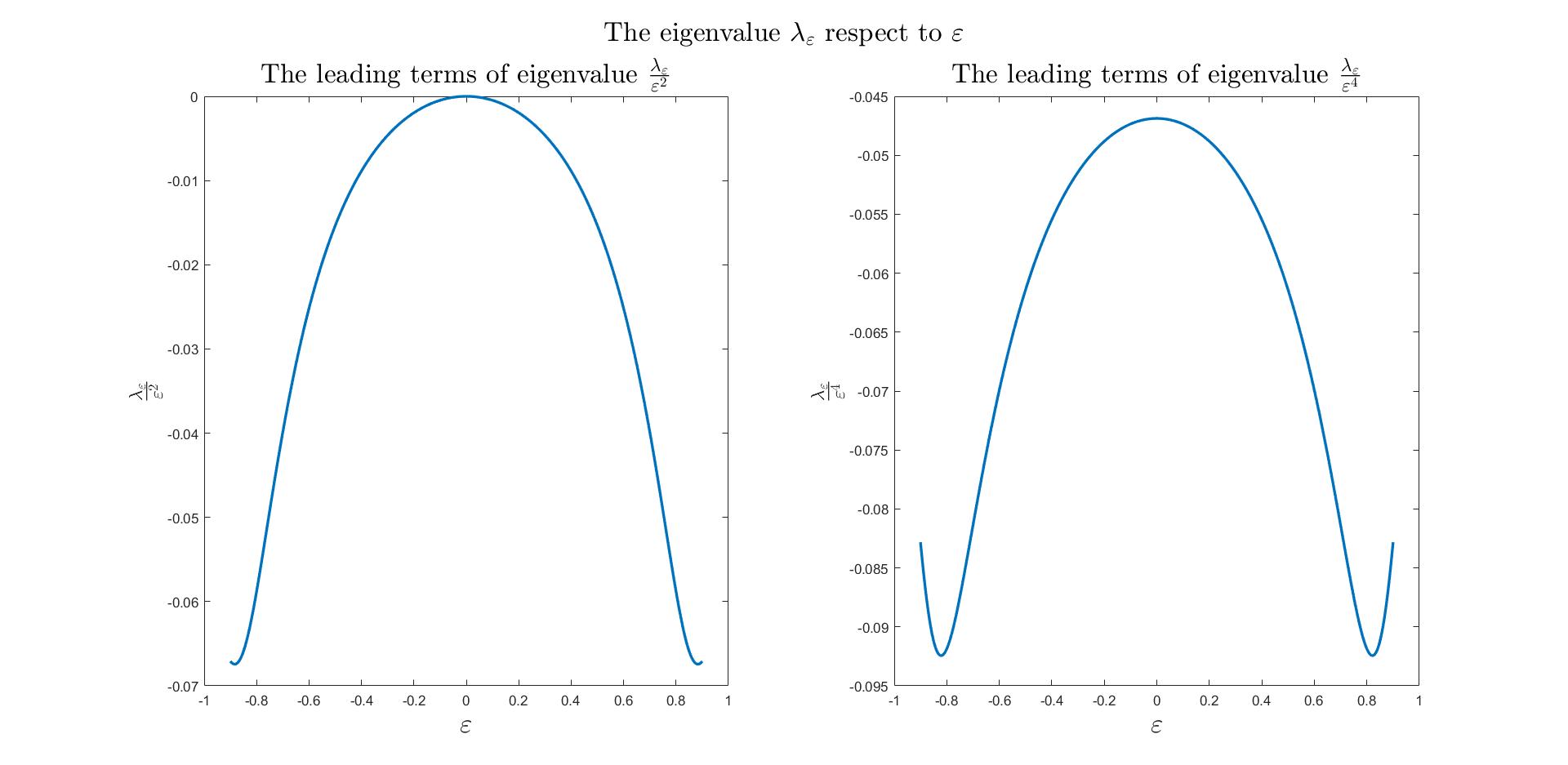}
    \caption{The leading terms of the eigenvalue $\lambda_\varepsilon$ respect to $\varepsilon$ in \textbf{Example 1}. Left: the value $\frac{\lambda_\varepsilon}{\varepsilon^2}$. Right: the value $\frac{\lambda_\varepsilon}{\varepsilon^4}$}
    \label{fig:ex01_2}
\end{figure}
\quad\\
\textbf{Example 2}\\
$\bm{Y}=\frac{1}{2}\paren{1+\cos\paren{ 2\theta}}\bm{X}$, so $\lambda_2=-\frac{3}{16}$.
We thus expect $\lambda_\varepsilon\approx -\frac{3}{16}\varepsilon^2$ in the neighborhood of $0$.
In the left figure of Figure \ref{fig:ex021}, the numerical values of $\lambda_\varepsilon$ matches $\lambda_2\varepsilon^2$ when $\varepsilon$ is near $0$.
The right figure shows that %$\frac{\lambda_\varepsilon}{\varepsilon^2}$, 
\begin{align*}
        \lim_{\varepsilon\rightarrow 0} \frac{\lambda_\varepsilon}{\varepsilon^2}=-\frac{3}{16} \mbox{ and } \frac{\lambda_\varepsilon}{\varepsilon^2}\approx -\frac{3}{16}+\frac{3}{16}\varepsilon.
\end{align*}
Hence, as we can see in the left figure, $\lambda_\varepsilon$ matches $-\frac{3}{16}\varepsilon^2+\frac{3}{16}\varepsilon^3$ well.
\begin{figure}[htbp]
    \centering
    \includegraphics[width=1\textwidth]{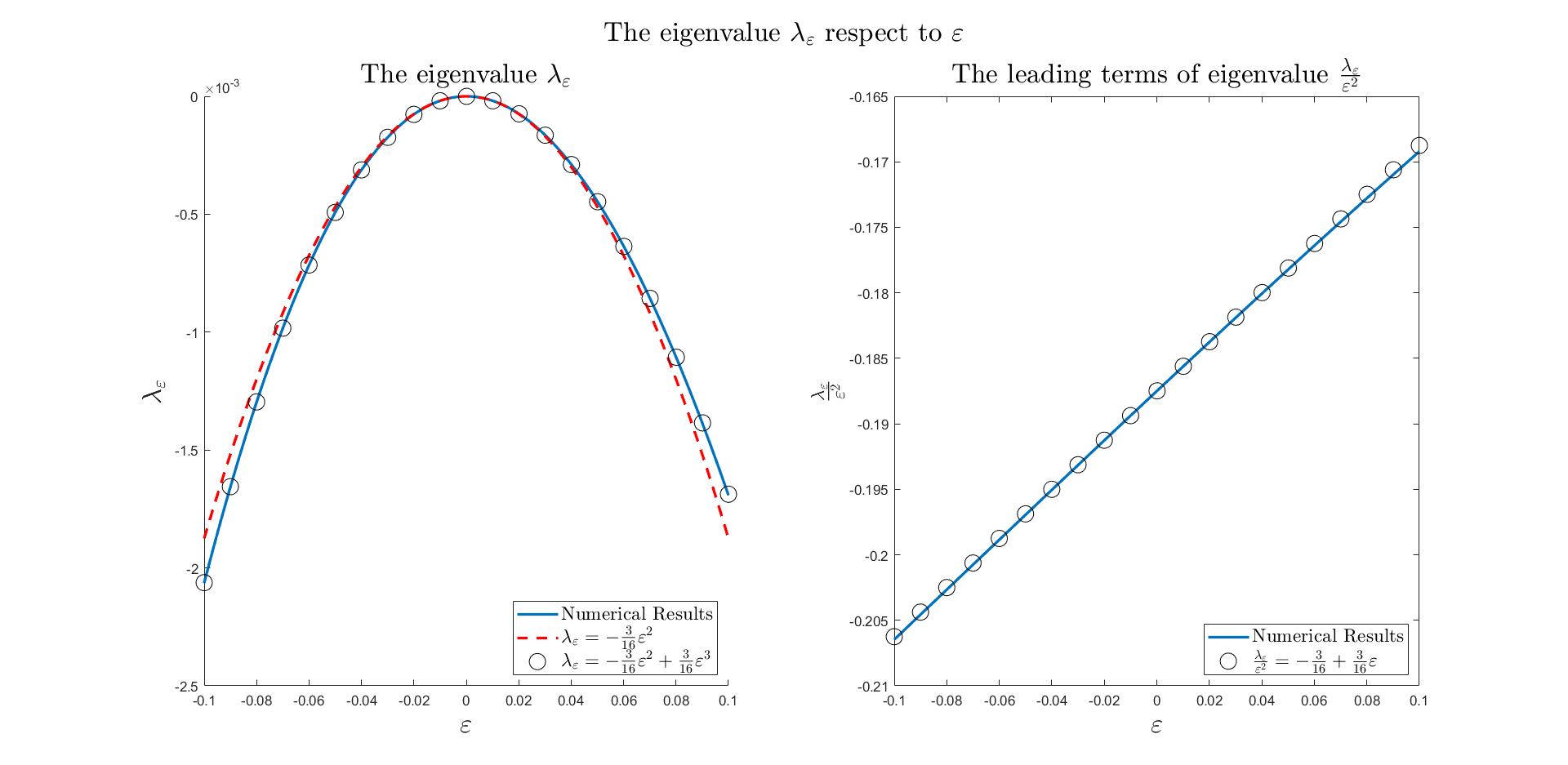}
    \caption{The leading terms of the eigenvalue $\lambda_\varepsilon$ respect to $\varepsilon$ in \textbf{Example 2}. Left: the value $\lambda_\varepsilon$. Right: the value $\frac{\lambda_\varepsilon}{\varepsilon^2}$}
    \label{fig:ex021}
\end{figure}
% \begin{figure}[htbp]
%     \centering
%     \includegraphics[width=1\textwidth]{eigen_ex03_1.jpg}
%     \includegraphics[width=1\textwidth]{eigen_ex03_2_1.jpg}
%     \includegraphics[width=1\textwidth]{eigen_ex03_2_2.jpg}
%     \caption{Caption}
%     \label{fig:ex03}
% \end{figure}
\section{Conclusion and Future Outlook}\label{conclusion}
In this paper, we established the well-posedness of the tension determination problem for a 1D interface in 2D Stokes flow.
In Theorem \ref{t: well-posed_noncircle} we show that the tension $\sigma$ can be determined uniquely if and only if $\Gamma$ is not a circle.
We have also established estimates on the tension $\sigma$ given the force density $\bm{F}$. When $\Gamma$ is close to a circle, 
the tension determination problem becomes increasingly singular. This approach to singularity was studied in detail 
leading to the results of Theorem \ref{t:lambdaveps}.

The static tension determination problem treated here is one component of the dynamic inextensible interface problem 
discussed in Section \ref{motivation}. The analytical understanding gained here is expected to be a key ingredient in the 
analytical study of the dynamic problem. It is also hoped that our study will lead to the development of better numerical 
methods for this and related problems. The analysis in \cite{LSX2019} was indeed motivated by the need to develop 
better numerical methods for the dynamic problem. 
%Our results indicate that the determination of $\sigma$ becomes 
%increasingly singular as $\Gamma$ approaches a circle. This implies that numerical methods for the inextensible 
%interface problem may become increasingly difficult to solve when $\Gamma$ is close to a unit circle.
%It is hoped that our detailed analysis may lead to 

We also point out that the related problem of inextensible filaments in a 3D Stokes fluid, which has 
been studied by many authors as models of swimming filaments \cite{tornberg2004simulating,moreau2018asymptotic,maxian2021integral,maxian2022hydrodynamics,mori2022well}.
It is hoped that the analysis here may also aid in understanding these problems.

\appendix
\section{Layer Potentials}\label{layer_potential_appendix}
In this section, we will discuss the single layer potentials for $\bm{u}$, $p$, and $\Sigma$, which are expressed as
\begin{align*}
\bm{u}(\bm{x}) &= \wh{\mc{S}}[\wh{\bm{F}}](\bm{x}):= \int_{\mbs} G(\bm{x} - \bm{X}(s')) \wh{\bm{F}}(s')ds',\\
p(\bm{x}) &:= \mc{P}[\wh{\bm{F}}](\bm{x})=\int_{\mbs} \Pi(\bm{x} - \bm{X}(s')) \wh{\bm{F}}(s')ds',\\
\Sigma_{ik} \paren{\bm{x}}&:=\mc{T}[\wh{\bm{F}}](\bm{x}):=\int_{\mbs} \Theta_{ijk}(\bm{x} - \bm{X}(s')) \wh{F}_j(s')ds',
\end{align*}
where
\begin{equation*}
G(\bm{r}) = G_L(\bm{r}) + G_T (\bm{r}) := \frac{1}{4\pi}\paren{-\log \abs{\bm{r}}\mb{I} + \frac{\bm{r}\otimes \bm{r}}{\abs{\bm{r}}^2}}, \quad \Pi(\bm{r}) = \frac{1}{2\pi}\frac{\bm{r}^{T}}{\abs{\bm{r}}^2},
\end{equation*}
and
\begin{align*}
    \Theta_{ijk}(\bm{r})=\pd{k}G_{ij}(\bm{r})+\pd{i}G_{kj}(\bm{r})-\Pi_j(\bm{r})\delta_{ik}=-\frac{1}{\pi}\frac{r_i r_j r_k}{\abs{\bm{r}}^4}.
\end{align*}
Since for all $\bm{r}\neq \bm{0}$, $\pd{i} G_{ij}(\bm{r})=0$ and,
\begin{align}
\begin{split}
        \pd{k}\Theta_{ijk}(\bm{r})
    =&\pdd{k}{2}G_{ij}(\bm{r})+\pd{i}\pd{k}G_{kj}(\bm{r})-\pd{k}\Pi_j(\bm{r})\delta_{ik}\\
    =&\Delta G_{ij}(\bm{r}) +\pd{i}\Pi_j(\bm{r})=0,
\end{split}\label{e:stoke_kernel}
\end{align}
we have that $\bm{u}\in C^2\paren{\mbr^2\setminus \Gamma}$, $p\in C^1\paren{\mbr^2\setminus \Gamma}$, and $\Sigma\in C^1\paren{\mbr^2\setminus \Gamma}$ satisfy the Stokes equations in $\mbr^2\setminus \Gamma$, i.e.,
\begin{align*}
    \nabla\cdot \bm{u}=0,\quad -\nabla\cdot \Sigma= -\Delta \bm{u}+\nabla p=0.
\end{align*}
Next, since $G,\Pi,\Theta_{ijk} \in C^\infty \paren{\mbr^2\setminus \{0\}}$, then $\bm{u}, p, \Sigma\in C^\infty\paren{\mbr^2\setminus \Gamma}$.
%Therefore, we only have to prove $\bm{u}$ is continuous at $x\in \Gamma$ for $\bm{u}\in C\paren{\mbr^2}$.
\begin{lemma}\label{l:u_conti}
If $\bm{X}\in C^2\paren{\mbs}$ and $\starnorm{\bm{X}}>0$, then $\bm{u}\in C^\infty\paren{\mbr^2\setminus \Gamma}\bigcap C\paren{\mbr^2}$.
\end{lemma}
\begin{proof}
We have obtained $\bm{u}\in C^\infty\paren{\mbr^2\setminus \Gamma}$, so let us prove $\bm{u}$ is continuous at $\bm{x}\in \Gamma$.
Given $\bm{x}_0 \in \Gamma$ and $\bm{X}\paren{s_0}=\bm{x}_0$,  we define $\bm{Y}\paren{s; s_0}$ as
\begin{align*}
    \bm{Y}\paren{s; s_0}=\mc{R}_{\theta}\paren{\bm{X}-\bm{x}_0} 
\end{align*}
where $\theta=\arg \paren{\pd{s}\bm{X}\paren{s_0}}$ and $\mc{R}_\theta$ is a rotation matrix with angle $\theta$ in a clockwise direction
\begin{align*}
    \mc{R}_\theta=
    \begin{bmatrix}
    \cos \theta & \sin \theta\\
    -\sin \theta&\cos \theta
    \end{bmatrix}.
\end{align*}
Since $\bm{X}\in C^2\paren{\mbs}$ and $\starnorm{\bm{X}}>0$, there exist an $\varepsilon_0$, $0<\varepsilon_0<\frac{1}{2}$ and a function $y_2=f\paren{y_1}$ such that for all $\abs{\bm{Y}\paren{s; s_0}}\leq \varepsilon_0$,
\begin{align*}
    Y_2\paren{s; s_0}=f\paren{Y_1\paren{s; s_0}}.
\end{align*}
Obviously, $f(0)=0$, $\pd{y_1} f(0)=0$ and there exist an $\varepsilon_1$, $0<\varepsilon_1<\varepsilon_0$ and $C>0$ depending on $\varepsilon_0$ such that $\abs{\paren{y_1,f\paren{y_1}}}<\varepsilon_0$ and for all $\abs{y_1}<\varepsilon_1$, $\abs{\pd{y_1} f\paren{y_1}}<C$.
Moreover, $s\paren{y_1}$ exists with
\begin{align*}
    \frac{ds}{dy_1}=\sqrt{1+\paren{\pd{y_1}f\paren{y_1}}^2}.%d y_1.
\end{align*}
% Now, for all $\varepsilon<\varepsilon_1$, we set $B_\varepsilon, \mc{I}_\varepsilon$ as
% \begin{align*}
%     B_\varepsilon=\left\{\bm{x}\in \Gamma\left| \abs{\bm{x}-\bm{x}_0}< \varepsilon \right.\right\}, \quad \mc{I}_\varepsilon=\left\{s\in \mbs\left| \bm{X}\paren{s}\in B_\varepsilon \right.\right\}.
% \end{align*}
Now, for all $\varepsilon<\varepsilon_1$, we set $\mc{I}_\varepsilon$ as
\begin{align*}
     \mc{I}_\varepsilon=\left\{s\in \mbs\left| \abs{\bm{x}-\bm{x}_0}< \varepsilon \right.\right\}.
\end{align*}
Then, for all $\abs{\bm{x}-\bm{x}_0}< \frac{\varepsilon}{2}$, we may split $\bm{u}\paren{\bm{x}}-\bm{u}\paren{\bm{x}_0}$ into
\begin{align*}
\begin{split}
    \bm{u}\paren{\bm{x}}-\bm{u}\paren{\bm{x}_0}
    =   &\quad\int_{\mc{I}_\varepsilon}G\paren{\bm{x}-\bm{X}\paren{s'}}\wh{\bm{F}}(s')ds'-\int_{\mc{I}_\varepsilon}G\paren{\bm{x}_0-\bm{X}\paren{s'}}\wh{\bm{F}}(s')ds'\\
        &+\int_{\mbs\setminus\mc{I}_\varepsilon}\paren{ G\paren{\bm{x}-\bm{X}\paren{s'}}-G\paren{\bm{x}_0-\bm{X}\paren{s'}}}\wh{\bm{F}}(s')ds'.
\end{split}
\end{align*}

Set $\bm{z}=\mc{R}_{\theta}\paren{\bm{x}-\bm{x}_0}, \bm{z}_0=\mc{R}_{\theta}\paren{\bm{x}_0-\bm{x}_0}=0$, we have $\abs{\bm{z}}< \frac{\varepsilon}{2}<\frac{\varepsilon_0}{2}$.
Then, for the first two terms, we obtain  $\abs{Y_1\paren{s; s_0}}<\varepsilon<\varepsilon_1$ on $\mc{I}_\varepsilon$ and 
\begin{align*}
    &\abs{\int_{\mc{I}_\varepsilon}G\paren{\bm{x}-\bm{X}\paren{s'}}\wh{\bm{F}}(s')ds'}
    =\abs{\int_{\mc{I}_\varepsilon}G\paren{\bm{z}-\bm{Y}\paren{s'; s_0}}\wh{\bm{F}}(s')ds'}\\
    \leq &\norm{\wh{\bm{F}}}_{C^0} \int_{\mc{I}_\varepsilon}\abs{ G\paren{\bm{z}-\bm{Y}\paren{s'; s_0}}(s')}ds'\\
    \leq &\norm{\wh{\bm{F}}}_{C^0} \int_{-\varepsilon}^\varepsilon\abs{ G\paren{\bm{z}-\paren{y_1, f\paren{y_1}}}}\sqrt{1+\paren{\pd{y_1}f\paren{y_1}}^2} d y_1
\end{align*}
Since $\abs{\bm{z}}, \abs{\paren{y_1, f\paren{y_1}}}<\frac{1}{2}$,
\begin{align*}
    0<\log \frac{1}{\abs{\bm{z}-\paren{y_1, f\paren{y_1}}}}\leq \log \frac{1}{\abs{z_1-y_1}}.
\end{align*}
Therefore,
\begin{align*}
\begin{split}
        &\int_{-\varepsilon}^\varepsilon\abs{ G\paren{\bm{z}-\paren{y_1, f\paren{y_1}}}}\sqrt{1+\paren{\pd{y_1}f\paren{y_1}}^2} d y_1\\
    \leq&C\int_{-\varepsilon}^\varepsilon\abs{ G_L\paren{\bm{z}-\paren{y_1, f\paren{y_1}}}+G_T\paren{\bm{z}-\paren{y_1, f\paren{y_1}}}} d y_1\\
    \leq&\frac{C}{4\pi}\int_{-\varepsilon}^\varepsilon\log \frac{1}{\abs{z_1-y_1}}+1~~ d y_1
    \leq\frac{C}{2\pi}\varepsilon\paren{3+\log \frac{1}{2\varepsilon}}
\end{split}
\end{align*}
where $C$ depends on $\varepsilon_0$ 
since $\abs{z_1}<\varepsilon$.
Next, for the last term, since $\abs{\bm{x}-\bm{x}_0}< \frac{\varepsilon}{2}$, there exists a $C>0$ such that for all $s'\in\mbs\setminus \mc{I}_\varepsilon$ and $0\leq t\leq 1$
\begin{align*}
\begin{split}
    \abs{\nabla G\paren{t\bm{x}+(1-t)\bm{x}_0-\bm{X}\paren{s'}}}\leq \frac{C}{2\pi}\frac{1}{\varepsilon}.
\end{split}
\end{align*}
Then,
\begin{align*}
\begin{split}
        &\abs{\int_{\mbs\setminus\mc{I}_\varepsilon}\paren{ G\paren{\bm{x}-\bm{X}\paren{s'}}-G\paren{\bm{x}_0-\bm{X}\paren{s'}}}\wh{\bm{F}}(s')ds'}\\
    \leq&\norm{\wh{\bm{F}}}_{C^0}\int_{\mbs\setminus\mc{I}_\varepsilon}\abs{ G\paren{\bm{x}-\bm{X}\paren{s'}}-G\paren{\bm{x}_0-\bm{X}\paren{s'}}}ds'\\
    \leq&\norm{\wh{\bm{F}}}_{C^0}\int_{\mbs\setminus\mc{I}_\varepsilon}\int_0^1\abs{\pd{t} G\paren{t\bm{x}+(1-t)\bm{x}_0-\bm{X}\paren{s'}}}dtds'\\
    \leq&\norm{\wh{\bm{F}}}_{C^0}\int_{\mbs\setminus\mc{I}_\varepsilon}\int_0^1\abs{\nabla G\paren{t\bm{x}+(1-t)\bm{x}_0-\bm{X}\paren{s'}}}\abs{\bm{x}-\bm{x}_0}dtds'\\
    \leq&\frac{C}{2\pi}\norm{\wh{\bm{F}}}_{C^0}\frac{1}{\varepsilon}\abs{\bm{x}-\bm{x}_0}\int_{\mbs\setminus\mc{I}_\varepsilon}\int_0^1dtds'
    \leq C\norm{\wh{\bm{F}}}_{C^0}\frac{1}{\varepsilon}\abs{\bm{x}-\bm{x}_0}.
\end{split}
\end{align*}
Finally, 
\begin{align*}
\begin{split}
        &\lim_{\bm{x}\rightarrow\bm{x}_0}\abs{\bm{u}\paren{\bm{x}}-\bm{u}\paren{\bm{x}_0}}
    \leq\lim_{\bm{x}\rightarrow\bm{x}_0}\frac{C}{\pi}\varepsilon\paren{3+\log \frac{1}{2\varepsilon}}+C\norm{\wh{\bm{F}}}_{C^0}\frac{1}{\varepsilon}\abs{\bm{x}-\bm{x}_0}\\
    \leq&\frac{C}{\pi}\varepsilon\paren{3+\log \frac{1}{2\varepsilon}}.
\end{split}
\end{align*}
Since $0<\varepsilon<\varepsilon_1$ is arbitrary and $C$ only depends on $\varepsilon_0$, taking $\varepsilon\rightarrow 0$, then we obtain that $\bm{u}\paren{\bm{x}}$ is continuous at $\bm{x}_0\in \Gamma$. 
Therefore, $\bm{u}\in C^\infty\paren{\mbr^2\setminus \Gamma}\bigcap C\paren{\mbr^2}$.
\end{proof}

Next, we will prove $\Sigma_{ik}$ satisfies
\begin{align*}
    \jump{\paren{\Sigma\bm{n}}_i}=\jump{\Sigma_{ik}n_k}=\wh{F}_i
\end{align*}
We first set a double layer potential of the flow as
\begin{align*}
    \wh{\mc{D}}[\wh{\bm{F}}](\bm{x}):=\int_{\mbs} K\paren{\bm{X}\paren{s},\bm{x}} \wh{\bm{F}}(s')ds'
\end{align*}
where the kernel $K_{ij}\paren{\bm{y},\bm{x}}:=\Theta_{ijk}\paren{\bm{y}-\bm{x}}n_k\paren{\bm{y}}$.
Then, we have the following result for its integral on $\Gamma$.
\begin{lemma}
\begin{align}
    \int_\Gamma K_{ij}\paren{\bm{y},\bm{x}}ds \paren{\bm{y}}=\left\{
        \begin{matrix}
            -\delta_{ij} & \mbox{ if } \bm{x}\in \Omega_1\\
            0           & \mbox{ if } \bm{x}\in \Omega_2\\
            -\frac{1}{2}\delta_{ij} & \mbox{ if } \bm{x}\in \Gamma
        \end{matrix}
    \right.\label{e:kernel_integral}
\end{align}
\end{lemma}
\begin{proof}
First, given $\bm{x}\in \Omega_2$, since $\abs{\bm{x}-\bm{y}}>0$ for all $\bm{y}\in\Gamma$, by \eqref{e:stoke_kernel},
\begin{align*}
    \int_\Gamma K_{ij}\paren{\bm{y},\bm{x}}ds\paren{\bm{y}}
    =\int_\Gamma \Theta_{ijk}\paren{\bm{y}-\bm{x}}n_k\paren{\bm{y}}ds\paren{\bm{y}}
    =\int_{\Omega_1} \pd{k}\Theta_{ijk}\paren{\bm{y}-\bm{x}}d\bm{y}=0.
\end{align*}
Next, given $\bm{x}\in \Omega_1$, set $B_\varepsilon=\left\{\bm{y}\in \Gamma\left| \abs{\bm{y}-\bm{x}}< \varepsilon \right.\right\}$ with $0<\varepsilon\ll 1$ and on $\partial B_\varepsilon$, $\bm{n}\paren{\theta}=\paren{\cos \theta, \sin \theta}$ and $\bm{y}\paren{\theta}=\bm{x}+\varepsilon\bm{n}\paren{\theta}$.
Then, we have
\begin{align*}
\begin{split}
    0=&\int_{\Omega_1\setminus B_\varepsilon} \pd{k}\Theta_{ijk}\paren{\bm{y}-\bm{x}}d\bm{y}\\
    =&\int_\Gamma \Theta_{ijk}\paren{\bm{y}-\bm{x}}n_k\paren{\bm{y}}ds\paren{\bm{y}}-\int_{\partial B_\varepsilon} \Theta_{ijk}\paren{\bm{y}-\bm{x}}n_k\paren{\bm{y}}ds\paren{\bm{y}}\\
    =&\int_\Gamma K_{ij}\paren{\bm{y},\bm{x}}ds \paren{\bm{y}}-\int_0^{2\pi} \Theta_{ijk}\paren{\varepsilon\bm{n}\paren{\theta}}n_k\paren{\theta}\varepsilon d\theta\\
    =&\int_\Gamma K_{ij}\paren{\bm{x},\bm{y}}ds \paren{\bm{y}}+\frac{1}{\pi}\int_0^{2\pi} \frac{n_i\paren{\theta} n_j\paren{\theta} }{\abs{\bm{n}\paren{\theta}}^4} d\theta.
\end{split}
\end{align*}
If $i\neq j$,
\begin{align*}
\begin{split}
    \int_0^{2\pi} \frac{n_i\paren{\theta} n_j\paren{\theta} }{\abs{\bm{n}\paren{\theta}}^4} d\theta
    =\int_0^{2\pi} \cos\theta \sin{\theta} d\theta=0.
\end{split}
\end{align*}
If $i=j=1\paren{=2}$,
\begin{align*}
    \int_0^{2\pi} \frac{n_i\paren{\theta} n_j\paren{\theta} }{\abs{\bm{n}\paren{\theta}}^4} d\theta
    =\int_0^{2\pi} \cos^2\theta  d\theta\paren{=\int_0^{2\pi} \sin^2\theta  d\theta}=\pi.
\end{align*}
Therefore,
\begin{align*}
    \int_\Gamma K_{ij}\paren{\bm{x},\bm{y}}ds \paren{\bm{y}}=-\frac{1}{\pi}\int_0^{2\pi} \frac{n_i\paren{\theta} n_j\paren{\theta} }{\abs{\bm{n}\paren{\theta}}^4} d\theta=-\delta_{ij}.
\end{align*}
Finally, set $\partial B^1_\varepsilon=\partial B_\varepsilon\bigcap \Omega_1$ and $\partial B^2_\varepsilon=\left\{\bm{y}\in \partial B_\varepsilon\left|\bm{n}\paren{\bm{y}}\cdot\bm{n}\paren{\bm{x}}< 0\right.\right\}$. 
Again, 
\begin{align*}
\begin{split}
    0=&\int_{\Omega_1\setminus B_\varepsilon} \pd{k}\Theta_{ijk}\paren{\bm{y}-\bm{x}}d\bm{y}\\
    =&\int_\Gamma \Theta_{ijk}\paren{\bm{y}-\bm{x}}n_k\paren{\bm{y}}ds\paren{\bm{y}}-\int_{\partial B^1_\varepsilon} \Theta_{ijk}\paren{\bm{y}-\bm{x}}n_k\paren{\bm{y}}ds\paren{\bm{y}},
\end{split}
\end{align*}
and
\begin{align*}
\begin{split}
    &\int_{\partial B^2_\varepsilon} \Theta_{ijk}\paren{\bm{y}-\bm{x}}n_k\paren{\bm{y}}ds\paren{\bm{y}}
    =\int_0^{2\pi} \Theta_{ijk}\paren{\varepsilon\bm{n}\paren{\theta}}n_k\paren{\theta}\varepsilon d\theta\\
    =&-\frac{1}{\pi}\int_{\theta_0}^{\theta_0+\pi} \frac{n_i\paren{\theta} n_j\paren{\theta} }{\abs{\bm{n}\paren{\theta}}^4}d\theta
    =-\frac{1}{2}
\end{split}
\end{align*}
where $\theta_0=\arg\paren{\bm{n}\paren{\bm{x}}}+\frac{\pi}{2}$.
The remaining part is only the difference of integrals between on $\partial B^1_\varepsilon$ and $\partial B^2_\varepsilon$.
Since $\Gamma\in C^2$, $\pdd{s}{2}\bm{X}\paren{s}$ is bounded, the symmetric difference between $\partial B^1_\varepsilon$ and $\partial B^2_\varepsilon$ is contained in 
\begin{align*}
    \partial B^3_\varepsilon=\left\{\bm{y}\in \partial B_\varepsilon\left|\quad\abs{\bm{n}\paren{\bm{y}}\cdot\bm{n}\paren{\bm{x}}}\leq \frac{\norm{\pdd{s}{2}\bm{X}\paren{s}}_{C^0}}{4} \varepsilon^2 \right.\right\}.
\end{align*}
Thus, as $\varepsilon\rightarrow 0$,
\begin{align*}
\begin{split}
        &\abs{\int_{\partial B^1_\varepsilon} \Theta_{ijk}\paren{\bm{y}-\bm{x}}n_k\paren{\bm{y}}ds\paren{\bm{y}}-\int_{\partial B^2_\varepsilon} \Theta_{ijk}\paren{\bm{y}-\bm{x}}n_k\paren{\bm{y}}ds\paren{\bm{y}}}\\
    \leq&\frac{1}{\pi}\int_{\partial B^3_\varepsilon}\abs{ \frac{n_i\paren{\bm{y}} n_j\paren{\bm{y}} }{\varepsilon\abs{\bm{n}\paren{\bm{y}}}^4}}ds\paren{\bm{y}}
    \leq\frac{4}{\pi}\sin^{-1}\paren{\frac{\norm{\pdd{s}{2}\bm{X}\paren{s}}_{C^0}}{4} \varepsilon^2}\longrightarrow 0.
\end{split}
\end{align*}
Therefore,
\begin{align*}
\begin{split}
    &\int_\Gamma \Theta_{ijk}\paren{\bm{y}-\bm{x}}n_k\paren{\bm{y}}ds\paren{\bm{y}}
    =\lim_{\varepsilon\rightarrow 0}\int_{\partial B^1_\varepsilon} \Theta_{ijk}\paren{\bm{y}-\bm{x}}n_k\paren{\bm{y}}ds\paren{\bm{y}}\\
    =&-\frac{1}{2}+\lim_{\varepsilon\rightarrow 0}\paren{\int_{\partial B^1_\varepsilon} -\int_{\partial B^2_\varepsilon}} \Theta_{ijk}\paren{\bm{y}-\bm{x}}n_k\paren{\bm{y}}ds\paren{\bm{y}}=-\frac{1}{2}.
\end{split}
\end{align*}

\end{proof}
Using the fact that $\Gamma$ is $C^2$, from \cite[Lemma 3.15]{F1995}, there exist a constant $C_\Gamma>0$ s.t. for all $\bm{x}, \bm{y}\in \Gamma$,
\begin{align}
    \abs{\paren{\bm{x}-\bm{y}}\cdot \bm{n}\paren{\bm{y}}}\leq C_\Gamma \abs{\bm{x}-\bm{y}}^2\label{e:tech_regularity}.
\end{align}
Thus, for all $\bm{x},\bm{y}\in \Gamma$, 
\begin{align}
    \abs{K_{ij}\paren{\bm{y},\bm{x}}}\leq \frac{C_\Gamma}{\pi},\label{e:kernel_regularity00}
\end{align}
so $\norm{K\paren{\cdot,\bm{x}}}_{L^1\paren{\Gamma}}$ is uniformly bounded on $ \Gamma$.
Next, we will claim $\norm{K\paren{\cdot,\bm{x}}}_{L^1\paren{\Gamma}}$ is uniformly bounded in $\mbr^2\setminus \Gamma$.
\begin{lemma}\label{l:kernel_L1bnd}
There exists a constant $C<\infty$ s.t. $\forall \bm{x}\in \mbr^2\setminus \Gamma$,
\begin{align}
    \int_\Gamma \abs{K_{ij}\paren{\bm{y},\bm{x}}}d s\paren{\bm{y}}\leq C\label{e:kernel_L1bnd}
\end{align}
\end{lemma}

\begin{proof}
Define $\dist \paren{\bm{x}, \Gamma}$ as the distance between point $\bm{x}$ and set $\Gamma$, and then there exist $0<\varepsilon_0<\frac{1}{2C_\Gamma}$ and $C_0>0$ s.t.
(1) for all $\bm{x}$ with $\dist \paren{\bm{x}, \Gamma}<\frac{1}{2}\varepsilon_0$, there exists a unique $\bm{x}_0=\bm{X}\paren{s_0}\in \Gamma$ and $t\in \paren{-\frac{1}{2}\varepsilon_0,\frac{1}{2}\varepsilon_0}$ s.t. $\bm{x}=\bm{x}_0+t \bm{n}\paren{\bm{x}_0}$,
(2) define $r_0\paren{s}=\abs{\bm{X}\paren{s}-\bm{X}\paren{s_0}}$, then $\sgn\paren{s-s_0}\pd{s}r_0\paren{s}>C_0$ for all $0<r_0\paren{s}< \varepsilon_0 $.

We will prove this result for two different cases.

(i) Given $\dist \paren{\bm{x}, \Gamma}\geq\frac{1}{2}\varepsilon_0$, we have $\abs{K_{ij}\paren{\bm{y},\bm{x}}}\leq \frac{1}{2\pi}\frac{1}{\varepsilon_0}$ for all $\bm{y}\in \Gamma$, so
\begin{align*}
    \int_\Gamma \abs{K_{ij}\paren{\bm{y},\bm{x}}}d s\paren{\bm{y}}\leq \frac{1}{2\pi}\frac{1}{\varepsilon_0}\int_\Gamma d s\paren{\bm{y}}\leq C_1\frac{1}{\varepsilon_0}
\end{align*}
where $C_1$ is only depends on $\Gamma$.

(ii) Given $\dist \paren{\bm{x}, \Gamma}<\frac{1}{2}\varepsilon_0$, set  $\bm{x}_0\in \Gamma$ be the unique point s.t. $\bm{x}=\bm{x}_0+t \bm{n}\paren{\bm{x}_0}$ with $t\in \paren{-\frac{1}{2}\varepsilon_0,\frac{1}{2}\varepsilon_0}$, and define $B_{\varepsilon_0}=\left\{\bm{y}\in \Gamma\left| \abs{\bm{y}-\bm{x}_0}< \varepsilon_0 \right.\right\}$.
We split the integral into
\begin{align*}
    \int_\Gamma \abs{K_{ij}\paren{\bm{y},\bm{x}}}d s\paren{\bm{y}}=\int_{B_{\varepsilon_0}} \abs{K_{ij}\paren{\bm{y},\bm{x}}}d s\paren{\bm{y}}+\int_{\Gamma\setminus B_{\varepsilon_0}} \abs{K_{ij}\paren{\bm{y},\bm{x}}}d s\paren{\bm{y}}
\end{align*}
In the second term, for all $\bm{y}\in {\Gamma\setminus B_{\varepsilon_0}}$,
\begin{align*}
    \abs{\bm{y}-\bm{x}}\geq \abs{\bm{y}-\bm{x}_0}-\abs{\bm{x}_0-\bm{x}}\geq \frac{1}{2}\varepsilon_0,
\end{align*}
so
\begin{align*}
    \int_{\Gamma\setminus B_{\varepsilon_0}} \abs{K_{ij}\paren{\bm{y},\bm{x}}}d s\paren{\bm{y}}\leq \frac{1}{2\pi}\frac{1}{\varepsilon_0}\int_{\Gamma\setminus B_{\varepsilon_0}} d s\paren{\bm{y}}\leq C_1\frac{1}{\varepsilon_0}.
\end{align*}
For the first term ,by \eqref{e:tech_regularity}, we obtain
\begin{align*}
\begin{split}
    &\pi \abs{K_{ij}\paren{\bm{y},\bm{x}}}
    =   \frac{\abs{\paren{y_i-x_i}\paren{y_j-x_j}\paren{\bm{y}-\bm{x}}\cdot\bm{n}\paren{\bm{y}}}}{\abs{\bm{y}-\bm{x}}^4}\\
    \leq&  \frac{\abs{\paren{\bm{y}-\bm{x}_0}\cdot\bm{n}\paren{\bm{y}}}+\abs{\paren{\bm{x}_0-\bm{x}}\cdot\bm{n}\paren{\bm{y}}}}{\abs{\bm{y}-\bm{x}}^2}
    \leq  \frac{C_\Gamma\abs{\bm{y}-\bm{x}_0}^2+\abs{\bm{x}_0-\bm{x}}}{\abs{\bm{y}-\bm{x}}^2}
\end{split}
\end{align*}
Moreover,
\begin{align*}
    \abs{\bm{y}-\bm{x}}^2=\abs{\bm{y}-\bm{x}_0}^2+\abs{\bm{x}_0-\bm{x}}^2+2\paren{\bm{y}-\bm{x}_0}\cdot\paren{\bm{x}_0-\bm{x}}
\end{align*}
Since
\begin{align*}
    \bm{x}_0-\bm{x}=t\bm{n}\paren{\bm{x}_0},\quad\mbox{ where }\quad \abs{t}=\abs{\bm{x}_0-\bm{x}}
\end{align*}
by \eqref{e:tech_regularity},
\begin{align*}
    \abs{\paren{\bm{y}-\bm{x}_0}\cdot\paren{\bm{x}_0-\bm{x}}}=\abs{\paren{\bm{y}-\bm{x}_0}\cdot\paren{\bm{x}_0}}\abs{\bm{x}_0-\bm{x}}\leq C_\Gamma\abs{\bm{y}-\bm{x}_0}^2\abs{\bm{x}_0-\bm{x}}.
\end{align*}
Then, $C_\Gamma\abs{\bm{y}-\bm{x}_0}< C_\Gamma\varepsilon_0\leq \frac{1}{2}$, so
\begin{align}\label{e: proj_lower_bnd}
\begin{split}
    \abs{\bm{y}-\bm{x}}^2
    \geq&\abs{\bm{y}-\bm{x}_0}^2+\abs{\bm{x}_0-\bm{x}}^2- \abs{\bm{y}-\bm{x}_0}\abs{\bm{x}_0-\bm{x}}\\
    \geq&\frac{1}{2}\paren{\abs{\bm{y}-\bm{x}_0}^2+\abs{\bm{x}_0-\bm{x}}^2}.
\end{split}  
\end{align}
Therefore,
\begin{align*}
\begin{split}
    \abs{K_{ij}\paren{\bm{y},\bm{x}}}
    \leq& \frac{2}{\pi} \frac{C_\Gamma\abs{\bm{y}-\bm{x}_0}^2+\abs{\bm{x}_0-\bm{x}}}{\abs{\bm{y}-\bm{x}_0}^2+\abs{\bm{x}_0-\bm{x}}^2}\\
    \leq& \frac{2}{\pi} C_\Gamma+\frac{2}{\pi} \frac{\abs{\bm{x}_0-\bm{x}}}{\abs{\bm{y}-\bm{x}_0}^2+\abs{\bm{x}_0-\bm{x}}^2}.
\end{split}
\end{align*}
For the second term, set $r=\abs{\bm{y}-\bm{x}_0}$ and $a=\abs{\bm{x}_0-\bm{x}}$.
Since $\sgn\paren{s-s_0}\pd{s}r_0\paren{s}>C_0$ for all $0<r_0\paren{s}< \varepsilon_0 $,
\begin{align*}
\begin{split}
    \int_{B_{\varepsilon_0}} \frac{\abs{\bm{x}_0-\bm{x}}}{\abs{\bm{y}-\bm{x}_0}^2+\abs{\bm{x}_0-\bm{x}}^2} ds\paren{\bm{y}}
    \leq \frac{2}{C_0} \int_0^{\varepsilon_0} \frac{a}{r^2+a^2}dr
    \leq\frac{2}{C_0} \int_0^{\infty} \frac{1}{r^2+1}dr.
\end{split}
\end{align*}
Hence,
\begin{align*}
\begin{split}
    \int_\Gamma \abs{K_{ij}\paren{\bm{y},\bm{x}}}d s\paren{\bm{y}}
    \leq& C_1\frac{1}{\varepsilon_0}+\frac{2}{\pi} C_\Gamma\int_{B_{\varepsilon_0}}ds\paren{\bm{y}}+\frac{2}{C_0} \int_0^{\infty} \frac{1}{r^2+1}dr\\
    \leq& C_1\frac{1}{\varepsilon_0}+C_2
\end{split}
\end{align*}
Since $\varepsilon_0, C_1, C_2$ only depend on $\Gamma$,
\begin{align*}
    \int_\Gamma \abs{K_{ij}\paren{\bm{y},\bm{x}}}d s\paren{\bm{y}}\leq C
\end{align*}
where $C$ only depends on $\Gamma$.

\end{proof}
\begin{remark}
Obviously, if $\wh{\bm{F}}\paren{s}\in C^0\paren{\mbs}$, $\wh{\mc{D}}[\wh{\bm{F}}](\bm{x})$ is smooth and bounded in $\mbr^2\setminus \Gamma$.
Then, by \eqref{e:kernel_regularity00} and \cite[Proposition 3.12]{F1995}, $\wh{\mc{D}}[\wh{\bm{F}}](\bm{x})$ exists and is bounded on $\Gamma$.
Note that $\wh{\mc{D}}[\wh{\bm{F}}](\bm{x})$ may be discontinuous across $\Gamma$.
\end{remark}
\begin{lemma}\label{l:dlayer_conti}
Given $\Gamma\in C^2$ and $\wh{\bm{F}}\in C\paren{\Gamma}$, if $\wh{\bm{F}}\paren{s_0}=0$ where $s_0\in \mbs$, then $\wh{\mc{D}}[\wh{\bm{F}}]\paren{\bm{x}}=\int_{\mbs} K\paren{\bm{X}\paren{s'},\bm{x}} \wh{\bm{F}}(s')ds'$ is continuous at $\bm{x}_0=\bm{X}\paren{s_0}$
\end{lemma}

\begin{proof}
By \eqref{e:tech_regularity} and \eqref{e:kernel_L1bnd}, there exists $C_0,C_1>0$ s.t.
\begin{align*}
    \int_{\mbs} \abs{K_{ij}\paren{\bm{X}\paren{s'},\bm{X}\paren{s}}}d s&\leq C_0,\quad \forall s \neq \mbs,\\
    \int_{\mbs} \abs{K_{ij}\paren{\bm{X}\paren{s'},\bm{x}}}d s&\leq C_1, \quad \forall \bm{x}\neq \Gamma.
\end{align*}
Given $\varepsilon>0$, we choose $0<\eta\ll 1$ s.t. $\abs{\wh{\bm{F}}(s)}<\frac{\varepsilon}{2\paren{C_0+C_1}}$ 
for all $s\in \mc{I}_\eta:=\left\{s\in \mbs \left|\abs{\bm{X}\paren{s}-\bm{X}\paren{s_0}}<\eta\right.\right\}$, and set $B_\eta=\bm{X}\paren{ \mc{I}_\eta}$.
Then, we use the technique in Lemma \ref{l:u_conti}, and for all $\abs{\bm{x}-\bm{x}_0}<\frac{\eta}{2}$,
\begin{align*}
\begin{split}
        &\abs{\wh{\mc{D}}_i[\wh{\bm{F}}]\paren{\bm{x}}-\wh{\mc{D}}_i[\wh{\bm{F}}]\paren{\bm{x}_0}}\\
    \leq&\quad \int_{\mc{I}_\eta} \paren{\abs{ K_{ij}\paren{\bm{X}\paren{s'},\bm{x}}}+\abs{ K_{ij}\paren{\bm{X}\paren{s'},\bm{X}\paren{s_0}}}}\abs{ \wh{\bm{F}}(s')}ds'\\
        &+\int_{\mbs\setminus\mc{I}_\eta} \abs{ K_{ij}\paren{\bm{X}\paren{s'},\bm{x}}-K_{ij}\paren{\bm{X}\paren{s'},\bm{x}_0}}\abs{ \wh{\bm{F}}(s')}ds'\\
    \leq&\epsilon+C\norm{ \wh{\bm{F}}}_{C^0\paren{\mbs}} \frac{1}{\eta^2}\abs{\bm{x}-\bm{x}_0}.
\end{split}
\end{align*}
Hence,
\begin{align*}
    \lim_{\bm{x}\rightarrow\bm{x_0}}\abs{\wh{\mc{D}}_i[\wh{\bm{F}}]\paren{\bm{x}}-\wh{\mc{D}}_i[\wh{\bm{F}}]\paren{\bm{x}_0}}\leq \epsilon.
\end{align*}
Letting $\varepsilon\rightarrow 0$,
\begin{align*}
    \lim_{\bm{x}\rightarrow\bm{x_0}}\abs{\wh{\mc{D}}[\wh{\bm{F}}]\paren{\bm{x}}-\wh{\mc{D}}[\wh{\bm{F}}]\paren{\bm{x}_0}}=0.
\end{align*}

\end{proof}
Now, we define limits $D_{\Omega_1}$ and $D_{\Omega_2}$ on $\Gamma$ as
\begin{align*}
    D_{\Omega_1}\paren{\bm{x}}=\lim_{t\rightarrow 0^+}\wh{\mc{D}}[\wh{\bm{F}}](\bm{x}-t \bm{n}\paren{\bm{x}}), \;  D_{\Omega_2}\paren{\bm{x}}=\lim_{t\rightarrow 0^+}\wh{\mc{D}}[\wh{\bm{F}}](\bm{x}+t \bm{n}\paren{\bm{x}}).
\end{align*}
We have the following results
\begin{theorem}
Given $\Gamma\in C^2$ and $\wh{\bm{F}}\in C\paren{\Gamma}$, 
\begin{align*}
    D_{\Omega_1}\paren{\bm{X}\paren{s}}=&-\frac{1}{2} \wh{\bm{F}}\paren{s}+\int_{\mbs} K\paren{\bm{X}\paren{s'},\bm{X}\paren{s}} \wh{\bm{F}}(s')ds'\\
    D_{\Omega_2}\paren{\bm{X}\paren{s}}=&\frac{1}{2} \wh{\bm{F}}\paren{s}+\int_{\mbs} K\paren{\bm{X}\paren{s'},\bm{X}\paren{s}} \wh{\bm{F}}(s')ds'
\end{align*}
\end{theorem}

\begin{proof}
For $D_{\Omega_1}$, since $\bm{X}\paren{s}-t \bm{n}\paren{\bm{X}\paren{s}}\in\Omega_1$ for all $0<t\ll 1$, by \eqref{e:kernel_integral},
\begin{align*}
\begin{split}
    &\wh{\mc{D}}[\wh{\bm{F}}](\bm{X}\paren{s}-t \bm{n}\paren{\bm{X}\paren{s}})
    =\int_{\mbs} K\paren{\bm{X}\paren{s'},\bm{X}\paren{s}-t \bm{n}\paren{\bm{X}\paren{s}}} \wh{\bm{F}}(s')ds'\\
    =   &\quad\wh{\bm{F}}(s)\int_{\mbs} K\paren{\bm{X}\paren{s'},\bm{X}\paren{s}-t \bm{n}\paren{\bm{X}\paren{s}}} ds'\\
        &+\int_{\mbs} K\paren{\bm{X}\paren{s'},\bm{X}\paren{s}-t \bm{n}\paren{\bm{X}\paren{s}}}\paren{ \wh{\bm{F}}(s')-\wh{\bm{F}}(s)}ds'\\
    =   &-\wh{\bm{F}}\paren{s}+\int_{\mbs} K\paren{\bm{X}\paren{s'},\bm{X}\paren{s}-t \bm{n}\paren{\bm{X}\paren{s}}}\paren{ \wh{\bm{F}}(s')-\wh{\bm{F}}(s)}ds'.
\end{split}
\end{align*}
Then, $\wh{\bm{F}}(s')-\wh{\bm{F}}(s)=0$ when $s'=s$, so by Lemma \ref{l:dlayer_conti}, the second term is continuous at $t= 0$.
Therefore, by \eqref{e:kernel_integral},
\begin{align*}
        &\lim_{t\rightarrow 0^+}\wh{\mc{D}}[\wh{\bm{F}}](\bm{x}-t \bm{n}\paren{\bm{x}})
    =-\wh{\bm{F}}\paren{s}+\int_{\mbs} K\paren{\bm{X}\paren{s'},\bm{X}\paren{s}}\paren{ \wh{\bm{F}}(s')-\wh{\bm{F}}(s)}ds'\\
    =   &-\wh{\bm{F}}\paren{s}-\wh{\bm{F}}(s)\int_{\mbs} K\paren{\bm{X}\paren{s'},\bm{X}\paren{s}}ds'\\
        &+\int_{\mbs} K\paren{\bm{X}\paren{s'},\bm{X}\paren{s}}\wh{\bm{F}}(s')ds'\\
    =   &-\frac{1}{2} \wh{\bm{F}}\paren{s}+\int_{\mbs} K\paren{\bm{X}\paren{s'},\bm{X}\paren{s}} \wh{\bm{F}}(s')ds'.
\end{align*}
Next, for $D_{\Omega_2}$, we use the same technique. 
Since $\bm{X}\paren{s}+t \bm{n}\paren{\bm{X}\paren{s}}\in\Omega_2$ for all $0<t\ll 1$, by \eqref{e:kernel_integral},
\begin{align*}
\begin{split}
    &\wh{\mc{D}}[\wh{\bm{F}}](\bm{X}\paren{s}+t \bm{n}\paren{\bm{X}\paren{s}})
    =\int_{\mbs} K\paren{\bm{X}\paren{s'},\bm{X}\paren{s}+t \bm{n}\paren{\bm{X}\paren{s}}} \wh{\bm{F}}(s')ds'\\
    =   &\int_{\mbs} K\paren{\bm{X}\paren{s'},\bm{X}\paren{s}+t \bm{n}\paren{\bm{X}\paren{s}}}\paren{ \wh{\bm{F}}(s')-\wh{\bm{F}}(s)}ds'.
\end{split}
\end{align*}
Thus,
\begin{align*}
        \lim_{t\rightarrow 0^+}\wh{\mc{D}}[\wh{\bm{F}}](\bm{x}+t \bm{n}\paren{\bm{x}})
    =   \frac{1}{2} \wh{\bm{F}}\paren{s}+\int_{\mbs} K\paren{\bm{X}\paren{s'},\bm{X}\paren{s}} \wh{\bm{F}}(s')ds'
\end{align*}
\end{proof}
Next, since $\Gamma\in C^2$, there exists an $\varepsilon_0>0$ s.t. for all $\bm{x}$ with $\dist\paren{\bm{x},\Gamma}<\varepsilon_0$, there exists a unique $s\in\mbs$ and $t\in\paren{-\varepsilon_0,\varepsilon_0}$ s.t. $\bm{x}=\bm{X}\paren{s}+t\bm{n}\paren{s}$.
Thus, we may define a tubular set $V=V\paren{\varepsilon_0, \Gamma}$ as $V:=\left\{ \bm{x}\mid  \dist\paren{\bm{x},\Gamma}<\varepsilon_0\right\}$,
within $V$, we may set functions $s\paren{\bm{x}}$ and $\bm{X}\paren{\bm{x}}=\bm{X}\paren{s\paren{\bm{x}}}$. 
Now, let us compute the limits \eqref{Siglim}:
\begin{align*}
    \bm{F}_{\Omega_1}(s)=\lim_{t\to 0+} \Sigma(\bm{X}(s)-t\bm{n}(s))\bm{n}(s), \; \bm{F}_{\Omega_2}(s)=\lim_{t\to 0+} \Sigma(\bm{X}(s)+t\bm{n}(s))\bm{n}(s).
\end{align*}
\begin{theorem}
Given $\Gamma\in C^2$ and $\wh{\bm{F}}\in C\paren{\Gamma}$, 
\begin{align*}
    \bm{F}_{\Omega_1}(s)=&\frac{1}{2} \wh{\bm{F}}\paren{s}+\int_{\mbs} K\paren{\bm{X}\paren{s},\bm{X}\paren{s'}} \wh{\bm{F}}(s')ds',\\
    \bm{F}_{\Omega_2}(s)=&-\frac{1}{2} \wh{\bm{F}}\paren{s}+\int_{\mbs} K\paren{\bm{X}\paren{s},\bm{X}\paren{s'}} \wh{\bm{F}}(s')ds'.
\end{align*}
Thus,
\begin{align*}
    \jump{\Sigma\bm{n}}=\bm{F}_{\Omega_1}-\bm{F}_{\Omega_2}=\wh{\bm{F}}
\end{align*}
\end{theorem}

\begin{proof}
First, we set $K^*\paren{\bm{y},\bm{x}}$ in $V\times V$ as
\begin{align*}
    K^*\paren{\bm{y},\bm{x}}:=\Theta_{ijk}\paren{\bm{x}-\bm{y}}n_k\paren{s\paren{\bm{x}}}.
\end{align*}
It is obvious that $K^*\paren{\bm{y},\bm{x}}=K\paren{\bm{x}, \bm{y}}$ for all $\bm{x},\bm{y}\in\Gamma$.
By \eqref{e:kernel_regularity00}, $\norm{K\paren{\bm{x},\cdot}}_{L^1\paren{\Gamma}}$ is uniformly bounded on $ \Gamma$, so we may set a function in $V$ as
\begin{align*}
    \bm{f}[\wh{\bm{F}}]\paren{\bm{x}}=\Sigma\paren{\bm{x}}\bm{n}\paren{s\paren{\bm{x}}}=\int_{\mbs} K^*\paren{\bm{X}\paren{s'},\bm{x}} \wh{\bm{F}}(s')ds'.
\end{align*}
Next, we will claim $\bm{f}\paren{\bm{x}}$ is continuous in $V$ where $\bm{f}\paren{\bm{x}}=\wh{\mc{D}}[\wh{\bm{F}}]\paren{\bm{x}}+\bm{F}[\wh{\bm{F}}]\paren{\bm{x}}$.
It is clear that $\bm{f}\paren{\bm{x}}$ is continuous in $\mbr^2\setminus\Gamma$, and by \eqref{e:kernel_regularity00} and \cite[Proposition 3.12]{F1995},$\bm{f}\paren{\bm{x}}$ is continuous on $\Gamma$.
Thus, we only have to prove $\bm{f}\paren{\bm{x}}$ is continuous at $\bm{x}_0=\bm{X}\paren{s}$ for all $s\in\mbs$.
We use the technique of the proofs of Lemma \ref{l:u_conti} and Lemma \ref{l:dlayer_conti} again.
Define $\mc{I}_\eta:=\set{s\in\mbs}{\abs{\bm{X}\paren{s}-\bm{X}\paren{s_0}}<\eta}$. Then for all $\abs{\bm{x}-\bm{x}_0}<\frac{1}{2}\eta$,
\begin{align*}
\begin{split}
    &\abs{\bm{f}\paren{\bm{x}}-\bm{f}\paren{\bm{x}_0}}\\
    \leq&\quad\int_{\mc{I}_\eta} \abs{ K_{ij}\paren{\bm{X}\paren{s'},\bm{x}}+ K^*_{ij}\paren{\bm{X}\paren{s'},\bm{x}}}\abs{ \wh{\bm{F}}(s')}ds'\\
        &+\int_{\mc{I}_\eta} \abs{ K_{ij}\paren{\bm{X}\paren{s'},\bm{x}_0}+ K^*_{ij}\paren{\bm{X}\paren{s'},\bm{x}_0}}\abs{ \wh{\bm{F}}(s')}ds'\\
        &+\int_{\mbs\setminus\mc{I}_\eta}\abs{ K_{ij}\paren{\bm{X}\paren{s'},\bm{x}}-K_{ij}\paren{\bm{X}\paren{s'},\bm{x}_0}}\abs{ \wh{\bm{F}}(s')}ds'\\
        &+\int_{\mbs\setminus\mc{I}_\eta}\abs{ K^*_{ij}\paren{\bm{X}\paren{s'},\bm{x}}-K^*_{ij}\paren{\bm{X}\paren{s'},\bm{x}_0}}\abs{ \wh{\bm{F}}(s')}ds'\\
    \leq&\quad\int_{\mc{I}_\eta} \abs{ K_{ij}\paren{\bm{X}\paren{s'},\bm{x}}+ K^*_{ij}\paren{\bm{X}\paren{s'},\bm{x}}}\abs{ \wh{\bm{F}}(s')}ds'\\
        &+\int_{\mc{I}_\eta} \abs{ K_{ij}\paren{\bm{X}\paren{s'},\bm{x}_0}+ K^*_{ij}\paren{\bm{X}\paren{s'},\bm{x}_0}}\abs{ \wh{\bm{F}}(s')}ds'\\
        &+C\norm{ \wh{\bm{F}}}_{C^0\paren{\mbs}} \frac{1}{\eta^2}\abs{\bm{x}-\bm{x}_0}.
\end{split}
\end{align*}
For the first two terms, we use some techniques in the proof of Lemma \ref{l:kernel_L1bnd}.
Consider $0<\eta<\frac{1}{2}\varepsilon_0$, where $\varepsilon_0$ is defined in the proof of Lemma \ref{l:kernel_L1bnd}. Then for all $s'\in \mc{I}_\eta$,
\begin{align*}
    \abs{\bm{n}\paren{s'}-\bm{n}\paren{s\paren{\bm{x}}}}\leq \norm{\pdd{s}{2}\bm{X}}_{C^0}\abs{s'-s\paren{\bm{x}}}\leq \frac{\norm{\pdd{s}{2}\bm{X}}_{C^0}}{\snorm{\bm{X}}}\abs{\bm{X}\paren{s'}-\bm{X}\paren{\bm{x}}}
\end{align*}
and $\abs{\bm{X}\paren{s'}-\bm{x}}>\frac{1}{2}\abs{\bm{X}\paren{s'}-\bm{X}\paren{\bm{x}}}$ by \eqref{e: proj_lower_bnd}.
\begin{align*}
\begin{split}
        &K_{ij}\paren{\bm{X}\paren{s'},\bm{x}}+ K^*_{ij}\paren{\bm{X}\paren{s'},\bm{x}}\\
    =   &-\frac{1}{\pi}\frac{\paren{X_i\paren{s}-x_i}\paren{X_j\paren{s}-x_j}\paren{\bm{X}\paren{s'}-\bm{x}}\cdot\bm{n}\paren{s'}}{\abs{\bm{X}\paren{s'}-\bm{x}}^4}\\
        &-\frac{1}{\pi}\frac{\paren{x_i-X_i\paren{s}}\paren{x_j-X_j\paren{s}}\paren{\bm{x}-\bm{X}\paren{s'}}\cdot\bm{n}\paren{s\paren{\bm{x}}}}{\abs{\bm{x}-\bm{X}\paren{s'}}^4}\\
    =   &-\frac{1}{\pi}\frac{\paren{X_i\paren{s}-x_i}\paren{X_j\paren{s}-x_j}\paren{\bm{X}\paren{s'}-\bm{x}}\cdot\paren{\bm{n}\paren{s'}-\bm{n}\paren{s\paren{\bm{x}}}}}{\abs{\bm{X}\paren{s'}-\bm{x}}^4},
\end{split}
\end{align*}
so
\begin{align*}
\begin{split}
        &\int_{\mc{I}_\eta} \abs{ K_{ij}\paren{\bm{X}\paren{s'},\bm{x}}+ K^*_{ij}\paren{\bm{X}\paren{s'},\bm{x}}}\abs{ \wh{\bm{F}}(s')}ds'\\
    \leq&\norm{ \wh{\bm{F}}}_{C^0\paren{\mbs}}\int_{\mc{I}_\eta} \abs{ K_{ij}\paren{\bm{X}\paren{s'},\bm{x}}+ K^*_{ij}\paren{\bm{X}\paren{s'},\bm{x}}}ds'\\
    \leq&\frac{2\norm{ \wh{\bm{F}}}_{C^0\paren{\mbs}}}{\pi}\int_{\mc{I}_\eta}ds'
    \leq\frac{4\norm{ \wh{\bm{F}}}_{C^0\paren{\mbs}}}{C_0\pi}\eta
\end{split}
\end{align*}
where $C_0$ is defined in the proof of Lemma \ref{l:kernel_L1bnd}.
Therefore,
\begin{align*}
    \lim_{\bm{x}\rightarrow\bm{x}_0}\abs{\bm{f}\paren{\bm{x}}-\bm{f}\paren{\bm{x}_0}}\leq \frac{4\norm{ \wh{\bm{F}}}_{C^0\paren{\mbs}}}{C_0\pi}\eta.
\end{align*}
Letting $\eta\rightarrow 0$, we see that $\bm{f}\paren{\bm{x}}$ continuous at $\bm{x}_0$.
Finally, we have
\begin{align*}
    \bm{F}_{\Omega_1}(s)
    =&\lim_{t\to 0+} \wh{\mc{D}}[\wh{\bm{F}}](\bm{X}(s)-t\bm{n}(s))\\
    =&\lim_{t\to 0+} \bm{f}(\bm{X}(s)-t\bm{n}(s))-\lim_{t\to 0+} \wh{\mc{D}}[\wh{\bm{F}}](\bm{X}(s)-t\bm{n}(s))\\
    =&\bm{f}(\bm{X}(s))-D_{\Omega_1}\paren{\bm{X}\paren{s}}\\
    =&\frac{1}{2} \wh{\bm{F}}\paren{s}+\int_{\mbs} K^*\paren{\bm{X}\paren{s'},\bm{X}(s)} \wh{\bm{F}}(s')ds'\\
    =&\frac{1}{2} \wh{\bm{F}}\paren{s}+\int_{\mbs} K\paren{\bm{X}(s),\bm{X}\paren{s'}} \wh{\bm{F}}(s')ds'.
\end{align*}
Similarly, one can obtain
\begin{align*}
    \bm{F}_{\Omega_2}(s)=-\frac{1}{2} \wh{\bm{F}}\paren{s}+\int_{\mbs} K\paren{\bm{X}(s),\bm{X}\paren{s'}} \wh{\bm{F}}(s')ds'.
\end{align*}
\end{proof}
Now, we will prove the necessary and sufficient condition of $\bm{u}\paren{\bm{x}}$ vanishes at infinity. 
\begin{lemma}
\begin{align*}
    \abs{\bm{u}\paren{\bm{x}}}\rightarrow 0 \mbox{ as } \abs{\bm{x}}\rightarrow \infty \Longleftrightarrow \int_\mbs \wh{\bm{F}}(s)ds=0
\end{align*}
\end{lemma}

\begin{proof}
Since $\Omega_1$ is bounded, there exists a constant $R_0>0$ s.t. $\abs{\bm{X}\paren{s}}<R_0$ for all $s\in \mbs$.
Then, 
\begin{align*}
\begin{split}
     G\paren{\bm{x}-\bm{X}\paren{s}}-G\paren{\bm{x}} =  \int_0^1 \nabla G\paren{\bm{x}-t\bm{X}\paren{s}} \cdot \bm{X}\paren{s}  dt
\end{split}    
\end{align*}
so there exists a constant $C>0$ s.t. for all $\abs{\bm{x}}>2R_0$,
\begin{align*}
\begin{split}
    \abs{G\paren{\bm{x}}-G\paren{\bm{x}-\bm{X}\paren{s}}}\leq C \sup_{0\leq t\leq 1} \frac{\abs{\bm{X}\paren{s}}}{\abs{\bm{x}-t\bm{X}\paren{s}}}\leq C  \frac{R_0}{\abs{\bm{x}}-R_0}.
\end{split}
\end{align*}
Next, for all $\abs{\bm{x}}>2R_0$
\begin{align*}
    \bm{u}\paren{\bm{x}}
    =&\int_\mbs G\paren{\bm{x}-\bm{X}\paren{s'}}\wh{\bm{F}}(s')ds'\\
    =&\int_\mbs \paren{ G\paren{\bm{x}-\bm{X}\paren{s'}}-G\paren{\bm{x}}}\wh{\bm{F}}(s')ds'+G\paren{\bm{x}}\int_\mbs \wh{\bm{F}}(s')ds'.
\end{align*}
Since
\begin{align*}
    \lim_{\abs{\bm{x}}\rightarrow \infty}\abs{\int_\mbs \paren{ G\paren{\bm{x}-\bm{X}\paren{s'}}-G\paren{\bm{x}}}\wh{\bm{F}}(s')ds'}
    \leq\lim_{\abs{\bm{x}}\rightarrow \infty}C  \frac{R_0}{\abs{\bm{x}}-R_0}
    =0
\end{align*}
and
\begin{align*}
    \lim_{\abs{\bm{x}}\rightarrow \infty}G_{ii}\paren{\bm{x}}=\infty,\quad
    \lim_{\abs{\bm{x}}\rightarrow \infty}\abs{ G_{i\paren{2-i}}\paren{\bm{x}}}\leq 1,
\end{align*}
we obtain
\begin{align*}
    \lim_{\abs{\bm{x}}\rightarrow \infty}\abs{\bm{u}\paren{\bm{x}}}\rightarrow 0  \Longleftrightarrow \int_\mbs \wh{\bm{F}}(s)ds=0
\end{align*}
\end{proof}

\section{Some Calculus for the Operator $\mc{L}$ near the unit circle}\label{unitcircle_appendix}
\subsection{Computation for ${\mc{S}_i,\tau_i} $ }\quad\\
Set 
\begin{align*}
    \bm{X}_c=
    \begin{bmatrix}
    \cos \theta\\
    \sin \theta\\
    \end{bmatrix}
    \mbox{ and } \bm{X}_\varepsilon=\bm{X}_c+\varepsilon\bm{Y}.
\end{align*}
Then,

\begin{align*}
    4\pi G_L\paren{\Delta\bm{X}_\varepsilon}=-\log\abs{\bm{X}_\varepsilon}=G_{L0}+G_{L1}\varepsilon+G_{L2}\varepsilon^2+O\paren{\varepsilon^3},
\end{align*}
where
\begin{align}
    G_{L0}&=-\log\abs{\Delta\bm{X}_c}\label{t:GL00}\\
    G_{L1}&=-\frac{\Delta\bm{X}_c\cdot\Delta\bm{Y}}{\abs{\Delta\bm{X}_c}^2}\label{t:GL01}\\
    G_{L2}&=\frac{\paren{\Delta\bm{X}_c\cdot\Delta\bm{Y}}^2}{\abs{\Delta\bm{X}_c}^4}-\frac{1}{2}\frac{\abs{\Delta\bm{Y}}^2}{\abs{\Delta\bm{X}_c}^2},\label{t:GL02}
\end{align}
and
\begin{align*}
    4\pi G_T\paren{\Delta\bm{X}_\varepsilon}=\frac{\Delta\bm{X}_\varepsilon\otimes\Delta\bm{X}_\varepsilon}{\abs{\Delta\bm{X}_\varepsilon}^2}=G_{T0}+G_{T1}\varepsilon+G_{T2}\varepsilon^2+O\paren{\varepsilon^3},
\end{align*}
where
\begin{align}
    G_{T0}&=\frac{\Delta\bm{X}_c\otimes\Delta\bm{X}_c}{\abs{\Delta\bm{X}_c}^2}\label{t:GT00}\\
    G_{T1}&=\frac{\Delta\bm{Y}\otimes\Delta\bm{X}_c+\Delta\bm{X}_c\otimes\Delta\bm{Y}}{\abs{\Delta\bm{X}_c}^2}-2\frac{\paren{\Delta\bm{X}_c\cdot\Delta\bm{Y}}\paren{\Delta\bm{X}_c\otimes\Delta\bm{X}_c}}{\abs{\Delta\bm{X}_c}^4}\label{t:GT01}
    \end{align}
\begin{align}
\begin{split}
        G_{T2}
    =&\quad\frac{\Delta\bm{Y}\otimes\Delta\bm{Y}}{\abs{\Delta\bm{X}_c}^2}
    -2\frac{\paren{\Delta\bm{X}_c\cdot\Delta\bm{Y}}\paren{\Delta\bm{Y}\otimes\Delta\bm{X}_c+\Delta\bm{X}_c\otimes\Delta\bm{Y}}}{\abs{\Delta\bm{X}_c}^4}\\
    &+4\frac{\paren{\Delta\bm{X}_c\cdot\Delta\bm{Y}}^2\paren{\Delta\bm{X}_c\otimes\Delta\bm{X}_c}}{\abs{\Delta\bm{X}_c}^6}
    -\frac{\abs{\Delta\bm{Y}}^2\paren{\Delta\bm{X}_c\otimes\Delta\bm{X}_c}}{\abs{\Delta\bm{X}_c}^4}\label{t:GT02}
\end{split}
\end{align}
Next,
\begin{align*}
    \bm{\tau}_\varepsilon=\frac{\pd{\theta}\bm{X}_\varepsilon}{\abs{\pd{\theta}\bm{X}_\varepsilon}}=\bm{\tau}_0+\bm{\tau}_1\varepsilon+\bm{\tau}_2\varepsilon^2+O\paren{\varepsilon^3},
\end{align*}
where
\begin{align}
    \bm{\tau}_0&=\pd{\theta}\bm{X}_c, \quad \pd{\theta}\bm{\tau}_0=-\bm{X}_c\label{e:tau001}\\
    \bm{\tau}_1&=\pd{\theta}\bm{Y}-\paren{\pd{\theta}\bm{X}_c\cdot\pd{\theta}\bm{Y}}\pd{\theta}\bm{X}_c=\paren{\bm{X}_c\cdot\pd{\theta}\bm{Y}}\bm{X}_c\label{e:tau011}
\end{align}

Therefore, for $G_{T1}, G_{T2}$, we have some results for some computations in Section \ref{s:lambda2}
\begin{align}
    \bm{X}_c\cdot G_{T1}  \bm{X}_c'
    =&0\label{e:GT0100} \\    
    \pd{\theta}\bm{X}_c\cdot G_{T1}  \bm{X}_c'
    =&-\frac{1}{2}\pd{\theta}\bm{X}_c\cdot\Delta\bm{Y}-\frac{\paren{\pd{\theta}\bm{X}_c\cdot\bm{X}_c'}}{\abs{\Delta\bm{X}_c}^2}\Delta\bm{Y}\cdot\paren{\bm{X}_c'+\Delta\bm{X}_c}\label{e:GT01d00}\\
    \bm{X}_c\cdot G_{T2}  \bm{X}_c'
    =&\frac{\paren{\bm{X}_c\cdot\Delta\bm{Y}}\paren{\Delta\bm{Y}\cdot \bm{X}_c'}}{\abs{\Delta\bm{X}_c}^2}+\frac{1}{4}\abs{\Delta\bm{Y}}^2\label{e:GT0200}
\end{align}
% \begin{align}
% \begin{split}
%     \bm{X}_c\cdot G_{T1}  \bm{X}_c'
%     =0\label{e:GT0100}
% \end{split}
% \end{align}

% \begin{align}
% \begin{split}
%     \bm{X}_c\cdot G_{T2}  \bm{X}_c'
%     =\frac{\paren{\bm{X}_c\cdot\Delta\bm{Y}}\paren{\Delta\bm{Y}\cdot \bm{X}_c'}}{\abs{\Delta\bm{X}_c}^2}+\frac{1}{4}\abs{\Delta\bm{Y}}^2\label{e:GT0200}
% \end{split}
% \end{align}

% \begin{align}
% \begin{split}
%     \pd{\theta}\bm{X}_c\cdot G_{T1}  \bm{X}_c'
%     =-\frac{1}{2}\pd{\theta}\bm{X}_c\cdot\Delta\bm{Y}-\frac{\paren{\pd{\theta}\bm{X}_c\cdot\bm{X}_c'}}{\abs{\Delta\bm{X}_c}^2}\Delta\bm{Y}\cdot\paren{\bm{X}_c'+\Delta\bm{X}_c}
% \end{split}\label{e:GT01d00}
% \end{align}
Moreover,
\begin{align}
    \pd{\theta}\bm{\tau}_1\cdot G_{L1}\pd{\theta'}\bm{\tau}_0'
    =&\frac{\paren{\pd{\theta}\bm{\tau}_1\cdot\bm{X}_c'}\Delta\bm{X}_c}{\abs{\Delta\bm{X}_c}^2}\cdot\Delta \bm{Y}\label{e:GL0110}\\
    \pd{\theta}\bm{\tau}_1\cdot G_{T1}\pd{\theta'}\bm{\tau}_0'
    =&-\frac{\paren{\pd{\theta}\bm{\tau}_1\cdot\Delta\bm{X}_c}\bm{X}_c}{\abs{\Delta\bm{X}_c}^2}\cdot\Delta \bm{Y}+\frac{1}{2}\pd{\theta}\bm{\tau}_1\cdot\Delta \bm{Y}\label{e:GT0110}
\end{align}

Furthermore, set $\bm{Y}=g\bm{X}_c$, then we have
\begin{align*}
    \Delta \bm{Y}   &=g\Delta\bm{X}_c+\Delta g\bm{X}_c'
\end{align*}
We obtain some results of $\Delta \bm{Y}$ for computations in section \ref{s:lambda2}.
\begin{lemma}[The computations for $\Delta \bm{Y}$]\label{c:Y_01}
\quad\\
(1)
\begin{align*}
    \frac{\Delta \bm{X}_c\cdot\Delta \bm{Y}}{\abs{\Delta\bm{X}_c}^2}
    =&\frac{1}{2}\paren{g+g'}
\end{align*}
(2)
\begin{align*}
     \bm{X}_c\cdot\Delta \bm{Y}
     =&\frac{1}{2}g\abs{\Delta\bm{X}_c}^2+\Delta g \bm{X}_c\cdot \bm{X}_c'\\
     \bm{X}_c'\cdot\Delta \bm{Y}
     =&-\frac{1}{2}g\abs{\Delta\bm{X}_c}^2+\Delta g\\
     \pd{\theta}\bm{X}_c\cdot\Delta \bm{Y}
     =&-g \pd{\theta}\bm{X}_c\cdot\bm{X}_c'+\Delta g \pd{\theta}\bm{X}_c\cdot\bm{X}_c'
\end{align*}
(3)
\begin{align*}
    \abs{\Delta\bm{Y}}^2
    =&g g'\abs{\Delta\bm{X}_c}^2+\abs{\Delta g}^2\\
    \frac{\abs{\Delta\bm{Y}}^2}{\abs{\Delta\bm{X}_c}^2}
    =&g g'+\frac{\abs{\Delta g}^2}{\abs{\Delta\bm{X}_c}^2}
\end{align*}
\end{lemma}
\begin{proof}
\quad\\
(1)
\begin{align*}
    \Delta \bm{X}_c\cdot\Delta \bm{Y}
    =g\abs{\Delta\bm{X}_c}^2-\Delta g\paren{1-\bm{X}_c'\cdot\bm{X}_c}
    =g\abs{\Delta\bm{X}_c}^2-\frac{1}{2}\Delta g\abs{\Delta\bm{X}_c}^2,
\end{align*}
so
\begin{align*}
    \frac{\Delta \bm{X}_c\cdot\Delta \bm{Y}}{\abs{\Delta\bm{X}_c}^2}
    =g-\frac{1}{2}\Delta g
    =\frac{1}{2}\paren{g+g'}
\end{align*}
(2)
\begin{align*}
     \bm{X}_c\cdot\Delta \bm{Y}
     =&\bm{X}_c\cdot \left(g\Delta\bm{X}_c+\Delta g\bm{X}_c'\right)
     =g\paren{1-\bm{X}_c\cdot\bm{X}_c'}+\Delta g \bm{X}_c\cdot \bm{X}_c'\\
     =&\frac{1}{2}g\abs{\Delta\bm{X}_c}^2+\Delta g \bm{X}_c\cdot \bm{X}_c'\\
     \bm{X}_c'\cdot\Delta \bm{Y}
     =&\bm{X}_c'\cdot \left(g\Delta\bm{X}_c+\Delta g\bm{X}_c'\right)
     =g\paren{\bm{X}_c\cdot\bm{X}_c'-1}+\Delta g\\ 
     =&-\frac{1}{2}g\abs{\Delta\bm{X}_c}^2+\Delta g\\
     \pd{\theta}\bm{X}_c\cdot\Delta \bm{Y}
     =&\pd{\theta}\bm{X}_c\cdot\paren{g\Delta\bm{X}_c+\Delta g\bm{X}_c'}
     =-g \pd{\theta}\bm{X}_c\cdot\bm{X}_c'+\Delta g \pd{\theta}\bm{X}_c\cdot\bm{X}_c'
\end{align*}
(3)
\begin{align*}
    \abs{\Delta\bm{Y}}^2
    =&\left(g\Delta\bm{X}_c+\Delta g\bm{X}_c'\right)^2
    =g^2\abs{\Delta\bm{X}_c}^2-g\Delta g\abs{\Delta\bm{X}_c}^2+\abs{\Delta g}^2\\
    =&g g'\abs{\Delta\bm{X}_c}^2+\abs{\Delta g}^2
\end{align*}
\end{proof}
Moreover, for $g$, we have the following equation of Hilbert transforms.
\begin{lemma}[Toland]%\todo{write the reference}
\label{t:Toland}
If $g\in C^1\paren{\mbs}$, then
\begin{align*}
    g\mc{H}\pd{\theta}g-\frac{1}{2}\mc{H}\pd{\theta}g^2=\frac{1}{8\pi}\int_\mbs \frac{\abs{\Delta g}^2}{\sin^2\paren{\frac{\theta-\theta'}{2}}}d\theta'.
\end{align*}
\end{lemma}
\begin{proof}
\begin{align*}
     &g\mc{H}\pd{\theta}g-\frac{1}{2}\mc{H}\pd{\theta}g^2
     =g\mc{H}\pd{\theta}g-\mc{H}g\pd{\theta}g\\
     =&g\frac{1}{2\pi}\int_\mbs \cot \paren{\frac{\theta-\theta'}{2}}\pd{\theta'}g' d\theta'-\frac{1}{2\pi}\int_\mbs \cot \paren{\frac{\theta-\theta'}{2}}g'\pd{\theta'}g' d\theta'\\
     =&\frac{1}{2\pi}\int_\mbs \cot \paren{\frac{\theta-\theta'}{2}}\Delta g\pd{\theta'}g' d\theta'
     =-\frac{1}{4\pi}\int_\mbs \cot \paren{\frac{\theta-\theta'}{2}}\pd{\theta'}\abs{\Delta g}^2 d\theta'\\
     =&\frac{1}{8\pi}\int_\mbs \frac{\abs{\Delta g}^2}{\sin^2\paren{\frac{\theta-\theta'}{2}}}d\theta'.
\end{align*}
\end{proof}

\subsection{Some computations of Hilbert transform for Fourier series}
In this section,we will compute some Hilbert transforms for Section \ref{s:lambda2}.
First, for Proposition \ref{t: computation_circle}, we need to compute Hilbert transforms for trigonometric functions. 
\begin{lemma}\label{c:hilbert_01}
For $n\geq 2$,
\begin{align*}
    \begin{array}{ll}
        \mc{H} \left[\cos \paren{n\theta}\bm{X}_c\right]=\sin \paren{n\theta}\bm{X}_c, & \mc{H} \left[\cos \paren{n\theta}\pd{\theta}\bm{X}_c\right]=\sin \paren{n\theta}\pd{\theta}\bm{X}_c, \\
        \mc{H} \left[\sin \paren{n\theta}\bm{X}_c\right]=-\cos \paren{n\theta}\bm{X}_c,& \mc{H} \left[\sin \paren{n\theta}\pd{\theta}\bm{X}_c\right]=-\cos \paren{n\theta}\pd{\theta}\bm{X}_c.
    \end{array}
\end{align*}
Moreover,
\begin{align*}
    &\mc{H} \left[\cos \theta\bm{X}_c\right]=\frac{1}{2}
    \begin{bmatrix}
    \sin \paren{2\theta}\\
    -\cos \paren{2\theta}\\
    \end{bmatrix},
    &\mc{H} \left[\cos \theta\pd{\theta}\bm{X}_c\right]=\frac{1}{2}
    \begin{bmatrix}
    \cos \paren{2\theta}\\
    \sin \paren{2\theta}\\
    \end{bmatrix},\\
    &\mc{H} \left[\sin \theta\bm{X}_c\right]=-\frac{1}{2}
    \begin{bmatrix}
    \cos \paren{2\theta}\\
    \sin \paren{2\theta}\\
    \end{bmatrix},
    &\mc{H} \left[\sin \theta\pd{\theta}\bm{X}_c\right]=\frac{1}{2}
    \begin{bmatrix}
    \sin \paren{2\theta}\\
    \cos \paren{2\theta}\\
    \end{bmatrix}.\\
\end{align*}
\begin{align*}
    \mc{H} \left[\bm{X}_c\right]=-\pd{\theta}\bm{X}_c,\quad
    \mc{H} \left[\pd{\theta}\bm{X}_c\right]=\bm{X}_c.
\end{align*}
\end{lemma}

\begin{lemma}\label{c:hilbert_02}
For $n\geq 2$,
\begin{align*}
    \mc{H} \left[\pd{\theta}\paren{\cos \paren{n\theta}\bm{X}_c}\right]&=n\cos \paren{n\theta}\bm{X}_c+\sin \paren{n\theta}\pd{\theta}\bm{X}_c,\\
    \mc{H} \left[\pd{\theta}\paren{\sin \paren{n\theta}\bm{X}_c}\right]&=n\sin \paren{n\theta}\bm{X}_c-\cos \paren{n\theta}\pd{\theta}\bm{X}_c,\\
    \mc{H} \left[\pd{\theta}\paren{\cos \paren{n\theta}\pd{\theta}\bm{X}_c}\right]&=-\sin \paren{n\theta}\bm{X}_c+n\cos \paren{n\theta}\pd{\theta}\bm{X}_c,\\
    \mc{H} \left[\pd{\theta}\paren{\sin \paren{n\theta}\pd{\theta}\bm{X}_c}\right]&=\cos \paren{n\theta}\bm{X}_c+n\sin \paren{n\theta}\pd{\theta}\bm{X}_c.
\end{align*}
Moreover,
\begin{align*}
    \begin{array}{ll}
        \mc{H} \left[\pd{\theta}\paren{\cos \theta\bm{X}_c}\right]=
        \begin{bmatrix}
        \cos \paren{2\theta}\\
        \sin \paren{2\theta}\\
        \end{bmatrix}, &  
        \mc{H} \left[\pd{\theta}\paren{\cos \theta\pd{\theta}\bm{X}_c}\right]=
        \begin{bmatrix}
        -\sin \paren{2\theta}\\
        \cos \paren{2\theta}\\
        \end{bmatrix},\\
        &\\
         \mc{H} \left[\pd{\theta}\paren{\sin \theta\bm{X}_c}\right]=
        \begin{bmatrix}
        \sin \paren{2\theta}\\
        -\cos \paren{2\theta}\\
        \end{bmatrix},&
        \mc{H} \left[\pd{\theta}\paren{\sin \theta\pd{\theta}\bm{X}_c}\right]=
        \begin{bmatrix}
        \cos \paren{2\theta}\\
        \sin \paren{2\theta}\\
        \end{bmatrix}.
    \end{array}
\end{align*}
\end{lemma}
Next, when we compute $\lambda_2$ in Section \ref{s:lambda2}, we expand $\bm{Y}$ as a Fourier series
\begin{align*}
    \bm{Y}\paren{\theta}=g\bm{X}_c\paren{\theta}, \mbox{ where } g=g_0+\sum_{n\geq 1} g_{n1}\cos\paren{n\theta}+g_{n2}\sin\paren{n\theta}
\end{align*}
Then, we obtain some results for $\bm{\tau_1}$

\begin{lemma}[The computation for $\bm{\tau}_1$]\label{c:tau_01}\quad\\
(1)
\begin{align*}
    \bm{\tau}_1 =&\pd{\theta}g\bm{X}_c=\sum_{n\geq 1} n\left[ g_{n2}\cos \paren{n\theta}-g_{n1}\sin \paren{n\theta}\right]\bm{X}_c\\
    \pd{\theta}\bm{\tau}_1
    =&-\sum_{n\geq 1} n^2\left[ g_{n1}\cos \paren{n\theta}+g_{n2}\sin \paren{n\theta}\right]\bm{X}_c\\
     &+\sum_{n\geq 1} n\left[ g_{n2}\cos \paren{n\theta}-g_{n1}\sin \paren{n\theta}\right]\pd{\theta}\bm{X}_c
\end{align*}
(2)
\begin{align*}
    \mc{H}\pd{\theta}\bm{\tau}_1
    =&\quad\sum_{n\geq 2} n^2\left[g_{n2}\cos \paren{n\theta}- g_{n1}\sin \paren{n\theta}\right]\bm{X}_c\\
     &+\sum_{n\geq 2} n\left[ g_{n1}\cos \paren{n\theta}+ g_{n2}\sin \paren{n\theta}\right]\pd{\theta}\bm{X}_c\\
     &+ g_{12}
    \begin{bmatrix}
    \cos \paren{2\theta}\\
    \sin \paren{2\theta}\\
    \end{bmatrix}
    + g_{11}
    \begin{bmatrix}
    -\sin \paren{2\theta}\\
    \cos \paren{2\theta}\\
    \end{bmatrix}
\end{align*}
(3)
\begin{align*}
    \pd{\theta}\bm{\tau}_1\cdot\Delta\bm{X}_c&=\frac{1}{2}\pdd{\theta}{2} g\abs{\Delta\bm{X}_c}^2-\pd{\theta}g\sin(\theta-\theta')\\
    \pd{\theta}\bm{\tau}_1\cdot\bm{X}_c'&=\pdd{\theta}{2} g\bm{X}_c\cdot\bm{X}_c'+\pd{\theta}g\pd{\theta}\bm{X}_c\cdot\bm{X}_c'\\
\end{align*}

\end{lemma}

\bibliographystyle{plain}
\bibliography{reference}

\newpage

\end{document}